\documentclass{article}
\usepackage[a4paper,top=2cm,bottom=2cm,left=2cm,right=2cm]{geometry}

\usepackage{amsmath}
\usepackage{amssymb}
\usepackage{amsthm}
\usepackage{mathrsfs}
\usepackage{tikz-cd}
\usepackage{stmaryrd}
\usepackage{hyperref}
\usepackage{authblk}
\usepackage{t-angles}
\usepackage[all]{xy}
\usepackage[toc,page]{appendix}
\usepackage{mathtools}
	
\usepackage[T1]{fontenc}
\usepackage[utf8]{inputenc}
\usepackage{lmodern}
	
\newtheorem{proposition}{Proposition}[section]
\newtheorem{definition}[proposition]{Definition}
\newtheorem{lemma}[proposition]{Lemma}
\newtheorem{theorem}[proposition]{Theorem}
\newtheorem{remark}[proposition]{Remark}
\newtheorem{example}[proposition]{Example}
\newtheorem{corollary}[proposition]{Corollary}

\title{{\bf On the \DJ ur\dj evi\'c approach to quantum principal bundles}}
\author{{\large Antonio Del Donno}\footnote{antonio.deldonno2@unibo.it}~}
\author{{\large Emanuele Latini}\footnote{emanuele.latini@unibo.it}~}
\author{{\large Thomas Weber}\footnote{thomasmartin.weber@unibo.it}$^{~,\mathparagraph}$}
\affil{\centerline{\sl
{ Alma Mater Studiorum - Università di Bologna}}
\centerline{\sl  Via Zamboni 33, 40126 Bologna, Italy}

\bigskip

\centerline{$^\mathparagraph${\sl Istituto Nazionale di Fisica Nucleare}}

\centerline{\sl Sezione di Bologna,  I-40126 Bologna, Italy}
}

\date{\today}
	
\begin{document}
		
\maketitle
		
\begin{abstract}
We revisit and extend the \DJ ur\dj evi\'c theory of complete calculi on quantum principal bundles. In this setting one naturally obtains a graded Hopf--Galois extension of the higher order calculus and an intrinsic decomposition of degree 1-forms into horizontal and vertical forms. This proposal is appealing, since it is consistently equipped with a canonical braiding and exactness of the Atiyah sequence is guaranteed. Moreover, we provide  examples of complete calculi, including the noncommutative 2-torus, the quantum Hopf fibration and differential calculi on crossed product algebras.
\end{abstract}
		
\tableofcontents
		
\section{Introduction}

Principal bundles and principal connections are very interesting and deep mathematical objects. They encode many relevant features of the underlying geometry: for example, they are used in Cartan geometry to describe locally homogeneous spaces, as well as certain structures, e.g. Riemannian structures, emerging as reduction of the first order frame bundle. In mathematical physics, principal bundles and their associated bundles are used to describe gauge bosons and their interaction with matter fields. For all these reasons it is then an interesting and fascinating task to construct their quantum deformations. In doing so, an algebraic approach is needed;  the structure group of the principal bundle is replaced by a structure Hopf algebra $H,$ while the total space is replaced by an $H$-comodule algebra $A$. Following \cite{Schn90}, principality of the bundle is captured by the notion of a faithfully flat Hopf--Galois extension $B:=A^{coH}\subset A$, while local triviality is encompassed by cleft extensions. 

In order to introduce connections in this algebraic context extra ingredients are required. While in classical differential geometry there is a functorial construction of the differential calculus on a differentiable manifold, in noncommutative differential geometry \cite{Connes,Wor89} there does not exist a unique differential calculus on an algebra; in spite of that, it is well known that every first order differential calculus can be obtained as the quotient of a universal one. On Hopf algebras and $H$-comodule algebras one can require natural (co)invariance conditions leading thus to $H$-covariant calculi. In this perspective emerges the necessity of constructing a consistent differential geometry of a quantum principal bundle that requires also certain natural compatibility conditions among the calculi on $A$ and $H$. In a nutshell, one requires the coaction $\Delta_A\colon A\to A\otimes H$ to extend to differential forms as a morphism of differential graded algebras. This leads then to the notion of complete calculus (first introduced by \DJ ur\dj evi\'c in \cite{DurI,DurII}), which is the main objective of the present work; this manuscript in particular aims to give a solid contribution to the differential geometry of quantum principal bundles on an affine variety. We remark that the case of bundles on projective varieties is intrinsically more difficult since it requires patching affine opens; a solution to this problem can be found in \cite{AFL,AFLW,FLP24} where the authors proposed an approach to noncommutative differential geometry which employs the sheaf-theoretic language. Note that there is also a recent interest from the noncommutative geometry community in the \DJ ur\dj evi\'c approach, since it permits a suitable notion of gauge transformation that naturally generalises the classical case \cite{ALP1,ALP2,Cac,Cac2}.  

In this paper we rebuild part of the quantum principal bundle framework \cite{DurI,DurII,DurDS,DurGauge,DurHG} of \DJ ur\dj evi\'c, which he developed in the 90's and which is partially unpublished to this day (see also the book \cite{Son15} for more recent developments). Doing so, our contribution is threefold. First of all, we translate the original papers into a mathematical language which seems to be more broadly accepted, (roughly) following the conventions of \cite{BegMaj,Haj96,SchDC}. This is crucial to make the approach more accessible. Thereby, we also establish the \DJ ur\dj evi\'c braiding on the total space algebra $A$, which makes the latter braided-commutative. This braiding is obtained by pulling back the algebra structure of $A\otimes H$ to $A\otimes_BA$ via the canonical Hopf--Galois map and it extends to higher order differential forms. Secondly, we relate the results to the established literature on quantum principal bundles, most importantly to \cite{BrzMaj}. In this way we clarify that the \DJ ur\dj evi\'c approach is not in opposition to the one of \cite{BrzMaj}, but in fact extends the latter from differential $1$-forms to higher order forms. Lastly, we provide new examples of complete calculi on quantum principal bundles, including calculi on the noncommutative $2$-torus, crossed product algebras, as well as on the fundamental example of $\mathrm{SL}_q(2)$. The corresponding \DJ ur\dj evi\'c braidings and base, vertical and horizontal forms are discussed in detail. The unification with the Brzezi\'nski--Majid approach together with the supply of relevant examples is, in our opinion, great evidence for the relevance of \DJ ur\dj evi\'c's approach in noncommutative differential geometry. In addition, we make an effort to present and prove the results of \DJ ur\dj evi\'c, particularly the ones contained in the work \cite{DurII}, in new concise and conceptual ways, thereby also clarifying and resolving some of the material. Note that, in contrast to \cite{DurI,DurII} we do not assume the existence of a compact quantum space and a compact quantum group acting on it, but rather work in the algebraic setting, which concerns principal comodule algebras and Hopf algebras. Furthermore, for the sake of brevity we do not consider $*$-involutions, even though this could lead to an interesting continuation of this paper. More details about the content of this paper can be found in the master Thesis \cite{AntonioThesis}.

In chapter \ref{C2} we review the main definitions and properties of differential calculi on Hopf algebras and $H$-comodule algebras. 
In chapter \ref{C3} we discuss the notion of Hopf--Galois extension and quantum principal bundle. We introduce the  complete calculus on the total space algebra, as well as the corresponding base, vertical and horizontal forms and show that the Atiyah sequence is always exact in this framework. We then compare \DJ ur\dj evi\'c's theory to the approach of \cite{BrzMaj}.
Given a complete calculus, one can then develop the formalism of connections and connection 1-forms, which is the main topic of chapter \ref{sec:connections}.
Chapter \ref{C5} treats the previously mentioned braiding, which emerges from the \DJ ur\dj evi\'c approach. We prove that the higher order complete calculus yields a graded Hopf--Galois extension. This shows that complete calculi preserve and in fact prolong the Hopf--Galois property of the underlying algebras.
In chapter \ref{sec:ex} we provide explicit realizations of complete differential calculi for many relevant examples, like the noncommutative algebraic 2-torus, the quantum Hopf fibration and cleft Hopf--Galois extensions.

\subsection*{Conventions and notation}

We fix a field $\Bbbk$. Throughout this paper all spaces are $\Bbbk$-vector spaces and all maps are $\Bbbk$-linear if not stated otherwise. All algebras are assumed to be associative and unital and all coalgebras are assumed to be coassociative and counital if not stated otherwise. The tensor product of $\Bbbk$-vector spaces is denoted by $\otimes$, while the balanced tensor product over an algebra $B$ is denoted by $\otimes_B$. Given a coalgebra $C$ with comultiplication $\Delta\colon C\to C\otimes C$ and counit $\varepsilon\colon C\to\Bbbk$ we employ Sweedler's notation for the coproduct $\Delta(c)=:c_1\otimes c_2$ of an element $c\in C$. In particular, coassociativity and counitality of $C$ read for $c\in C$
$$
(c_1)_1\otimes(c_1)_2\otimes c_2=c_1\otimes(c_2)_1\otimes(c_2)_2=:c_1\otimes c_2\otimes c_3
$$
and $\varepsilon(c_1)c_2=c=c_1\varepsilon(c_2)$, respectively.
For a right $C$-comodule $M$ of a coalgebra $C$ we denote the right $C$-coaction by $\Delta_M\colon M\to M\otimes C$ and employ the Sweedler-like notation $\Delta_M(m)=:m_0\otimes m_1$,
$$
(m_0)_0\otimes(m_0)_1\otimes m_1=m_0\otimes(m_1)_1\otimes(m_1)_2=:m_0\otimes m_1\otimes m_2
$$
and similar for repeated coactions of $m\in M$. Similarly, the left $C$-coaction of a coalgebra $C$ on a left $C$-comodule $M$ is denoted by ${}_M\Delta\colon M\to C\otimes M$ with short notation ${}_M\Delta(m)=:m_{-1}\otimes m_0$, etc. and in case of a $C$-bicomodule $M$ we write $(m_0)_{-1}\otimes(m_0)_0\otimes m_1=m_{-1}\otimes(m_0)_0\otimes(m_0)_1=:m_{-1}\otimes m_0\otimes m_1$ for all $m\in M$.
Given a Hopf algebra $H$ we always assume that its antipode $S\colon H\to H$ is invertible and denote the inverse by $S^{-1}\colon H\to H$. As standard references for Hopf algebras and comodules we refer to the textbooks \cite{Kas95,Mon94}.

In this paper all graded vector spaces are $\mathbb{N}_0$-graded $\Bbbk$-vector spaces $V=\bigoplus\nolimits_{k\in\mathbb{N}_0}V^k$. Given two graded vector spaces $V=\bigoplus\nolimits_{k\in\mathbb{N}_0}V^k$ and $W=\bigoplus\nolimits_{k\in\mathbb{N}_0}W^k$ their tensor product is the graded vector space $V\otimes W$ with degree $k$-component $\bigoplus\nolimits_{r+s=k}V^r\otimes V^s$ and a map $\Phi\colon V\to W$ is of degree $\ell\in\mathbb{N}_0$ if $\Phi(V^k)\subseteq W^{k+\ell}$ for all $k\in\mathbb{N}_0$. Note that a map of degree $\ell$ is completely determined by its components $\Phi^k:=\Phi|_{V^k}\colon V^k\to W^{k+\ell}$ for all $k\in\mathbb{N}_0$ and we sometimes write $\Phi\colon V^\bullet\to W^{\bullet+\ell}$ to indicate the degree shift.
In the following, a graded algebra is always understood as a graded vector space $\Omega^\bullet=\bigoplus\nolimits_{k\in\mathbb{N}_0}\Omega^k$ with associative multiplication $\wedge\colon\Omega^\bullet\otimes\Omega^\bullet\to\Omega^\bullet$ which is a map of degree $0$, i.e. $\Omega^k\wedge\Omega^\ell\subseteq\Omega^{k+\ell}$. The unit of $\Omega^\bullet$ is necessarily an element of $\Omega^0$. 
A differential graded algebra is a graded algebra $\Omega^\bullet$ with a degree $1$ map $\mathrm{d}\colon\Omega^\bullet\to\Omega^{\bullet+1}$, the differential, such that $\mathrm{d}\circ\mathrm{d}=0$ and such that the Leibniz rule
$$
\mathrm{d}(\omega\wedge\eta)=\mathrm{d}\omega\wedge\eta+(-1)^{|\omega|}\omega\wedge\mathrm{d}\eta
$$
is satisfied for homogeneous elements $\omega\in\Omega^{|\omega|}$, where $|\omega|$ denotes the degree of $\omega$ and for all $\eta\in\Omega^\bullet$. A morphism of differential graded algebras is a degree $0$-map which intertwines the multiplications, units and differentials in the obvious ways.
More details on differential graded algebras can be found in \cite{SchDC}.

\subsection*{Abbreviations}

\begin{center}
\begin{tabular}{l|l|l|l}
Abbreviation & Description & Notation & First appearance\\
\hline
 & & & \\
DC & Differential calculus & $\Omega^\bullet$ or $\Omega^\bullet(A)$ & Definition \ref{def:DC}\\
DCi & Differential calculi & & \\
& & & \\
DGA & Differential graded algebra & $(\Omega^\bullet,\wedge,\mathrm{d})$ & Previous paragraph\\
DGAs & Differential graded algebras & & \\
& & & \\
FODC & First order differential calculus & $(\Gamma,\mathrm{d})$ & Definition \ref{def:FODC}\\
FODCi & First order differential calculi & & \\
 & & & \\
QPB & Quantum principal bundle & $B=A^{\mathrm{co}H}\subseteq A$ & Definition \ref{def:QPB}\\
QPBs & Quantum principal bundles & &
\end{tabular}
\end{center}

\section{Noncommutative differential geometry}\label{C2}

In this preliminary section we recall the notion of noncommutative differential calculus, with basics about first order and higher order calculi in Section \ref{sec:DC} and the treatment of covariant calculi in Section \ref{sec:covDC}. The main results are the Woronowicz classification of covariant first order calculi on Hopf algebras and the expression of the corresponding maximal prolongations via relations of the Cartan-Maurer form.
Our main references for this standard material are \cite{BegMaj,DurII,SchDC,Wor89}.

\subsection{Differential calculi}\label{sec:DC}

One of the main concepts in noncommutative differential geometry is the notion of noncommutative differential calculus. We give the definition of first order differential calculus, their classification in terms of the universal calculus and recall the universal construction of higher order calculi via the maximal prolongation. This works in great generality for arbitrary associative unital algebras $A$.

\begin{definition}\label{def:FODC}
A \textbf{first order differential calculus} (abbreviated by FODC) on an algebra $A$ is a tuple $(\Gamma,\mathrm{d})$ such that
\begin{enumerate}
\item[i.)] $\Gamma$ is an $A$-bimodule,

\item[ii.)] $\mathrm{d}\colon A\to\Gamma$ is a map which satisfies the \textbf{Leibniz rule} 
\begin{equation}
	\mathrm{d}(aa')=\mathrm{d}(a)a'+a\mathrm{d}(a')
\end{equation}
for all $a,a'\in A$,

\item[iii.)] $\Gamma$ is generated by $A$, i.e. $\Gamma=A\mathrm{d}A:=\mathrm{span}_\Bbbk\{a\mathrm{d}(a')~|~a,a'\in A\}$.
\end{enumerate}
\end{definition}
The last condition is sometimes referred to as the \textbf{surjectivity condition}. In the following we sometimes write $\mathrm{d}a$ instead of $\mathrm{d}(a)$ in order to ease the notation. 
A morphism between a FODC $(\Gamma,\mathrm{d})$ on $A$ and another FODC $(\Gamma',\mathrm{d}')$ on another associative unital algebra $A'$ is a tuple $(\phi,\Phi)$, where $\phi\colon A\to A'$ is an algebra morphism and $\Phi\colon\Gamma\to\Gamma'$ is an $A$-bimodule morphism, where we view $\Gamma'$ as an $A$-bimodule via $\phi$. One further demands compatibility with the differentials: $\mathrm{d}'\circ\phi=\Phi\circ\mathrm{d}$. In case $A'=A$ and $\phi=\mathrm{id}_A$ we simply refer to $\Phi\colon\Gamma\to\Gamma'$ as the morphism of first order differential calculi (FODCi).

We call $(\Gamma',\mathrm{d}')$ a \textbf{quotient FODC} of $(\Gamma,\mathrm{d})$ if there is a surjective morphism $(\Gamma,\mathrm{d})\to(\Gamma',\mathrm{d}')$ of FODCi, i.e. such that the corresponding $A$-bimodule morphism $\Phi\colon\Gamma\to\Gamma'$ is surjective.

Every algebra $A$ admits the \textbf{trivial FODC} $\Gamma=\{0\}$ and the \textbf{universal FODC} $(\Gamma_u,\mathrm{d}_u)$. The latter is defined by $\Gamma_u:=\ker m_A$, where $m_A\colon A\otimes A\to A$ denotes the multiplication of $A$ and its kernel is endowed with the obvious $A$-bimodule structure, and $\mathrm{d}_u\colon A\to\Gamma_u$ is defined by $\mathrm{d}_u(a):=1\otimes a-a\otimes 1$ for all $a\in A$.
\begin{proposition}[{\cite[Proposition 1.1]{Wor89}}]
Let $A$ be an algebra. Every FODC $(\Gamma,\mathrm{d})$ on $A$ is a quotient of the universal FODC $(\Gamma_u,\mathrm{d}_u)$ on $A$. The corresponding surjective morphism of FODCi is given by
\begin{equation}
	\Gamma_u\twoheadrightarrow\Gamma,\qquad a^i\otimes b^i\mapsto a^i\mathrm{d}b^i,
\end{equation}
where a finite sum over repeated indices is understood.
\end{proposition}
Other examples of FODCi which are relevant for us will be given in Section \ref{sec:ex}.

\begin{definition}\label{def:DC}
A \textbf{differential calculus} (DC) on an algebra $A$ is a differential graded algebra $(\Omega^\bullet,\mathrm{d})$ which is generated in degree zero by $\Omega^0=A$. The latter condition means that for every $k>1$ we have
\begin{equation}
	\Omega^k=\mathrm{span}_\Bbbk\{a^0\mathrm{d}a^1\wedge\ldots\wedge\mathrm{d}a^k~|~a^0,a^1,\ldots,a^k\in A\}.
\end{equation}
We sometimes write $\Omega^\bullet(A)$ if we want to stress that $\Omega^\bullet$ is a DC on the algebra $A$. If we want to emphasize the underlying DGA structure of $\Omega^\bullet$ we write $(\Omega^\bullet,\wedge,\mathrm{d})$.
\end{definition}
A morphism of DCi is a morphism of the underlying differential graded algebras (DGAs). Quotients of DCi correspond to surjective morphisms of DCi.
\begin{remark}
Note that a morphism of DCi $\Phi\colon(\Omega^\bullet,\wedge,\mathrm{d})\to(\Omega^{'\bullet},\wedge',\mathrm{d}')$ is completely determined by its zeroth order and the DGA structures. In fact, for all $a^0,a^1,\ldots,a^k\in A$ we obtain
\begin{equation}
	\Phi(a^0\mathrm{d}(a^1)\wedge\ldots\wedge\mathrm{d}(a^k))
	=\Phi(a^0)\mathrm{d}'(\Phi(a^1))\wedge'\ldots\wedge'\mathrm{d}'(\Phi(a^k))
\end{equation}
and since $\Omega^\bullet$ is generated in degree $0$ this determines $\Phi$ completely. In particular, if $\Phi$ exists and its zeroth order is surjective, then $\Phi$ is surjective.
\end{remark}

\begin{example}
Given two DCi $\Omega^\bullet(A)$ and $\Omega^\bullet(A')$ on associative unital algebras $A$ and $A'$, we define a DC $\Omega^\bullet(A\otimes A')$ on the tensor product algebra $A\otimes A'$ as the tensor product of the corresponding DGAs. Namely, $\Omega^n(A\otimes A'):=\bigoplus_{k+\ell=n}\Omega^k(A)\otimes\Omega^\ell(A')$ with multiplication
\begin{equation}
	(\omega\otimes\omega')\wedge(\eta\otimes\eta'):=(-1)^{|\omega'|\cdot|\eta|}(\omega\wedge\eta)\otimes(\omega'\wedge\eta'),
\end{equation}
where $\omega,\eta\in\Omega^\bullet(A)$ and $\omega',\eta'\in\Omega^\bullet(A')$. The differential is determined by
\begin{equation}
	\mathrm{d}_\otimes(\omega\otimes\omega'):=\mathrm{d}(\omega)\otimes\omega'+(-1)^{|\omega|}\omega\otimes\mathrm{d}(\omega')
\end{equation}
for all $\omega\in\Omega^\bullet(A)$ and $\omega'\in\Omega^\bullet(A')$. All axioms of a differential calculus, including the surjectivity property, are easily verified. We call $\Omega^\bullet(A\otimes A')$ the \textbf{tensor product calculus} on $A\otimes A'$.
\end{example}

Clearly, given a DC $\Omega^\bullet$ on $A$ we obtain a FODC $(\Gamma,\mathrm{d})$ on $A$ by setting $\Gamma:=\Omega^1$ and $\mathrm{d}\colon A\to\Gamma$ the restriction of $\mathrm{d}\colon\Omega^\bullet\to\Omega^{\bullet+1}$. We call $(\Gamma,\mathrm{d})$ the \textbf{truncation} of $\Omega^\bullet$ and $\Omega^\bullet$ an \textbf{extension} of $(\Gamma,\mathrm{d})$. Every FODC $(\Gamma,\mathrm{d})$ on $A$ admits the \textbf{trivial extension} $\Omega^k=\{0\}$ for $k>1$ and the \textbf{maximal prolongation} $\Gamma^\wedge$. The latter is the extension of $(\Gamma,\mathrm{d})$ defined in the following way: consider the tensor algebra 
\begin{equation}
	\Gamma^{\otimes_A}:=\bigoplus\nolimits_{k\in\mathbb{N}_0}\Gamma^{\otimes_A k}=A\oplus\Gamma\oplus(\Gamma\otimes_A\Gamma)\oplus\ldots
\end{equation}
and its graded ideal $I^\wedge$ generated by elements $\mathrm{d}a^i\otimes_A\mathrm{d}b^i\in\Gamma^{\otimes_A2}$ (sum over repeated indices understood) with $a^i,b^i\in A$ such that $a^i\mathrm{d}b^i=0$. We denote by $\Gamma^\wedge:=\Gamma^{\otimes_A}/I^\wedge$ the graded quotient algebra. The differential $\mathrm{d}\colon A\to\Gamma$ extends to $1$-forms on the quotient
\begin{equation}
	\mathrm{d}\colon\Gamma^{\wedge 1}\to\Gamma^{\wedge 2},\qquad
	a^i\mathrm{d}b^i\mapsto[\mathrm{d}a^i\otimes_A\mathrm{d}b^i]
\end{equation}
precisely by the definition of $I^\wedge$.
\begin{proposition}[{\cite[Appendix B]{DurI}}]\label{prop:maxprol}
Let $(\Gamma,\mathrm{d})$ be a FODC on an algebra $A$. Then, the maximal prolongation $\Gamma^\wedge$ is a DC on $A$ and an extension of $(\Gamma,\mathrm{d})$. Moreover, every extension $\Omega^\bullet$ of $(\Gamma,\mathrm{d})$ is a quotient of $\Gamma^\wedge$.
\end{proposition}

\subsection{Covariant differential calculi}\label{sec:covDC}

In this section we implement the notion of "symmetry" of an algebra $A$ in the context of noncommutative differential geometry. The role of the "symmetry group" is played by a Hopf algebra $H$ which coacts on $A$ in a way that respects the algebra structure of $A$. Modules of $A$, particularly differential calculi on $A$, have to be compatible with this coaction and maps between comodules, like the differential, are assumed to commute with the coactions. We introduce the corresponding mathematical concepts in the following and discuss calculi in this covariant context.

We fix a Hopf algebra $H$ and consider a \textbf{right $H$-comodule algebra} $A$. The latter is an algebra and a right $H$-comodule such that the corresponding right $H$-coaction $\Delta_A\colon A\to A\otimes H$ is an algebra homomorphism.
We call an $A$-bimodule $M$ a \textbf{right $H$-covariant $A$-bimodule} if there is a right $H$-coaction $\Delta_M\colon M\to M\otimes H$ such that
\begin{equation}
	\Delta_M(a\cdot m\cdot a')=\Delta_A(a)\Delta_M(m)\Delta_A(a')
\end{equation}
for all $a,a'\in A$ and $m\in M$. On the right hand side we use the obvious tensor product multiplication, which reads 
$
(a\cdot m\cdot a')_0\otimes(a\cdot m\cdot a')_1
=a_0\cdot m_0\cdot a'_0\otimes a_1m_1a'_1
$
in Sweeder's notation. The vector subspace of right $H$-coinvariant elements is denoted by 
\begin{equation}
	M^{\mathrm{co}H}=\{m\in M~|~\Delta_M(m)=m\otimes 1\}\subseteq M.
\end{equation}
Similarly, in case $A$ is a left $H$-comodule algebra (resp. $H$-bicomodule algebra), left $H$-covariant $A$-bimodules (resp. $H$-bicovariant $A$-bimodules) are defined. The vector subspace of left $H$-coinvariant elements of a left $H$-comodule $M$ with left $H$-coaction ${}_M\Delta\colon M\to H\otimes M$ is denoted by 
\begin{equation}
	{}^{\mathrm{co}H}M=\{m\in M~|~{}_M\Delta(m)=1\otimes m\}\subseteq M.
\end{equation} 
In the following we often omit $\cdot$ and simply write $am$ and $ma$ for the left and right action of $a$ on $m$.
\begin{definition}\label{def:CovCal}
Let $A$ be a right $H$-comodule algebra.
\begin{enumerate}
\item[i.)] A FODC $(\Gamma,\mathrm{d})$ on $A$ is called right $H$-covariant if $\Gamma$ is a right $H$-covariant $A$-bimodule and $\mathrm{d}\colon A\to\Gamma$ is right $H$-colinear.

\item[ii.)] A DC $\Omega^\bullet$ on $A$ is called right $H$-covariant if $\Omega^\bullet$ is a graded right $H$-comodule algebra and $\mathrm{d}\colon\Omega^\bullet\to\Omega^{\bullet+1}$ is right $H$-colinear. Explicitly, $\Omega^\bullet$ is right $H$-covariant if for every $k>0$ there is a right $H$-coaction $\Delta_{\Omega^k}\colon\Omega^k\to\Omega^k\otimes H$ such that
\begin{equation}
	\Delta_{\Omega^{k+\ell}}(\omega\wedge\eta)=\Delta_{\Omega^k}(\omega)\wedge_\otimes\Delta_{\Omega^\ell}(\eta)\qquad\text{and}\qquad
	\Delta_{\Omega^{k+1}}(\mathrm{d}(\omega))=(\mathrm{d}\otimes\mathrm{id})(\Delta_{\Omega^k}(\omega))
\end{equation}
for all $\omega\in\Omega^k$ and $\eta\in\Omega^\ell$, where $(\omega\otimes h)\wedge_\otimes(\omega'\otimes h'):=\omega\wedge\omega'\otimes hh'$ for all $\omega,\omega'\in\Omega^\bullet$ and $h,h'\in H$.
\end{enumerate}
Similarly one defines left $H$-covariant FODCi of a left $H$-comodule algebra and $H$-bicovariant FODCi for an $H$-bicomodule algebra and the same for DCi.
\end{definition}
\begin{remark}
Let $A$ be a right $H$-comodule algebra and $(\Gamma,\mathrm{d})$ a FODC on $A$. We give two characterizations of right $H$-covariance of $(\Gamma,\mathrm{d})$, which are equivalent to the one in Definition \ref{def:CovCal} $i.)$. The first one is the definition originally given by Woronowicz in \cite{Wor89}, while the second one is due to Schauenburg, see \cite{SchDC}.
\begin{enumerate}
\item[i.)] $(\Gamma,\mathrm{d})$ is right $H$-covariant if and only if the map 
\begin{equation}\label{eq:DelGam}
	\Delta_\Gamma\colon\Gamma\to\Gamma\otimes H,\qquad\sum_ia^i\mathrm{d}b^i\mapsto\sum_ia^i_0\mathrm{d}(b^i_0)\otimes a^i_1b^i_1,
\end{equation}
with $a^i,b^i\in A$,
is well-defined. In particular, the right $H$-coaction $\Delta_\Gamma\colon\Gamma\to\Gamma\otimes H$ on $\Gamma$ for a right $H$-covariant FODC $(\Gamma,\mathrm{d})$ is necessarily of the form \eqref{eq:DelGam}.

\item[ii.)] $(\Gamma,\mathrm{d})$ is right $H$-covariant if and only if $\Delta_A\colon A\to A\otimes H$ extends to a morphism of FODCi 
\begin{equation}
	(\Delta_A,\Delta_\Gamma)\colon(\Gamma,\mathrm{d})\to(\Gamma\otimes H,\mathrm{d}\otimes\mathrm{id}_H),
\end{equation}
where $(\Gamma\otimes H,\mathrm{d}\otimes\mathrm{id}_H)$ is understood as a FODC on $A\otimes H$ and the $A$-bimodule actions on $A\otimes H$ are given by $\Delta_A$. The compatibility of $\Delta_\Gamma$ with the differentials can be expressed via commutativity of the diagram
\begin{equation}
\begin{tikzcd}
\Gamma \arrow{r}{\Delta_\Gamma}
& \Gamma\otimes H \\
A \arrow{u}{\mathrm{d}} \arrow{r}{\Delta_A}
& A\otimes H \arrow{u}[swap]{\mathrm{d}\otimes\mathrm{id}_H}
\end{tikzcd},
\end{equation}
which, in the language of \cite{SchDC}, says that $\Delta_A$ is (1-time) \textbf{differentiable} with its differential being $\Delta_\Gamma$.
\end{enumerate}
The generalizations of these characterizations to higher order calculi are straightforward:
a DC $\Omega^\bullet$ on a right $H$-comodule algebra $A$ is right $H$-covariant if and only if for all $k>0$ the map
\begin{equation}\label{eq:DelGamK}
\Delta_{\Omega^k}\colon\Omega^k\to\Omega^k\otimes H,\qquad
a^0\mathrm{d}a^1\wedge\ldots\wedge\mathrm{d}a^k\mapsto
a^0_0\mathrm{d}(a^1_0)\wedge\ldots\wedge\mathrm{d}(a^k_0)\otimes a^0_1a^1_1\ldots a^k_1,
\end{equation}
for all $a^0\mathrm{d}a^1\wedge\ldots\wedge\mathrm{d}a^k\in\Omega^k$, is well-defined, which is the case if and only if $\Delta_A\colon A\to A\otimes H$ extends to a morphism of DCi
\begin{equation}
	\Delta_{\Omega^\bullet}\colon\Omega^\bullet\to\Omega^\bullet\otimes H,
\end{equation}
where $\Omega^\bullet\otimes H$ is understood as a DC over $A\otimes H$. In particular, the restriction of the right $H$-coaction $\Delta_{\Omega^\bullet}$ to $\Omega^k$ is necessarily of the form \eqref{eq:DelGamK} and $\Delta_A$ is $k$-times differentiable (for all $k>0$) in the sense of \cite{SchDC}.
\end{remark}
Clearly, there are left $H$-covariant and $H$-bicovariant versions of the above characterizations.

An important special case is that of the $H$-bicomodule algebra $H$ endowed with the comultiplication. In this case a right (resp. left) $H$-covariant $H$-bimodule is simply called a right or left covariant $H$-bimodule. Right $H$-covariant calculi on $H$ are simply called right covariant calculi and similarly for the left covariant and bicovariant cases.
\begin{lemma}[Fundamental Theorem of Hopf modules, {\cite{Sweedler69}}]\label{FundThm}
~
\begin{enumerate}
\item[i.)] For every right covariant $H$-bimodule $M$ there is an isomorphism of right covariant $H$-bimodules
\begin{equation}\label{eq:FT1}
M\xrightarrow{\cong}M^{\mathrm{co}H}\otimes H,\qquad
m\mapsto m_0\cdot S(m_1)\otimes m_2,
\end{equation}
where $M^{\mathrm{co}H}$ is an $H$-bimodule via $h\cdot(\overline{m}\otimes h')\cdot h'':=h_1\cdot\overline{m}\cdot S(h_2)\otimes h_3h'h''$ for all $\overline{m}\in M^{\mathrm{co}H}$ and $h,h',h''\in H$.
If $M$ is even a bicovariant $H$-bimodule, then \eqref{eq:FT1} is an isomorphism of bicovariant $H$-bimodules.
	
\item[ii.)] For every left covariant $H$-bimodule $M$ there is an isomorphism of left covariant $H$-bimodules
\begin{equation}\label{eq:FT2}
M\xrightarrow{\cong}H\otimes{}^{\mathrm{co}H}M,\qquad
m\mapsto m_{-2}\otimes S(m_{-1})\cdot m_0,
\end{equation}
where ${}^{\mathrm{co}H}M$ is an $H$-bimodule via $h\cdot(h'\otimes\overline{m})\cdot h'':=hh'h''_1\otimes S(h''_2)\cdot\overline{m}\cdot h''_3$ for all $\overline{m}\in{}^{\mathrm{co}H}M$ and $h,h',h''\in H$.
If $M$ is even a bicovariant $H$-bimodule, then \eqref{eq:FT2} is an isomorphism of bicovariant $H$-bimodules.
\end{enumerate}
\end{lemma}
Given a left covariant FODC $(\Gamma,\mathrm{d})$ on $H$ we denote the space of $1$-forms which are invariant under the left $H$-coaction by
\begin{equation}
	\Lambda^1:={}^{\mathrm{co}H}\Gamma=\{\omega\in\Gamma~|~{}_\Gamma\Delta(\omega)=1\otimes\omega\}.
\end{equation}
We further define the \textbf{quantum Cartan--Maurer form}
\begin{equation}
	\varpi\colon H^+\to\Lambda^1,\qquad
	\varpi(h):=S(h_1)\mathrm{d}(h_2),
\end{equation}
where $H^+:=\ker\varepsilon$. As one easily verifies, $\varpi$ is right $H$-linear ($\Lambda^1$ is endowed with the adjoint right $H$-action) and surjective. Thus, $\Lambda^1\cong H^+/\ker\varpi$ as right $H$-modules. On the other hand, every right $H$-ideal $I$ determines a left covariant FODC $(\Gamma,\mathrm{d})$ on $H$ via $\Gamma:=H\otimes H^+/I$ and $\mathrm{d}\colon H\to\Gamma$ defined by $\mathrm{d}(h):=(\mathrm{id}\otimes\pi)(\Delta(h)-h\otimes 1)$, where $\pi\colon H^+\to H^+/I$ denotes the projection. This completely determines covariant FODCi on Hopf algebras and leads to the following classification theorem of Woronowicz, see also \cite[Theorem 2.26]{BegMaj}.

\begin{theorem}[{\cite{Wor89}}]\label{ThmWor}
There is a $1$$:$$1$-correspondence
\begin{equation}
\begin{Bmatrix}
\text{left covariant}\\
\text{FODC}\\
\text{on $H$}
\end{Bmatrix}\xleftrightarrow{1:1}
\begin{Bmatrix}
\text{right}\\
\text{$H$-ideals}\\
\text{in $H^+$}
\end{Bmatrix}
\end{equation}
Moreover, a FODC on $H$ is bicovariant if and only if the corresponding right ideal $I$ is closed under the adjoint right $H$-coaction, i.e. if $\mathrm{Ad}(I)\subseteq I\otimes H$, where $\mathrm{Ad}\colon H\to H\otimes H$, $\mathrm{Ad}(h):=h_2\otimes S(h_1)h_3$.
\end{theorem}
By the previous theorem we can uniquely identify any left covariant FODC $(\Gamma,\mathrm{d})$ on $H$ with a quotient $H\otimes\Lambda^1$, where $\Lambda^1=H^+/I$ for some right $H$-ideal $I\subseteq H^+$. In terms of the Cartan--Maurer form, the corresponding isomorphism reads
\begin{equation}
\begin{split}
\Gamma&\xrightarrow{\cong}H\otimes\Lambda^1\\
h\mathrm{d}(h')&\mapsto hh'_1\otimes\varpi(\pi_\varepsilon(h'_2)),
\end{split}
\end{equation}
where $\pi_\varepsilon\colon H\to H^+$, $\pi_\varepsilon(h):=h-\varepsilon(h)1$ is the canonical projection.
A natural question is if the maximal prolongation of $\Gamma$ can also be expressed in terms of the Cartan--Maurer form.
\begin{proposition}[{\cite[Proposition 2.31]{BegMaj}}]\label{prop:MaxGammaH}
Let $(\Gamma,\mathrm{d})$ be a left covariant FODC on $H$ with corresponding right $H$-ideal $I\subseteq H^+$. Consider the maximal prolongation $\Omega^\bullet(H)$ of $(\Gamma,\mathrm{d})$ and denote the graded subspace of left coinvariant forms by $\Lambda^\bullet:=\{\omega\in\Omega^\bullet(H)~|~{}_{\Omega^\bullet(H)}\Delta(\omega)=1\otimes\omega\}\subseteq\Omega^\bullet(H)$. Then $\Lambda^\bullet$ is isomorphic to the free graded algebra generated by $\Lambda^1$ modulo the relations
\begin{equation}
	\wedge\circ(\varpi\circ\pi_\varepsilon\otimes\varpi\circ\pi_\varepsilon)\Delta(I)=0,
\end{equation}
where above $\wedge$ denotes the product in the free graded algebra. Moreover, the Cartan--Maurer equation
\begin{equation}\label{eq:CartanMaurer}
\mathrm{d}(\varpi(\pi_\varepsilon(h)))+\varpi(\pi_\varepsilon(h_1))\wedge\varpi(\pi_\varepsilon(h_2))=0
\end{equation}
holds for all $h\in H$.
\end{proposition}

Examples of covariant calculi will be discussed in Section \ref{sec:ex}.

\section{Quantum principal bundles}\label{C3}

In this section we present the main concept of this paper, the quantum principal bundle in the sense of \DJ ur\dj evi\'c. Such an object is given by a faithfully flat Hopf--Galois extension $B=A^{\mathrm{co}H}\subseteq A$, where $H$ is a Hopf algebra, $A$ a comodule algebra and $B$ the subalgebra of coinvariant elements. For the convenience of the reader we summarize some of the main features of Hopf--Galois extensions in Section \ref{sec:HG}. In particular, we discuss the special cases of cleft and trivial extension, together with their crossed product and smash product structure, which will be relevant for Section \ref{ex:CP}.
We then continue to give a differential structure to $B=A^{\mathrm{co}H}\subseteq A$, starting with a bicovariant (even complete) calculus on the structure Hopf algebra $H$ in Section \ref{sec:StrHA}, a calculus on the total space algebra $A$ in Section \ref{sec:totalforms} and a calculus on the base algebra $B$ in Section \ref{sec:baseforms}. Vertical forms and horizontal forms are introduced in Sections \ref{sec:VF} and \ref{sec:totalforms}, respectively and their first orders are featured in the noncommutative Atiyah sequence in Section \ref{sec:totalforms}. Finally, in Section \ref{sec:compare} we compare this approach of \DJ ur\dj evi\'c to other related concepts in the literature, particularly to the quantum principal bundle theory of Brzezi\'nski--Majid \cite{BrzMaj}. 
As pointed out in the in the introduction, our main reference is \cite{DurII}. However, we present the theory in the framework of Hopf algebras and comodule algebras rather than compact quantum groups and compact quantum spaces. We provide elementary algebraic proofs and clarify some of the material.

\subsection{Hopf--Galois extensions}\label{sec:HG}

Consider a Hopf algebra $H$ and a right $H$-comodule algebra $A$. The subalgebra of elements which are invariant under the coaction is denoted by
\begin{equation}
	B:=A^{\mathrm{co}H}:=\{a\in A~|~\Delta_A(a)=a\otimes 1\}\subseteq A.
\end{equation}
\begin{definition}[\cite{KrTa}]
For a right $H$-comodule algebra $A$ we call $B:=A^{\mathrm{co}H}\subseteq A$ a \textbf{Hopf--Galois extension} if the canonical map
\begin{equation}
	\chi\colon A\otimes_BA\to A\otimes H,\qquad
	a\otimes a'\mapsto a\Delta_A(a')=aa'_0\otimes a'_1
\end{equation}
is a bijection.
In this case, we define the \textbf{translation map} $\tau\colon H\to A\otimes_BA$ by $\tau(h):=\chi^{-1}(1\otimes h)$ for all $h\in H$.
We frequently employ the short notation
\begin{equation}
	\tau(h)=:h^{\langle 1\rangle}\otimes_Bh^{\langle 2\rangle}\in A\otimes_BA
\end{equation}
for $h\in H$.
\end{definition}
It is well-known (see e.g. \cite{Brz96}) that the following equations hold for all $h,g\in H$, $a\in A$.
\begin{align}
	h^{\langle 1\rangle}(h^{\langle 2\rangle})_0\otimes(h^{\langle 2\rangle})_1&=1_A\otimes h\label{tau6}\\
	a_0(a_1)^{\langle 1\rangle}\otimes_B(a_1)^{\langle 2\rangle}&=1_A\otimes_B a\label{tau5}\\
	\tau(hg)&=h^{\langle 1\rangle}g^{\langle 1\rangle}\otimes_Bg^{\langle 2\rangle}h^{\langle 2\rangle}\label{tau2}\\
	h^{\langle 1\rangle}h^{\langle 2\rangle}&=\varepsilon(h)1\label{tau1}\\
	h^{\langle 1\rangle}\otimes_B(h^{\langle 2\rangle})_0\otimes(h^{\langle 2\rangle})_1&=(h_1)^{\langle 1\rangle}\otimes_B(h_1)^{\langle 2\rangle}\otimes h_2\label{tau3}\\
	(h^{\langle 1\rangle})_0\otimes_Bh^{\langle 2\rangle}\otimes(h^{\langle 1\rangle})_1&=(h_2)^{\langle 1\rangle}\otimes_B(h_2)^{\langle 2\rangle}\otimes S(h_1)\label{tau4}
\end{align}
Moreover, the "centrality property"
\begin{equation}\label{tau7}
	b\tau(h)=\tau(h)b
\end{equation}
holds for all $h\in H$ and $b\in B$.

\medskip

Recall that $A$ is \textbf{faithfully flat} as a right $B=A^{\mathrm{co}H}$-module if the functor $A\otimes_B\cdot\colon{}_B\mathcal{M}\to\mathrm{Vec}_\Bbbk$ from the category of left $B$-modules to the category of vector spaces preservers and reflects exact sequences. By definition, a Hopf--Galois extension is faithfully flat if $A$ is faithfully flat as a right $B=A^{\mathrm{co}H}$-module. Faithful flatness of a Hopf--Galois extension ensures a preferable "geometric" behavior of the extension: it ensures the existence of a strong connection (as we will discuss in Section \ref{sec:connections}) and is equivalent to the categorical equivalence of $B$-modules with $H$-covariant $A$-modules (seminal paper \cite{Schn90} of Schneider). For this reason faithful flatness is included in the following definition of quantum principal bundle.
\begin{definition}\label{def:QPB}
	A right $H$-comodule algebra $A$ is called a \textbf{quantum principal bundle} (QPB) if $B=A^{\mathrm{co}H}\subseteq A$ is a faithfully flat Hopf--Galois extension. In this case we call $A$ the \textbf{total space algebra}, $B$ the \textbf{base space algebra} and $H$ the \textbf{structure Hopf algebra} of the QPB.
\end{definition}

\begin{remark}
	Let us make some comments on Definition \ref{def:QPB}.
	\begin{enumerate}
		\item[i.)] In \cite[Definition 3.1]{DurII} \DJ ur\dj evi\'c defines quantum principal bundles in the context of compact quantum groups. There, only surjectivity of the Hopf--Galois map is required and the faithful flatness assumption is not explicitly stated. Later, in \cite{DurHG}, \DJ ur\dj evi\'c clarifies that the Hopf--Galois map is also injective without additional assumptions. We thus adopt this revised notion of QPB in Definition \ref{def:QPB}, with the difference that we do not assume a compact quantum group but faithful flatness of the extension instead.
		
		\item[ii.)] Since we assume an invertible antipode and are working over a field $\Bbbk$ the Definition \ref{def:QPB} is equivalent to that of a \textbf{principal comodule algebra} in \cite{DGH}.
		
		\item[iii.)] To ease the setup we avoid including a $*$-structure in the definition of QPB, while this is among the assumptions in \cite{DurII}. 
		
		\item[iv.)] Note that the results of Section \ref{C3}, \ref{C4.1} and \ref{C5} do not require faithful flatness, i.e. they are valid for arbitrary Hopf--Galois extensions. Faithful flatness is only invoked when we consider strong connections in Section \ref{C4.2}. As we commented before, faithfully flat Hopf--Galois extension are more "geometric" than arbitrary extension, which is the reason why this assumption is included in the definition. Furthermore, all our examples are in fact faithfully flat. Thus, for simplicity, we assume faithful flatness from the beginning in the definition of QPB. 
	\end{enumerate}
\end{remark}
In the following we discuss the two well-studied classes of faithfully flat Hopf--Galois extensions (or quantum principal bundles), given by cleft extensions and trivial extensions.
\begin{definition}[\cite{DoiTak}]
Let $A$ be a right $H$-comodule algebra. Then we call $B:=A^{\mathrm{co}H}\subseteq A$ a
\begin{enumerate}
	\item[i.)] \textbf{cleft extension} if there exists a convolution invertible right $H$-comodule morphism $j\colon H\to A$, the so-called \textbf{cleaving map}, where $H$ is understood as a right $H$-comodule via the comultiplication.
		
	\item[ii.)] \textbf{trivial extensions} if there exists a convolution invertible right $H$-comodule algebra morphism $j\colon H\to A$.
\end{enumerate}
\end{definition}
Recall that the convolution inverse of a map $j\colon H\to A$ is a map $j^{-1}\colon H\to A$ satisfying $j(h_1)j^{-1}(h_2)=\varepsilon(h)1=j^{-1}(h_1)j(h_2)$ for all $h\in H$. One assumes without loss of generality that $j$ and $j^{-1}$ are unital. As one easily verifies, for every cleft extension the inverse of the canonical map is given by $\chi^{-1}(a\otimes h)=aj^{-1}(h_1)\otimes j(h_2)$ for all $a\otimes h\in A\otimes H$. Thus, every cleft extension is a Hopf--Galois extension. 
It is shown in \cite{BrJaMa} that for a cleft extension $B=A^{\mathrm{co}H}\subseteq A$ it follows that $A$ is faithfully flat as a right $B$-module. Thus, the corresponding Hopf--Galois extension is automatically faithfully flat. Note that we assume the antipode of $H$ to be invertible, which is crucial for this statement.
Clearly every trivial extension is in particular a cleft extension.

We recall that cleft extensions and trivial extensions correspond to crossed product algebras and smash product algebras, respectively. For this we follow the presentation of \cite{Mon94}. Before stating the characterization theorem we recall the involved algebraic structures. 
For this, let $B$ be an arbitrary algebra. A \textbf{measure} is a map $\cdot\colon H\otimes B\to B$ such that
\begin{equation}
	h\cdot(bb')=(h_1\cdot b)(h_2\cdot b')\qquad\text{and}\qquad
	h\cdot 1=\varepsilon(h)1
\end{equation}
for all $h\in H$ and $b,b'\in B$. In order to avoid confusion we want to point out that $\cdot$ sometimes denotes an action, while the measure is \textit{not} an action, in general. From the context it will be clear if an action or a measure is indicated. Given a measure $\cdot$ and any convolution invertible map $\sigma\colon H\otimes H\to B$ we define the map
\begin{equation}
\begin{split}
	m_{\#_\sigma}\colon(B\otimes H)\otimes(B\otimes H)&\to B\otimes H\\
	(b\otimes h)\otimes(b'\otimes h')&\mapsto b(h_1\cdot b')\sigma(h_2\otimes h'_1)\otimes h_3h'_2
\end{split}
\end{equation}
and denote $B\#_\sigma H:=(B\otimes H,m_{\#_\sigma})$. The elements of $B\#_\sigma H$ are sometimes written $b\#_\sigma h$ instead of $b\otimes h$.
\begin{lemma}[{\cite[Lemma 7.1.2]{Mon94}}]\label{lem:crossedalg}
The space $B\#_{\sigma}H$ is an associative algebra with unit $1_{B}\otimes1_{H}$ if and only if the following two conditions hold: 
\begin{enumerate}
	\item[i.)] $B$ is a $\sigma$-twisted left $H$-module, i.e., $1_{H}\cdot b=b$ for all $b\in B$ and 
\begin{equation}
		h\cdot(h'\cdot b)=\sigma(h_{1}\otimes h'_{1})(h_{2}h'_{2}\cdot b)\sigma^{-1}(h_{3}\otimes h'_{3})
\end{equation}
for all $h,h'\in H$ and $b\in B$.
		
	\item[ii.)] $\sigma$ is a 2-cocycle with values in $B$, i.e. $\sigma(h\otimes1_{H})=\varepsilon(h)1_{B}=\sigma(1_{H}\otimes h)$ and 
\begin{equation}
		(h_{1}\cdot\sigma(h'_{1}\otimes h''_{1}))\sigma(h_{2}\otimes h'_{2}h''_{2})=\sigma(h_{1}\otimes h'_{1})\sigma(h_{2}h'_{2}\otimes h'')
\end{equation}
for all $h,h',h''\in H$.
\end{enumerate}
\end{lemma}
Given a $\sigma$-twisted left $H$-module $B$ and $\sigma$, a 2-cocycle with values in $B$, we call $B\#_\sigma H$ the corresponding \textbf{crossed product algebra}. If $\sigma$ is the trivial $2$-cocycle $h\otimes h'\mapsto\epsilon(hh')1$ we call $B\#_\sigma H$ the \textbf{smash product algebra} and write $B\#H$ instead. The following proposition is due to \cite{DoiTak}, see also \cite[Theorem 7.2.2]{Mon94}.

\begin{proposition}
There is a $1$$:$$1$-correspondence of crossed product algebras and cleft extensions. Explicitly,
\begin{enumerate}
\item[i.)] every crossed product algebra $B\#_{\sigma}H$ is a cleft extension $B\subseteq B\#_{\sigma}H$ with corresponding cleaving map given by $j\colon H\to B\#_{\sigma}H,\ h\mapsto1_{B}\otimes h$.
		
\item[ii.)] for every cleft extension $B\subseteq A$ with cleaving map $j:H\to A$ we define a measure on $B$
\begin{equation}
	\cdot\colon H\otimes B\to B,\qquad h\otimes b\mapsto h\cdot b:=j(h_{1})bj^{-1}(h_{2}) 
\end{equation}
and a 2-cocycle with values in $B$
\begin{equation}
	\sigma\colon H\otimes H\to B,\qquad h\otimes h'\mapsto\sigma(h\otimes h'):=j(h_{1})j(h'_{1})j^{-1}(h_{2}h'_{2}). 
\end{equation}
Then $A\cong B\#_{\sigma}H$ are isomorphic as right $H$-comodule algebras via $\theta\colon a\mapsto a_{0}j^{-1}(a_{1})\otimes a_{2}$, with inverse $\theta^{-1}\colon b\otimes h\mapsto bj(h)$.
\end{enumerate}
A cleft extension is a trivial extension if and only if the corresponding crossed product algebra is a smash product algebra.
\end{proposition}

\begin{example}\label{ex:HG}
We recall some fundamental examples of Hopf--Galois extensions, which will be relevant for our purposes later on. More details on Hopf--Galois extensions can be found in \cite{Mon09}.
\begin{enumerate}
\item[i.)] Any Hopf algebra $H$ gives a Hopf--Galois extension $\Bbbk=H^{\mathrm{co}H}\subseteq H$, with $\chi^{-1}\colon H\otimes H\to H\otimes H$, $h\otimes h'\mapsto hS(h'_1)\otimes h'_2$ the inverse of the canonical map. Moreover, for any bialgebra $H$, seen as a right comodule algebra over itself, one obtains $\Bbbk=H^{\mathrm{co}H}\subseteq H$ and this is a Hopf--Galois extension if and only if $H$ is a Hopf algebra.
		
\item[ii.)] We define the \textbf{noncommutative $2$-torus} $A=\mathcal{O}_\theta(\mathbb{T}^2)=\mathbb{C}[u,u^{-1},v,v^{-1}]/\langle vu-e^{\mathrm{i}\theta}uv\rangle$ as the algebra generated by two invertible generators $u,v$, module the relation $vu=e^{\mathrm{i}\theta}uv$, where $\theta\in\mathbb{R}$ is a real parameter. It is a right comodule algebra over $H=\mathcal{O}(U(1))=\mathbb{C}[t,t^{-1}]$, with coaction determined on generators by
\begin{equation}
\Delta_A\colon\begin{pmatrix}
	u \\
	v
\end{pmatrix}\mapsto\begin{pmatrix}
	u \\
	v
\end{pmatrix}\otimes\begin{pmatrix}
	t \\
	t^{-1}
\end{pmatrix}
\end{equation}
and extended as an algebra morphism $\Delta_A\colon A\to A\otimes H$. The subalgebra of coinvariants is $B=A^{\mathrm{co}H}=\mathrm{span}_\mathbb{C}\{(uv)^k~|~k\in\mathbb{Z}\}\}$ and we obtain a cleft extension $B\subseteq A$, with cleaving map $j\colon H\to A$ determined by
\begin{equation}
j\colon\begin{pmatrix}
	t^k\\
	t^{-k}
\end{pmatrix}\mapsto\begin{pmatrix}
	u^k\\
	v^k
\end{pmatrix}
\end{equation}
for all $k\geq 0$. Note that the extension is not trivial, since $j$ is not an algebra map. This example is taken from \cite{KhaLanVS}.
		
\item[iii.)] Let $q\in\mathbb{C}$ be non-zero and not a root of unity. The \textbf{$q$-deformed quantum group of special linear $2\times 2$-matrices} $A=\mathcal{O}_q(\mathrm{SL}_2(\mathbb{C}))=\mathbb{C}[\alpha,\beta,\gamma,\delta]/I$ is defined as the algebra generated by $\alpha,\beta,\gamma,\delta$, modulo the ideal $I$ generated by the Manin relations
\begin{align}
		\beta\alpha=q\alpha\beta,\quad
		\gamma\alpha=q\alpha\gamma,\quad
		\delta\beta=q\beta\delta,\quad
		\delta\gamma=q\gamma\delta,\quad
		\delta\alpha-\alpha\delta=(q-q^{-1})\beta\gamma,\quad
		\gamma\beta=\beta\gamma
\end{align}
and the quantum determinant relation
\begin{equation}
	\alpha\delta-q^{-1}\beta\gamma=1.
\end{equation}
It is well-known, that $A$ is a Hopf algebra with comultiplication, counit and antipode determined on generators by
\begin{align*}
	\Delta\colon\begin{pmatrix}
		\alpha & \beta\\
		\gamma & \delta
\end{pmatrix}\mapsto
\begin{pmatrix}
		\alpha & \beta\\
		\gamma & \delta
\end{pmatrix}\otimes\begin{pmatrix}
		\alpha & \beta\\
		\gamma & \delta
\end{pmatrix},\qquad
		\varepsilon\colon\begin{pmatrix}
		\alpha & \beta\\
		\gamma & \delta
\end{pmatrix}\mapsto\begin{pmatrix}
		1 & 0\\
		0 & 1
\end{pmatrix},\qquad
		S\colon\begin{pmatrix}
		\alpha & \beta\\
		\gamma & \delta
\end{pmatrix}\mapsto\begin{pmatrix}
		\delta & -q\beta\\
		-q^{-1}\gamma & \alpha
\end{pmatrix}
\end{align*}
and extended as (anti-)algebra morphisms.
There is a right $H=\mathcal{O}(U(1))=\mathbb{C}[t,t^{-1}]$-coaction on $A$ determined by
\begin{align}
	\Delta_A\colon\begin{pmatrix}
	\alpha & \beta\\
	\gamma & \delta
\end{pmatrix}\mapsto\begin{pmatrix}
	\alpha & \beta\\
	\gamma & \delta
\end{pmatrix}\otimes\begin{pmatrix}
	t & 0\\
	0 & t^{-1}
\end{pmatrix},
\end{align}
structuring $A$ as a right $H$-comodule algebra. The subalgebra of coinvariants $B=A^{\mathrm{co}H}=\mathbb{C}_q[\mathbb{S}^2]=\mathrm{span}_\mathbb{C}\{B_\pm,B_0\}$ is the \textbf{Podle\'s sphere} with generators $B_-=\alpha\beta$, $B_+=\gamma\delta$, $B_0=\beta\gamma$, satisfying the relations
\begin{equation}
	B_\pm B_0=q^{\pm 2}B_0B_\pm,\quad
	B_+B_-=q^4B_-B_++(1-q^2)B_0,\quad
	B_-B_+=q^{-2}B_0(1+q^{-1}B_0).
\end{equation}
One can show that $B\subseteq A$ is a faithfully flat Hopf--Galois extension, see e.g. \cite[Example 6.26]{BrJaMa}.
\end{enumerate}
\end{example}

\subsection{Bicovariant calculi on the structure Hopf algebra}\label{sec:StrHA}

In this section we discuss DCi on a Hopf algebra $H$, which we think of as the structure Hopf algebra of a QPB $B=A^{\mathrm{co}H}\subseteq A$. However, the results in this subsection are independent of the underlying base and total space algebra and can thus be understood solely considering a Hopf algebra $H$. We are mainly interested in extensions of the comultiplication $\Delta\colon H\to H\otimes H$ to differential forms, either as a morphism of graded algebras or even as a morphism of \textit{differential} graded algebras. It turns out that only the former exists canonically on a bicovariant DC, while the latter exists at least on the maximal prolongation of a bicovariant FODC. 

We consider a Hopf algebra $H$.
\begin{proposition}[{\cite{SchDC}}]\label{prop:GradHopf}
	Let $\Omega^\bullet(H)$ be a bicovariant DC on $H$. The following maps of degree $0$ structure $\Omega^\bullet(H)$ as a graded Hopf algebra. For homogeneous elements $\omega\in\Omega^k(H)$ with $k>0$ we set
	\begin{align}
		\Delta^\bullet_\mathrm{bc}\colon\Omega^\bullet(H)\to\Omega^\bullet(H)\otimes\Omega^\bullet(H),\qquad
		&\Delta^\bullet_\mathrm{bc}(\omega):=\omega_0\otimes\omega_1+\omega_{-1}\otimes\omega_0,\\
		\varepsilon^\bullet_\mathrm{bc}\colon\Omega^\bullet(H)\to\Bbbk,\qquad
		&\varepsilon^\bullet_\mathrm{bc}(\omega):=0,\\
		S^\bullet_\mathrm{bc}\colon\Omega^\bullet(H)\to\Omega^\bullet(H),\qquad
		&S^\bullet_\mathrm{bc}(\omega):=-S(\omega_{-1})\cdot\omega_0\cdot S(\omega_1),
	\end{align}
	while on $\Omega^0(H)=H$ we apply the comultiplication, counit and antipode of $H$.
\end{proposition}
The subscript of $\Delta^\bullet_\mathrm{bc}$ indicates this is the canonical extension of the comultiplication for a \textbf{bicovariant} DC on $H$. In what follows we clarify that there might be other extensions of $\Delta$, which justifies the usage of a subscript a posteriori. In the following we are frequently using the short notation
\begin{equation}\label{eq:DeltaBul}
	\Delta^\bullet_\mathrm{bc}(\omega)=\omega_0\otimes\omega_1+\omega_{-1}\otimes\omega_0=:\omega_{(1)}\otimes\omega_{(2)}
\end{equation}
for all homogeneous $\omega\in\Omega^k(H)$ with $k>0$. In this notation the coassociativity, counitality and antipode axiom of the graded Hopf algebra $\Omega^\bullet(H)$ read
\begin{align}
	(\omega_{(1)})_{(1)}\otimes(\omega_{(1)})_{(2)}\otimes\omega_{(2)}
	&=\omega_{(1)}\otimes(\omega_{(2)})_{(1)}\otimes(\omega_{(2)})_{(2)}
	=:\omega_{(1)}\otimes\omega_{(2)}\otimes\omega_{(3)},\\
	\varepsilon^\bullet_\mathrm{bc}(\omega_{(1)})\omega_{(2)}&=\omega=\omega_{(1)}\varepsilon^\bullet_\mathrm{bc}(\omega_{(2)}),\\
	S^\bullet_\mathrm{bc}(\omega_{(1)})\omega_{(2)}&=\varepsilon^\bullet(\omega)1=0=\omega_{(1)}S^\bullet_\mathrm{bc}(\omega_{(2)})
\end{align}
for all homogeneous $\omega\in\Omega^k(H)$ with $k>0$.

While \eqref{eq:DeltaBul} is a canonical extension $\Delta^\bullet_\mathrm{bc}\colon\Omega^\bullet(H)\to\Omega^\bullet(H)\otimes\Omega^\bullet(H)$ of $\Delta\colon H\to H\otimes H$ to a bicovariant DC $\Omega^\bullet(H)$ as a morphism of \textit{graded algebras} it is in general \textit{not unique} and in general \textit{not a morphism of DGAs}. To understand this issue better we first discuss the following result, see also \cite[Corollary 5.9]{SchDC}.

\begin{lemma}\label{lem:BicExt}
A FODC $(\Gamma,\mathrm{d})$ on $H$ is bicovariant if and only if $\Delta\colon H\to H\otimes H$ extends to a morphism
\begin{equation}
	\Delta^1\colon\Gamma\to\Omega^1(H\otimes H)=\Gamma\otimes H\oplus H\otimes\Gamma
\end{equation}
of FODCi. In this case $\Delta^1=\Delta_\Gamma+{}_\Gamma\Delta$ is the sum of the right and left $H$-coaction on $\Gamma$.
\end{lemma}
\begin{proof}
$"\Rightarrow"$: If $(\Gamma,\mathrm{d})$ is bicovariant there are the coactions $\Delta_\Gamma\colon\Gamma\to\Gamma\otimes H$ and ${}_\Gamma\Delta\colon\Gamma\to H\otimes\Gamma$. We define $\Delta^1:=\Delta_\Gamma+{}_\Gamma\Delta$. It is straight forward to verify that $(\Delta,\Delta^1)$ is a morphism of FODCi. In particular, for all $h\in H$
$$
\Delta^1(\mathrm{d}(h))
=\Delta_\Gamma(\mathrm{d}(h))+{}_\Gamma\Delta(\mathrm{d}(h))
=\mathrm{d}(h_1)\otimes h_2+h_1\otimes\mathrm{d}(h_2)
=\mathrm{d}_\otimes(\Delta(h))
$$
holds, since the differential is $H$-bicolinear. 
	
$"\Leftarrow"$: Let $(\Delta,\Delta^1)$ be a morphism of FODCi with $\Delta^1\colon\Gamma\to\Gamma\otimes H\oplus H\otimes\Gamma$. Composing projections with $\Delta^1$ we obtain two maps $\Delta_\Gamma\colon\Gamma\to\Gamma\otimes H$ and ${}_\Gamma\Delta\colon\Gamma\to H\otimes\Gamma$. Since $(\Delta,\Delta^1)$ is a morphism of FODCi we obtain 
\begin{align*}
	\Delta_\Gamma(h\mathrm{d}(h'))
	&=\mathrm{pr}_{\Gamma\otimes H}(\Delta^1(h\mathrm{d}(h')))\\
	&=\mathrm{pr}_{\Gamma\otimes H}(\Delta(h)\Delta^1(\mathrm{d}(h')))\\
	&=\mathrm{pr}_{\Gamma\otimes H}(\Delta(h)\mathrm{d}_\otimes(\Delta(h')))\\
	&=\Delta(h)(\mathrm{d}\otimes\mathrm{id})(\Delta(h'))\\
	&=h_1\mathrm{d}(h'_1)\otimes h_2h'_2
\end{align*}
for all $h,h'\in H$. By the surjectivity of $(\Gamma,\mathrm{d})$ this implies that $\Delta_\Gamma$ is a right $H$-coaction on $\Gamma$, such that $\mathrm{d}\colon H\to\Gamma$ is right $H$-colinear. Similarly, one proves that ${}_\Gamma\Delta(h\mathrm{d}(h'))=h_1h'_1\otimes h_2\mathrm{d}(h'_2)$, which implies that ${}_\Gamma\Delta$ is a left $H$-coaction on $\Gamma$, such that $\mathrm{d}$ is left $H$-colinear. Thus, $(\Gamma,\mathrm{d})$ is bicovariant.
\end{proof}
Note that the generalization of Lemma \ref{lem:BicExt} to higher order forms is \textit{not} correct. Namely, given a DC $\Omega^\bullet(H)$ on $H$, the two statements
\begin{enumerate}
\item[i.)] $\Omega^\bullet(H)$ is a bicovariant DC on $H$.

\item[ii.)] $\Delta\colon H\to H\otimes H$ extends to a morphism $\Delta^\bullet\colon\Omega^\bullet(H)\to\Omega^\bullet(H\otimes H)=\Omega^\bullet(H)\otimes\Omega^\bullet(H)$ of DGAs.
\end{enumerate}
are \textit{not} equivalent. It follows that $ii.)$ implies $i.)$, since for all $k>0$ with projection $\pi^k\colon\Omega^\bullet(H)\to\Omega^k(H)$ we define
\begin{equation}\label{coactions}
\begin{split}
	\Delta_{\Omega^k(H)}&:=(\pi^k\otimes\mathrm{id})\circ\Delta^\bullet|_{\Omega^k(H)}\colon\Omega^k(H)\to\Omega^k(H)\otimes H,\\
	{}_{\Omega^k(H)}\Delta&:=(\mathrm{id}\otimes\pi^k)\circ\Delta^\bullet|_{\Omega^k(H)}\colon\Omega^k(H)\to H\otimes\Omega^k(H),
\end{split}
\end{equation}
which are easily shown to be right and left $H$-coactions on $\Omega^k(H)$, intertwining the differential, such that $\Omega^\bullet(H)$ is bicovariant. While on the other hand, given a bicovariant DC $\Omega^\bullet(H)$ on $H$, we can obtain a graded algebra morphism $\Delta^\bullet_\mathrm{bc}\colon\Omega^\bullet(H)\to\Omega^\bullet(H)\otimes\Omega^\bullet(H)$ according to Proposition \ref{prop:GradHopf}, however, $\Delta^\bullet_\mathrm{bc}$ might not be compatible with the differential, since for a $1$-form $\omega\in\Omega^1(H)$ the expression $\Delta^\bullet_\mathrm{bc}(\mathrm{d}(\omega))=\mathrm{d}(\omega_0)\otimes\omega_1+\omega_{-1}\otimes\mathrm{d}(\omega_0)$ might differ from
\begin{equation}\label{eq:counterex}
	\mathrm{d}_\otimes(\Delta^\bullet_\mathrm{bc}(\omega))
	=\mathrm{d}_\otimes(\omega_0\otimes\omega_1+\omega_{-1}\otimes\omega_0)
	=\mathrm{d}(\omega_0)\otimes\omega_1
	-\omega_0\otimes\mathrm{d}(\omega_1)
	+\mathrm{d}(\omega_{-1})\otimes\omega_0
	+\omega_{-1}\otimes\mathrm{d}(\omega_0)
\end{equation}
in general.
We want to point out that in \cite[page 6]{DurII} it is claimed that $\Delta^\bullet_\mathrm{bc}$ is automatically a morphism of DGAs if $\Omega^\bullet(H)$ is the maximal prolongation of a bicovariant FODC. We do not agree with this statement, considering the following counterexample.
\begin{example}\label{counterex1}
Let $H=\mathbb{C}[t,t^{-1},s,s^{-1}]$ be the Hopf algebra of rational polynomials in two variables $t$ and $s$, with $\Delta(t)=t\otimes t$, $\Delta(s)=s\otimes s$ and $S(t)=t^{-1}$, $S(s)=s^{-1}$. Consider the classical calculus on $H$, namely $\Omega^1(H)=\mathrm{span_H\{\mathrm{d}t,\mathrm{d}s\}}$, $\Omega^2(H)=\mathrm{span_H\{\mathrm{d}t\wedge\mathrm{d}s\}}$ and $\Omega^k(H)=\{0\}$ $k>2$. Then $\Delta^\bullet_\mathrm{bc}$ is \textbf{not} a morphism of DGAs, eventhough $\Omega^\bullet(H)$ is the maximal prolongation of the bicovariant FODC $\Omega^1(H)$.
This can be easily checked by inserting $\omega=t\mathrm{d}s$ in \eqref{eq:counterex} and comparing this with $\Delta^\bullet_\mathrm{bc}(\mathrm{d}(\omega))=\mathrm{d}(\omega_0)\otimes\omega_1+\omega_{-1}\otimes\mathrm{d}(\omega_0)$.
\end{example}
In the rest of this paper we follow the approach \cite{DurII} of \DJ ur\dj evi\'c, but we clarify in which situations it is sufficient to work with the canonical extension $\Delta^\bullet_\mathrm{bc}$ and when one has to assume the existence of a (in this case necessarily unique) extension of $\Delta$ to a morphism $\Delta^\bullet$ of DGAs.
In the following we always assume that the DC $\Omega^\bullet(H)$ on $H$ is bicovariant, in which case $\Delta^\bullet_\mathrm{bc}$ refers to \eqref{eq:DeltaBul} with short notation $\Delta^\bullet_\mathrm{bc}(\omega)=\omega_{(1)}\otimes\omega_{(2)}$. If we even assume that $\Delta\colon H\to H\otimes H$ extends to a morphism of DGAs $\Omega^\bullet(H)\to\Omega^\bullet(H)\otimes\Omega^\bullet(H)$ we mention this explicitly.
\begin{definition}\label{def:completeH}
	A DC $\Omega^\bullet(H)$ on a Hopf algebra $H$ is called \textbf{complete} if $\Delta\colon H\to H\otimes H$ extends to a (necessarily unique) morphism $\Delta^\bullet\colon\Omega^\bullet(H)\to\Omega^\bullet(H)\otimes\Omega^\bullet(H)$ of DGAs.
\end{definition}
For a complete calculus $\Omega^\bullet(H)$ on $H$ we 
use the short notation 
\begin{equation}
	\Delta^\bullet(\omega)=:\omega_{[1]}\otimes\omega_{[2]}
\end{equation}
for $\omega\in\Omega^k(H)$, $k>0$, to distinguish it from the canonical bicovariant extension $\Delta^\bullet_\mathrm{bc}(\omega)=\omega_{(1)}\otimes\omega_{(2)}$ defined in \eqref{eq:DeltaBul}.

Let us note that in the framework of Example \ref{counterex1} $\Delta\colon H\to H\otimes H$ (uniquely) extends to a morphism of DGAs $\Delta^\bullet\colon\Omega^\bullet(H)\to\Omega^\bullet(H)\otimes\Omega^\bullet(H)$ if we set $\Delta^\bullet(\mathrm{d}t\wedge\mathrm{d}s)=\mathrm{d}t\wedge\mathrm{d}s\otimes ts+\mathrm{d}(t)s\otimes t\mathrm{d}(s)-t\mathrm{d}(s)\otimes\mathrm{d}(t)s+ts\otimes\mathrm{d}t\wedge\mathrm{d}s$.
This turns out to be a general feature of the maximal prolongation of a bicovariant calculus, as proven in the next lemma (essentially following \cite[Lemma 4.29]{BegMaj}).
\begin{lemma}\label{lem:Hcomplete}
The maximal prolongation $\Omega^\bullet(H)$ of a bicovariant FODC $\Gamma$ on $H$ is complete.
\end{lemma}
\begin{proof}
According to Proposition \ref{prop:maxprol} the maximal prolongation is the tensor algebra where we quotient the ideal generated by $\mathrm{d}h^i\otimes_H\mathrm{d}g^i$ with $h^i,g^i\in H$ if $h^i\mathrm{d}g^i=0$. Since the previous ideal is generated in degree $2$ it suffices to prove that $\Delta\colon H\to H\otimes H$ is $2$-times differentiable. By Lemma \ref{lem:BicExt} $\Delta^1=\Delta_{\Gamma}+{}_\Gamma\Delta\colon\Gamma\to\Gamma\otimes H\oplus H\otimes\Gamma$ is the well-defined extension to $\Omega^1(H)=\Gamma$. It remains to prove that $\Delta^2(h\mathrm{d}g\wedge\mathrm{d}k)
=\Delta(h)\mathrm{d}_\otimes(\Delta(g))\wedge_\otimes\mathrm{d}(\Delta(k))$ is well-defined. This is the case if for $h^i,g^i\in H$ such that $h^i\mathrm{d}g^i=0$ it follows that
\begin{equation}\label{eq:lemhelp}
\begin{split}
	\Delta^2(\mathrm{d}h^i\wedge\mathrm{d}g^i)
	=\mathrm{d}(h^i_1)\wedge\mathrm{d}(g^i_1)\otimes h^i_2g^i_2
	+h^i_1g^i_1\otimes\mathrm{d}(h^i_2)\wedge\mathrm{d}(g^i_2)
	+\mathrm{d}(h^i_1)g^i_1\otimes h^i_2\mathrm{d}(g^i_2)
	-h^i_1\mathrm{d}(g^i_1)\otimes \mathrm{d}(h^i_2)g^i_2
\end{split}
\end{equation}
vanishes.
Let $h^i,g^i\in H$ such that $h^i\mathrm{d}g^i=0$. Then $0=\Delta^1(h^i\mathrm{d}g^i)=h^i_1\mathrm{d}(g^i_1)\otimes h^i_2g^i_2+h^i_1g^i_1\otimes h^i_2\mathrm{d}(g^i_2)$, which by the direct sum decomposition implies $h^i_1\mathrm{d}(g^i_1)\otimes h^i_2g^i_2=0$ and $h^i_1g^i_1\otimes h^i_2\mathrm{d}(g^i_2)=0$. Applying $\mathrm{d}\otimes\mathrm{id}$ to the first equation and $\mathrm{id}\otimes\mathrm{d}$ to the second equation implies that the first and second term on the right hand side of \eqref{eq:lemhelp} vanish, respectively. Applying $\mathrm{id}\otimes\mathrm{d}$ to the first and $\mathrm{d}\otimes\mathrm{id}$ to the second equation and then subtracting the result of the first from the result of the second equation implies that the third and fourth term on the right hand side of \eqref{eq:lemhelp} vanish, as well. This completes the proof of the lemma.
\end{proof}
Later on we often employ the maximal prolongation of a bicovariant FODC on $H$, which is then complete by the previous lemma. Some results however hold in generality for arbitrary complete calculi on $H$.
\begin{proposition}\label{prop:gradedHopf}
Let $\Omega^\bullet(H)$ be a complete DC on a Hopf algebra $H$.
Then $\Omega^\bullet(H)$ is a bicovariant DC and a graded Hopf algebra with graded comultiplication and counit
\begin{align}
\Delta^\bullet\colon\Omega^\bullet(H)\to\Omega^\bullet(H)\otimes\Omega^\bullet(H),\qquad
&\Delta^\bullet(\omega)=\omega_{[1]}\otimes\omega_{[2]}\\
\varepsilon^\bullet\colon\Omega^\bullet(H)\to\Bbbk,\qquad
&\varepsilon^\bullet(\omega)=0
\end{align}
for all $\omega\in\Omega^\bullet(H)$ with $|\omega|>0$.
The antipode $S^\bullet\colon\Omega^\bullet(H)\to\Omega^\bullet(H)$ is determined by
\begin{equation}\label{gradedS}
	S^\bullet(h^0\mathrm{d}(h^1)\wedge\ldots\wedge\mathrm{d}(h^k))
	=\mathrm{d}(S(h^k))\wedge\ldots\wedge\mathrm{d}(S(h^1))S(h^0)
\end{equation}
for all $h^0\mathrm{d}(h^1)\wedge\ldots\wedge\mathrm{d}(h^k)\in\Omega^k(H)$.
\end{proposition}
\begin{proof}
The first statement is clear, using the $H$-coactions defined in \eqref{coactions}. For the second statement we follow \cite[Lemma 5.4]{SchDC}. By assumption $\Delta^\bullet\colon\Omega^\bullet(H)\to\Omega^\bullet(H)\otimes\Omega^\bullet(H)$ is a morphism of DGAs. In particular, $\Delta^\bullet$ is a graded algebra morphism and so is $\varepsilon^\bullet$. It remains to prove coassociativity
\begin{equation}\label{eq:gradedCoAs}
(\Delta^\bullet\otimes\mathrm{id})\circ\Delta^\bullet=(\mathrm{id}\otimes\Delta^\bullet)\circ\Delta^\bullet,
\end{equation}
counitality
\begin{equation}\label{eq:gradedCoUni}
	(\varepsilon^\bullet\otimes\mathrm{id})\circ\Delta^\bullet=\mathrm{id}=(\mathrm{id}\otimes\varepsilon^\bullet)\circ\Delta^\bullet
\end{equation}
and that \eqref{gradedS} is a well-defined antipode. For coassociativity we observe that both sides of \eqref{eq:gradedCoAs} are morphisms of DGAs. In particular, they are uniquely determined by their zeroth order, which are $(\Delta\otimes\mathrm{id})\circ\Delta$ and $(\mathrm{id}\otimes\Delta)\circ\Delta$, respectively. Since the latter coincide by the coassociativity of $\Delta$, the DGA morphisms in \eqref{eq:gradedCoAs} have to coincide, as well. Note that $\varepsilon^\bullet\colon\Omega^\bullet(H)\to\Bbbk$ is a morphism of DGAs, as well. Since the DGA $\Bbbk$ is concentrated in degree zero $\varepsilon^\bullet$ has to vanish on higher degrees. All maps in \eqref{eq:gradedCoUni} are morphisms of DGAs, as such determined by their zeroth order, which implies that \eqref{eq:gradedCoUni} is an equality. For the antipode we refer to \cite[Corollary 6.3]{SchDC}.
\end{proof}
Note that up to degree $1$ the morphisms $\Delta^\bullet$ and $\Delta^\bullet_\mathrm{bc}$ coincide. The same is true for the antipode, since
\begin{equation}
	S^1_\mathrm{bc}(h\mathrm{d}h')
	=-S(h_1h'_1)h_2\mathrm{d}(h'_2)S(h_3h'_3)
	=-S(h'_1)\mathrm{d}(h'_2)S(h'_3)S(h)
	=\mathrm{d}(S(h'))S(h)
	=S^1(h\mathrm{d}h')
\end{equation}
for all $h,h'\in H$.

\subsection{Coinvariant forms and vertical forms}\label{sec:VF}

Fix a quantum principal bundle (QPB) $A$ with coinvariant subalgebra $B:=A^{\mathrm{co}H}$. Recall from Definition \ref{def:QPB} that this means $B\subseteq A$ is a faithfully flat Hopf--Galois extension. We further fix a bicovariant DC $\Omega^\bullet(H)$ on $H$ and denote the \textbf{left coinvariant forms} by 
\begin{equation}
	\Lambda^\bullet:={}^{\mathrm{co}H}\Omega^\bullet(H)
	=\bigoplus_{k\geq 0}\{\omega\in\Omega^k(H)~|~{}_{\Omega^k(H)}\Delta(\omega)=1\otimes\omega\},
\end{equation}
where ${}_{\Omega^k(H)}\Delta\colon\Omega^k(H)\to H\otimes\Omega^k(H)$ denotes the left $H$-coaction on $\Omega^k(H)$. Clearly, $\Lambda^\bullet\subseteq\Omega^\bullet(H)$ is a differential graded subalgebra and $\Lambda^0\cong\Bbbk$.
\begin{lemma}\label{lem:YD}
For any bicovariant DC $\Omega^\bullet(H)$ on $H$ it follows that
\begin{enumerate}
\item[i.)] with respect to the
right $H$-action $\leftharpoonup\colon\Lambda^\bullet\otimes H\to\Lambda^\bullet$ defined by $\vartheta\leftharpoonup h:=S(h_1)\vartheta h_2$ the graded algebra $\Lambda^\bullet$ is a graded right $H$-module algebra, i.e.
\begin{equation}
	(\vartheta\wedge\vartheta')\leftharpoonup h=(\vartheta\leftharpoonup h_1)\wedge(\vartheta'\leftharpoonup h_2)\qquad\text{and}\qquad
	1\leftharpoonup h=\varepsilon(h)1
\end{equation}
for all $\vartheta,\vartheta'\in\Lambda^\bullet$ and $h\in H$.
	
\item[ii.)] $\Lambda^\bullet$ is a graded right-right Yetter-Drinfel'd module with
right $H$-action $\leftharpoonup\colon\Lambda^\bullet\otimes H\to\Lambda^\bullet$ and right $H$-coaction defined as the restriction $\Delta^\mathrm{bc}_{\Lambda^\bullet}:=\Delta^\bullet_\mathrm{bc}|_{\Lambda^\bullet}\colon\Lambda^\bullet\to\Lambda^\bullet\otimes\Omega^\bullet(H)$ of $\Delta^\bullet_\mathrm{bc}\colon\Omega^\bullet(H)\to\Omega^\bullet(H)\otimes\Omega^\bullet(H)$. Explicitly, the compatibility condition reads
\begin{equation}\label{eq:YDbc}
(\vartheta\leftharpoonup h)_{(1)}\otimes(\vartheta\leftharpoonup h)_{(2)}
=\vartheta_{(1)}\leftharpoonup h_2\otimes S(h_1)\vartheta_{(2)}h_3
\end{equation}
for all $\vartheta\in\Lambda^\bullet$ and $h\in H$.

\item[iii.)] there is a surjective morphism
\begin{equation}
	\pi_{\Lambda^\bullet}\colon\Omega^\bullet(H)\twoheadrightarrow\Lambda^\bullet,\qquad
	\omega\mapsto S(\omega_{-1})\omega_0
\end{equation}
of graded vector spaces.
\end{enumerate}
\end{lemma}
\begin{proof}
\begin{enumerate}
\item[i.)] The map $\leftharpoonup$ is clearly a right $H$-action which preserves $\Lambda^\bullet$. Then the statement follows by the associativity of the wedge product $\wedge$.

\item[ii.)] Note that $\Delta_\mathrm{bc}^\bullet(\Lambda^\bullet)\subseteq\Lambda^\bullet\otimes\Omega^\bullet(H)$ since for $\vartheta\in\Lambda^k$ we obtain $\Delta_\mathrm{bc}^k(\vartheta)=\vartheta_0\otimes\vartheta_1+1\otimes\vartheta$ and while $1\in\Lambda^0$ we also have $\vartheta_0\otimes\vartheta_1\in\Lambda^k\otimes H$ by the commutativity of $\Delta_{\Omega^\bullet(H)}$ and ${}_{\Omega^\bullet(H)}\Delta$.
In order to verify the Yetter-Drinfel'd property \eqref{eq:YDbc} we observe that
\begin{align*}
(\vartheta\leftharpoonup h)_{(1)}\otimes(\vartheta\leftharpoonup h)_{(2)}
&=(S(h_1)\vartheta h_2)_{(1)}\otimes(S(h_1)\vartheta h_2)_{(2)}\\
&=\Delta_{\Omega^\bullet(H)}(S(h_1)\vartheta h_2)+{}_{\Omega^\bullet(H)}\Delta(S(h_1)\vartheta h_2)\\
&=S(h_2)\vartheta_0h_3\otimes S(h_1)\vartheta_1h_4+1\otimes S(h_1)\vartheta h_2
\end{align*}
coincides with
\begin{align*}
\vartheta_{(1)}\leftharpoonup h_2\otimes S(h_1)\vartheta_{(2)}h_3
&=\vartheta_0\leftharpoonup h_2\otimes S(h_1)\vartheta_1h_3
+1\leftharpoonup h_2\otimes S(h_1)\vartheta h_3\\
&=S(h_2)\vartheta_0h_3\otimes S(h_1)\vartheta_1h_4
+1\otimes S(h_1)\vartheta h_2
\end{align*}
for all $h\in H$ and $\vartheta\in\Lambda^\bullet$.

\item[iii.)] One easily verifies that $\pi_{\Lambda^\bullet}$ is a degree preserving section of the inclusion $\Lambda^\bullet\subseteq\Omega^\bullet(H)$.
\end{enumerate}
\end{proof}
Recall from the previous section that $\Delta^\bullet_\mathrm{bc}$ is an extension of the comultiplication, but not as a DGA morphism in general. In case the DGA morphism extension $\Delta^\bullet\colon\Omega^\bullet(H)\to\Omega^\bullet(H)\otimes\Omega^\bullet(H)$, $\Delta^\bullet(\omega)=\omega_{[1]}\otimes\omega_{[2]}$ exists we call $\Omega^\bullet(H)$ complete, according to Definition \ref{def:completeH}. Using $\Delta^\bullet$ instead of $\Delta^\bullet_\mathrm{bc}$ we prove an analogue of Proposition \ref{lem:YD}. We further provide an identification of $\Omega^\bullet(H)$ and $H\otimes\Lambda^\bullet$ as DCi via the fundamental theorem of Hopf modules. In particular, the following content extends (and clarifies) the considerations in \cite{DurII}.
\begin{proposition}\label{prop:gradedYD}
Let $\Omega^\bullet(H)$ be a complete DC on $H$. Then
\begin{enumerate}
\item[i.)] $\Lambda^\bullet$ is a graded right-right Yetter-Drinfel'd module with
right $H$-action $\leftharpoonup\colon\Lambda^\bullet\otimes H\to\Lambda^\bullet$ defined by $\vartheta\leftharpoonup h:=S(h_1)\vartheta h_2$
and right $H$-coaction defined as the restriction $\Delta_{\Lambda^\bullet}:=\Delta^\bullet|_{\Lambda^\bullet}\colon\Lambda^\bullet\to\Lambda^\bullet\otimes\Omega^\bullet(H)$ of $\Delta^\bullet\colon\Omega^\bullet(H)\to\Omega^\bullet(H)\otimes\Omega^\bullet(H)$. Explicitly, the compatibility condition reads
\begin{equation}\label{eq:YD}
(\vartheta\leftharpoonup h)_{[1]}\otimes(\vartheta\leftharpoonup h)_{[2]}
=\vartheta_{[1]}\leftharpoonup h_2\otimes S(h_1)\vartheta_{[2]}h_3
\end{equation}
for all $\vartheta\in\Lambda^\bullet$ and $h\in H$.

\item[ii.)] there is an isomorphism
\begin{equation}\label{iso:OmegaH}
\Xi\colon\Omega^\bullet(H)\xrightarrow{\cong}H\otimes\Lambda^\bullet,\qquad
\omega\mapsto\omega_{-2}\otimes S(\omega_{-1})\omega_0
\end{equation}
of bicovariant DCi on $H$ with inverse $H\otimes\Lambda^\bullet\to\Omega^\bullet(H)$, $h\otimes\vartheta\mapsto h\vartheta$. The multiplication and differential on $H\otimes\Lambda^\bullet$ is defined by
\begin{equation}\label{eq:FundWedge}
\begin{split}
	(h\otimes\vartheta)\wedge(h'\otimes\vartheta')&:=hh'_1\otimes(\vartheta\leftharpoonup h'_2)\wedge\vartheta',\\
	\mathrm{d}(h\otimes\vartheta)&:=h\otimes\mathrm{d}\vartheta+h_1\otimes\varpi(\pi_\varepsilon(h_2))\wedge\vartheta,
\end{split}
\end{equation}
respectively, where $h\otimes\vartheta,h'\otimes\vartheta'\in H\otimes\Lambda^\bullet$.
\end{enumerate}
\end{proposition}
\begin{proof}
\begin{enumerate}
\item[i.)] As in the proof of Lemma \ref{lem:YD} we have that $\leftharpoonup$ is a right $H$-action on $\Lambda^\bullet$.
From Proposition \ref{prop:gradedHopf} we know that $\Omega^\bullet(H)$ is a graded Hopf algebra. We show that $\Delta^\bullet(\Lambda^\bullet)\subseteq\Lambda^\bullet\otimes\Omega^\bullet(H)$. Let $\vartheta\in\Lambda^k$ and recall that $(\pi^0\otimes\mathrm{id})\circ\Delta^\bullet={}_{\Omega^k(H)}\Delta$ is the left $H$-coaction on $\Omega^k(H)$, where $\pi^0\colon\Omega^\bullet(H)\to H$ is the natural projection. Then, using the coassociativity of $\Delta^\bullet$ and that $\vartheta$ is left coinvariant, we obtain
\begin{align*}
({}_{\Omega^k(H)}\Delta\otimes\mathrm{id})\Delta^\bullet(\vartheta)
&=((\pi^0\otimes\mathrm{id}\otimes\mathrm{id})\circ(\Delta^\bullet\otimes\mathrm{id})\circ\Delta^\bullet)(\vartheta)\\
&=((\pi^0\otimes\mathrm{id}\otimes\mathrm{id})\circ(\mathrm{id}\otimes\Delta^\bullet)\circ\Delta^\bullet)(\vartheta)\\
&=((\mathrm{id}\otimes\Delta^\bullet)\circ(\pi^0\otimes\mathrm{id})\circ\Delta^\bullet)(\vartheta)\\
&=((\mathrm{id}\otimes\Delta^\bullet)\circ{}_{\Omega^k(H)}\Delta)(\vartheta)\\
&=1\otimes\Delta^\bullet(\vartheta).
\end{align*}
Thus, the restriction $\Delta_{\Lambda^\bullet}:=\Delta^\bullet|_{\Lambda^\bullet}\colon\Lambda^\bullet\to\Lambda^\bullet\otimes\Omega^\bullet(H)$ of $\Delta^\bullet\colon\Omega^\bullet(H)\to\Omega^\bullet(H)\otimes\Omega^\bullet(H)$ gives is a right $\Omega^\bullet(H)$-comodule structure on $\Lambda^\bullet$.
For the Yetter-Drinfel'd property \eqref{eq:YD} we use that $\Delta^\bullet$ is an algebra morphism to conclude
\begin{align*}
(\vartheta\leftharpoonup h)_{[1]}\otimes(\vartheta\leftharpoonup h)_{[2]}
&=(S(h_1)\vartheta h_2)_{[1]}\otimes(S(h_1)\vartheta h_2)_{[2]}\\
&=S(h_2)\vartheta_{[1]}h_3\otimes S(h_1)\vartheta_{[2]}h_4\\
&=\vartheta_{[1]}\leftharpoonup h_2\otimes S(h_1)\vartheta_{[2]}h_3
\end{align*}
for all $\vartheta\in\Lambda^\bullet$ and $h\in H$.

\item[ii.)] According to the fundamental theorem of Hopf modules (in the form of Lemma \ref{FundThm} ii.)) the map $\Xi$ defined in \eqref{iso:OmegaH} is an isomorphism of bicovariant $H$-bimodules and it is clearly degree preserving. We show that the wedge product and differential of $\Omega^\bullet(H)$ are transformed into \eqref{eq:FundWedge} under the isomorphism $\Xi$, which then concludes the proposition. In fact, for $h,h'\in H$ and $\vartheta,\vartheta'\in\Lambda^\bullet$ we obtain
\begin{align*}
\Xi((h\vartheta)\wedge(h'\vartheta'))
&=(h\vartheta)_{-2}(h'\vartheta')_{-2}\otimes S((h\vartheta)_{-1}(h'\vartheta')_{-1})(h\vartheta)_0\wedge(h'\vartheta')_0\\
&=h_1h'_1\otimes S(h_2h'_2)h_3\vartheta\wedge h'_3\vartheta'\\
&=hh'_1\otimes(\vartheta\leftharpoonup h'_2)\wedge\vartheta'
\end{align*}
and
\begin{align*}
\Xi(\mathrm{d}(h\vartheta))
&=\Xi(\mathrm{d}(h)\wedge\vartheta+h\mathrm{d}(\vartheta))\\
&=h_1\vartheta_{-2}\otimes S(h_2\vartheta_{-1})(\mathrm{d}(h_3)\wedge\vartheta_0
+h_3\mathrm{d}(\vartheta_0))\\
&=h_1\otimes\varpi(\pi_\varepsilon(h_2))\wedge\vartheta+h\otimes\mathrm{d}\vartheta,
\end{align*}
where we recall that $\varpi(g):=S(g_1)\mathrm{d}(g_2)$ on $g\in H^+$ and that $\vartheta$ is left $H$-covariant.
\end{enumerate}
\end{proof}
Using the identification $\Xi\colon\Omega^\bullet(H)\xrightarrow{\cong} H\otimes\Lambda^\bullet$ we show that the right $H$-coaction on $H\otimes\Lambda^\bullet$ extends to a morphism of DGAs if $\Omega^\bullet(H)$ is complete. 
\begin{lemma}
Let $\Omega^\bullet(H)$ be a complete DC on $H$. Then the right $H$-coaction 
\begin{equation}
	\Delta_v^0\colon H\otimes\Lambda^0\to (H\otimes\Lambda^0)\otimes H,\qquad
	h\otimes 1\mapsto(h_1\otimes 1)\otimes h_2
\end{equation}
extends to a morphism
\begin{equation}\label{eq:DelVer}
\Delta_v^\bullet\colon H\otimes\Lambda^\bullet\to(H\otimes\Lambda^\bullet)\otimes\Omega^\bullet(H),\qquad
h\otimes\vartheta\mapsto\Xi(h_1\vartheta_{[1]})\otimes h_2\vartheta_{[2]}
\end{equation}
of DGAs.
\end{lemma}
\begin{proof}
This follows from the commutativity of the diagram
\begin{equation}
\begin{tikzcd}
H\otimes\Lambda^\bullet \arrow{rr}{\Delta^\bullet_v} \arrow{d}[swap]{\Xi^{-1}}
& & (H\otimes\Lambda^\bullet)\otimes\Omega^\bullet(H)\\
\Omega^\bullet(H) \arrow{rr}{\Delta^\bullet}
& & \Omega^\bullet(H)\otimes\Omega^\bullet(H) \arrow{u}{\Xi\otimes\mathrm{id}}
\end{tikzcd},
\end{equation}
together with the assumption that $\Delta^\bullet$ is a morphism of DGAs and the fact that $\Xi$ is an isomorphism of DCi, as shown in Proposition \ref{prop:gradedYD} ii.).
\end{proof}
The idea is now to take the model $H\otimes\Lambda^\bullet$ and extend it to $A\otimes\Lambda^\bullet$ for a total space algebra $A$. Note that in this more general case there is no fundamental theorem of Hopf modules to induce a DC structure on $A\otimes\Lambda^\bullet$ from a certain "total space calculus" on $A$. Nevertheless, we can mimic the operations \eqref{eq:FundWedge} to obtain a DC on $A\otimes\Lambda^\bullet$. The following proposition is taken from \cite[Lemma 3.1]{DurII}. Below, we provide an elementary algebraic proof of this result.
\begin{proposition}
Let $B=A^{\mathrm{co}H}\subseteq A$ be a QPB and $\Omega^\bullet(H)$ the maximal prolongation of a bicovariant FODC on $H$. Then there is a DC $\mathrm{ver}^\bullet$ on $A$ defined by $\mathrm{ver}^\bullet:=A\otimes\Lambda^\bullet$ with wedge product and differential determined by
\begin{equation}\label{verwedge}
\begin{split}
(a\otimes\vartheta)\wedge(a'\otimes\vartheta')
&:=aa'_0\otimes(\vartheta\leftharpoonup a'_1)\wedge\vartheta',\\
\mathrm{d}_v(a\otimes\vartheta)
&:=a\otimes\mathrm{d}\vartheta+a_0\otimes\varpi(\pi_\varepsilon(a_1))\wedge\vartheta
\end{split}
\end{equation}
for all $a\otimes\vartheta,a'\otimes\vartheta'\in\mathrm{ver}^\bullet$.
We call $\mathrm{ver}^\bullet$ the \textbf{vertical forms} on $A$.
\end{proposition}
\begin{proof}
The operations \eqref{verwedge} close in $\mathrm{ver}^\bullet=A\otimes\Lambda^\bullet$, since $\Lambda^\bullet$ is a right $H$-module algebra by Proposition \ref{prop:gradedYD} i.) and the image of the Cartan--Maurer form is in $\Lambda^1$. 
Given $a,a',a''\in A$ and $\vartheta,\vartheta',\vartheta''\in\Lambda^\bullet$, the associativity of the wedge product follows from
\begin{align*}
\big((a\otimes\vartheta)\wedge(a'\otimes\vartheta')\big)\wedge(a''\otimes\vartheta'')
&=(aa'_0\otimes(\vartheta\leftharpoonup a'_1)\wedge\vartheta')\wedge(a''\otimes\vartheta'')\\
&=aa'_0a''_0\otimes((\vartheta\leftharpoonup a'_1)\wedge\vartheta')\leftharpoonup a''_1)\wedge\vartheta''\\
&=aa'_0a''_0\otimes(\vartheta\leftharpoonup(a'_1a''_1))\wedge(\vartheta'\leftharpoonup a''_2)\wedge\vartheta''\\
&=(a\otimes\vartheta)\wedge(a'a''_0\otimes(\vartheta'\leftharpoonup a''_1)\wedge\vartheta'')\\
&=(a\otimes\vartheta)\wedge\big((a'\otimes\vartheta')\wedge(a''\otimes\vartheta'')\big),
\end{align*}
where we used that $\Lambda^\bullet$ is a right $H$-module algebra. The map $\mathrm{d}_v$ is of degree $1$ and squares to zero, since
\begin{align*}
\mathrm{d}_v(\mathrm{d}_v(a\otimes\vartheta))
&=\mathrm{d}_v(a\otimes\mathrm{d}\vartheta+a_0\otimes\varpi(\pi_\varepsilon(a_1))\wedge\vartheta)\\
&=a\otimes\mathrm{d}^2\vartheta
+a_0\otimes\varpi(\pi_\varepsilon(a_1))\wedge\mathrm{d}\vartheta
+a_0\otimes\mathrm{d}(\varpi(\pi_\varepsilon(a_1))\wedge\vartheta)
+a_0\otimes\varpi(\pi_\varepsilon(a_1))\wedge\varpi(\pi_\varepsilon(a_2))\wedge\vartheta\\
&=a_0\otimes\varpi(\pi_\varepsilon(a_1))\wedge\mathrm{d}\vartheta
+a_0\otimes\mathrm{d}(\varpi(\pi_\varepsilon(a_1)))\wedge\vartheta
-a_0\otimes\varpi(\pi_\varepsilon(a_1))\wedge\mathrm{d}\vartheta\\
&\quad+a_0\otimes\varpi(\pi_\varepsilon(a_1))\wedge\varpi(\pi_\varepsilon(a_2))\wedge\vartheta\\
&=0,
\end{align*}
using the Cartan--Maurer equation \eqref{eq:CartanMaurer}. Moreover, the graded Leibniz rule holds since
\begin{align*}
\mathrm{d}_v(a\otimes\vartheta)&\wedge(a'\otimes\vartheta')
+(-1)^{|\vartheta|}(a\otimes\vartheta)\wedge\mathrm{d}_v(a'\otimes\vartheta')\\
&=(a\otimes\mathrm{d}\vartheta
+a_0\otimes\varpi(\pi_\varepsilon(a_1))\wedge\vartheta)\wedge(a'\otimes\vartheta')
+(-1)^{|\vartheta|}(a\otimes\vartheta)\wedge(a'\otimes\mathrm{d}\vartheta'
+a'_0\otimes\varpi(\pi_\varepsilon(a'_1))\wedge\vartheta')\\
&=aa'_0\otimes((\mathrm{d}\vartheta)\leftharpoonup a'_1)\wedge\vartheta'
+a_0a'_0\otimes((\varpi(\pi_\varepsilon(a_1))\wedge\vartheta)\leftharpoonup a'_1)\wedge\vartheta'\\
&\quad+(-1)^{|\vartheta|}aa'_0\otimes(\vartheta\leftharpoonup a'_1)\wedge\mathrm{d}\vartheta'
+(-1)^{|\vartheta|}aa'_0\otimes(\vartheta\leftharpoonup a'_1)\wedge\otimes\varpi(\pi_\varepsilon(a'_2))\wedge\vartheta'\\
&=aa'_0\otimes((\mathrm{d}\vartheta)\leftharpoonup a'_1)\wedge\vartheta'
+a_0a'_0\otimes\varpi(\pi_\varepsilon(a_1a'_1))\wedge(\vartheta\leftharpoonup a'_2)\wedge\vartheta'
-aa'_0\otimes\varpi(\pi_\varepsilon(a'_1))\wedge(\vartheta\leftharpoonup a'_2)\wedge\vartheta'\\
&\quad+(-1)^{|\vartheta|}aa'_0\otimes(\vartheta\leftharpoonup a'_1)\wedge\mathrm{d}\vartheta'
+(-1)^{|\vartheta|}aa'_0\otimes(\vartheta\leftharpoonup a'_1)\wedge\otimes\varpi(\pi_\varepsilon(a'_2))\wedge\vartheta'
\end{align*}
coincides with
\allowdisplaybreaks
\begin{align*}
\mathrm{d}_v((a\otimes\vartheta)\wedge(a'\otimes\vartheta'))
&=\mathrm{d}_v(aa'_0\otimes(\vartheta\leftharpoonup a'_1)\wedge\vartheta')\\
&=aa'_0\otimes\mathrm{d}((\vartheta\leftharpoonup a'_1)\wedge\vartheta')
+a_0a'_0\otimes\varpi(\pi_\varepsilon(a_1a'_1))\wedge(\vartheta\leftharpoonup a'_2)\wedge\vartheta'\\
&=aa'_0\otimes\mathrm{d}(\vartheta\leftharpoonup a'_1)\wedge\vartheta'
+(-1)^{|\vartheta|}aa'_0\otimes(\vartheta\leftharpoonup a'_1)\wedge\mathrm{d}\vartheta'\\
&\quad+a_0a'_0\otimes\varpi(\pi_\varepsilon(a_1a'_1))\wedge(\vartheta\leftharpoonup a'_2)\wedge\vartheta'\\
&=aa'_0\otimes((\mathrm{d}\vartheta)\leftharpoonup a'_1)\wedge\vartheta'
+aa'_0\otimes\mathrm{d}(S(a'_1))\wedge\vartheta a'_2\wedge\vartheta'\\
&\quad+(-1)^{|\vartheta|}aa'_0\otimes S(a'_1)\wedge\vartheta\wedge\mathrm{d}(a'_2)\wedge\vartheta'\\
&\quad+(-1)^{|\vartheta|}aa'_0\otimes(\vartheta\leftharpoonup a'_1)\wedge\mathrm{d}\vartheta'\\
&\quad+a_0a'_0\otimes\varpi(\pi_\varepsilon(a_1a'_1))\wedge(\vartheta\leftharpoonup a'_2)\wedge\vartheta'\\
&=aa'_0\otimes((\mathrm{d}\vartheta)\leftharpoonup a'_1)\wedge\vartheta'
-aa'_0\otimes S(a'_1)\mathrm{d}(a'_2)\wedge S(a'_3)\vartheta a'_4\wedge\vartheta'\\
&\quad+(-1)^{|\vartheta|}aa'_0\otimes S(a'_1)\wedge\vartheta\wedge a'_2S(a'_3)\mathrm{d}(a'_4)\wedge\vartheta'\\
&\quad+(-1)^{|\vartheta|}aa'_0\otimes(\vartheta\leftharpoonup a'_1)\wedge\mathrm{d}\vartheta'\\
&\quad+a_0a'_0\otimes\varpi(\pi_\varepsilon(a_1a'_1))\wedge(\vartheta\leftharpoonup a'_2)\wedge\vartheta',
\end{align*}
where we used that $\varpi(\pi_\varepsilon(hh'))=\varpi(\pi_\varepsilon(h))\leftharpoonup h'+\varepsilon(h)\varpi(\pi_\varepsilon(h'))$ for all $h,h'\in H$.

It remains to prove that $\mathrm{ver}^\bullet$ is generated by $\mathrm{ver}^0=A\otimes 1$. By Proposition \ref{prop:MaxGammaH} we know that $\Lambda^\bullet$ is generated by $\Lambda^1$. Together with the surjectivity of $\varpi\colon H^+\to\Lambda^1$ this means that it is sufficient to prove that every element $\omega=a\otimes\varpi(h^1)\wedge\ldots\wedge\varpi(h^k)$ for $a\in A$ and $h^1,\ldots,h^k\in H^+\subseteq H$ is generated by $\mathrm{ver}^0=A\otimes 1$. Since $B\subseteq A^{\mathrm{co}H}\subseteq A$ is a QPB, there is a surjection 
$$
\chi'\colon A\otimes A\twoheadrightarrow A\otimes_BA\twoheadrightarrow A\otimes H,\qquad
\chi'(a\otimes a'):=aa'_0\otimes a'_1.
$$
Successive application of $\chi'$ gives a surjection
\begin{align*}
f:=\chi'\circ(\mathrm{id}_A\otimes\chi')\circ\ldots\circ(\mathrm{id}_{A^{\otimes k}}\otimes\chi')\colon A^{\otimes(k+1)}&\twoheadrightarrow A^{\otimes k}\otimes H\twoheadrightarrow\ldots\twoheadrightarrow A\otimes H^{\otimes k}\\
a^0\otimes a^1 \otimes\ldots\otimes a^k&\xmapsto{\mathrm{id}_{A^{\otimes k}}\otimes\chi'} a^0\otimes a^1\otimes\ldots a^{k-1}a^k_0 \otimes a^k_1\\
&\cdots\\
&\xmapsto{\chi'} a^0a^1_0\ldots a^k_0\otimes a^1_1\ldots a^k_1\otimes\ldots\otimes a^{k-1}_{k-1}a^k_{k-1}\otimes a^k_k
\end{align*}
For $\omega=a\otimes\varpi(h^1)\wedge\ldots\wedge\varpi(h^k)$ as before we now consider an element $a^0\otimes a^1\otimes\ldots\otimes a^k\in A^{\otimes (k+1)}$ in the preimage $f^{-1}(a\otimes h^1\otimes\ldots\otimes h^k)$.
For $k=1$ it follows that
$$
(a^0\otimes 1)\mathrm{d}_v(a^1\otimes 1)
=(a^0\otimes 1)\wedge(a^1_0\otimes\varpi(\pi_\varepsilon(a^1_1)))
=a^0a^1_0\otimes\varpi(\pi_\varepsilon(a^1_1))
=a\otimes\varpi(\pi_\varepsilon(h^1))
=a\otimes\varpi(h^1)
=\omega,
$$
while for $k=2$ we obtain
\begin{align*}
	(a^0\otimes 1)\mathrm{d}_v(a^1\otimes 1)\wedge\mathrm{d}_v(a^2\otimes 1)
	&=(a^0\otimes 1)\wedge(a^1_0\otimes\varpi(\pi_\varepsilon(a^1_1)))\wedge(a^2_0\otimes\varpi(\pi_\varepsilon(a^2_1)))\\
	&=(a^0a^1_0\otimes\varpi(\pi_\varepsilon(a^1_1)))\wedge(a^2_0\otimes\varpi(\pi_\varepsilon(a^2_1)))\\
	&=a^0a^1_0a^2_0\otimes\varpi(\pi_\varepsilon(a^1_1a^2_1))\wedge\varpi(\pi_\varepsilon(a^2_2))\\
	&\quad-a^0a^1a^2_0\otimes\varpi(\pi_\varepsilon(a^2_1))\wedge\varpi(\pi_\varepsilon(a^2_2))\\
	&=a\otimes\varpi(\pi_\varepsilon(h^1))\wedge\varpi(\pi_\varepsilon(h^2))
	+a^0a^1a^2_0\otimes\mathrm{d}(\varpi(\pi_\varepsilon(a^2_1)))\\
	&=a\otimes\varpi(h^1)\wedge\varpi(h^2)\\
	&=\omega,
\end{align*}
using that $a^0a^1_0a^2_0\otimes a^1_1a^2_1\otimes a^2_2\in A\otimes H^+\otimes H^+$, which implies that $a^0a^1a^2_0\otimes a^2_1=0$.
The proof for $k>2$ follows in complete analogy.
\end{proof}

\subsection{Complete calculus and horizontal forms}\label{sec:totalforms}

Fix a QPB $A$ with subalgebra of coinvariants $B:=A^{\mathrm{co}H}$ and a bicovariant FODC $(\Gamma_H,\mathrm{d}_H)$ on $H$ with maximal prolongation $\Omega^\bullet(H)$. In particular, according to Lemma \ref{lem:Hcomplete} $\Delta\colon H\to H\otimes H$ extends to a morphism $\Delta^\bullet\colon\Omega^\bullet(H)\to\Omega^\bullet(H)\otimes\Omega^\bullet(H)$ of DGAs, i.e. $\Omega^\bullet(H)$ is a complete DC on $H$ in the sense of Definition \ref{def:completeH}.
We now come to the central definition of this paper.
\begin{definition}[\cite{DurII} page 10]\label{def:complete}
A DC $\Omega^\bullet(A)$ on $A$ is called \textbf{complete} if the right $H$-coaction $\Delta_A\colon A\to A\otimes H$ extends to a morphism
\begin{equation}\label{DeltaABullet}
\Delta_A^\bullet\colon\Omega^\bullet(A)\to\Omega^\bullet(A)\otimes\Omega^\bullet(H)
\end{equation}
of DGAs. In this case we refer to $\Omega^\bullet(A)$ as the \textbf{total space forms}.
\end{definition}
In the following we employ the short notation
\begin{equation}
	\Delta_A^\bullet(\omega)=:\omega_{[0]}\otimes\omega_{[1]}
\end{equation}
for the extended coaction \eqref{DeltaABullet} of a complete calculus. Clearly Definition \ref{def:complete} generalizes Definition \ref{def:completeH} which treats the QPB $\Bbbk=H^{\mathrm{co}H}\subseteq H$. The following lemma is a conceptualized version of \cite[Lemma 3.3]{DurII}.
\begin{lemma}
Let $\Omega^\bullet(A)$ be a complete calculus on $A$. Then $\Omega^\bullet(A)$ is a graded right $\Omega^\bullet(H)$-comodule algebra, i.e. for all $\omega,\eta\in\Omega^\bullet(A)$
\begin{equation}\label{GradComAlg}
\Delta_A^\bullet(\omega\wedge\eta)=\Delta_A^\bullet(\omega)\Delta_A^\bullet(\eta),\qquad\qquad\Delta_A^\bullet(1)=1\otimes 1_H,
\end{equation}
and the diagrams
\begin{equation}\label{DelABulCoAct}
	\begin{tikzcd}
		\Omega^\bullet(A) \arrow{rr}{\Delta_A^\bullet} \arrow{d}[swap]{\Delta_A^\bullet} 
		& & \Omega^\bullet(A)\otimes\Omega^\bullet(H) \arrow{d}{\Delta_A^\bullet\otimes\mathrm{id}}\\
		\Omega^\bullet(A)\otimes\Omega^\bullet(H) \arrow{rr}{\mathrm{id}\otimes\Delta^\bullet}
		& & \Omega^\bullet(A)\otimes\Omega^\bullet(H)\otimes\Omega^\bullet(H)
	\end{tikzcd}
	\quad
	\begin{tikzcd}
		\Omega^\bullet(A) \arrow{r}{\Delta_A^\bullet} \arrow{dr}[swap]{\mathrm{id}}
		& \Omega^\bullet(A)\otimes\Omega^\bullet(H) \arrow{d}{\mathrm{id}\otimes\varepsilon^\bullet}\\
		& \Omega^\bullet(A)
	\end{tikzcd}
\end{equation}
commute. In particular, $\Omega^\bullet(A)$ is a right $H$-covariant DC on $A$.
\end{lemma}
\begin{proof}
The equations in \eqref{GradComAlg} hold, since by assumption $\Delta_A^\bullet$ is a morphism of graded algebras. In order to show that $\Delta_A^\bullet$ is a graded right $\Omega^\bullet(H)$-coaction we have to prove commutativity of the diagrams \eqref{DelABulCoAct}, which read
\begin{equation}
	(\Delta_A^\bullet\otimes\mathrm{id})\circ\Delta_A^\bullet=(\mathrm{id}\otimes\Delta^\bullet)\circ\Delta_A^\bullet,\qquad
	(\mathrm{id}\otimes\varepsilon^\bullet)\circ\Delta_A^\bullet=\mathrm{id}.
\end{equation}
Since both sides of the above equations are morphisms of DGAs they coincide if their lowest orders coincide. This is the case since $\Delta_A\colon A\to A\otimes H$ is a right $H$-coaction. 
Note that here the assumption that $\Delta\colon H\to H\otimes H$ extends to a morphisms of DGAs is crucial.
This concludes the proof of the lemma.
\end{proof}
We continue to prove that there is a canonical surjection $\Omega^\bullet(A)\twoheadrightarrow\mathrm{ver}^\bullet$ from the total space forms the vertical forms. This is taken from \cite[Proposition 3.6]{DurII}. We clarify however, that the complete calculus structure on $\mathrm{ver}^\bullet$ is not automatic, but rather induced from a complete calculus on the total space algebra (if such a calculus exists).% Via this surjection we induce a complete structure on $\mathrm{ver}^\bullet$.
\begin{proposition}\label{prop:piver}
Let $\Omega^\bullet(A)$ be a complete calculus on $A$. Then
\begin{enumerate}
\item[i.)] there is a surjective morphism of DGAs
\begin{equation}
	\pi_v\colon\Omega^\bullet(A)\to\mathrm{ver}^\bullet
\end{equation}
extending the identity $\mathrm{id}\colon A\to A$. It is determined on homogeneous elements by
\begin{equation}\label{piver}
\pi_v|_{\Omega^k(A)}:=(\mathrm{id}\otimes(\pi_{\Lambda^\bullet}\circ\pi^k))\circ\Delta_A^\bullet|_{\Omega^k(A)}\colon\Omega^k(A)\to\mathrm{ver}^k
\end{equation}
for all $k>0$, where $\pi^k\colon\Omega^\bullet(A)\to\Omega^k(A)$ is the obvious projection. Explicitly, we have
\begin{equation}\label{piverExpl}
\pi_v(a^0\mathrm{d}a^1\wedge\ldots\wedge\mathrm{d}a^k)
=a^0_0a^1_0\ldots a^k_0\otimes S(a^0_1a^1_1\ldots a^k_1)a^0_2\mathrm{d}a^1_2\wedge\ldots\wedge\mathrm{d}a^k_2
\end{equation}
for all $a^0,\ldots,a^k\in A$.

\item[ii.)] the vertical forms $\mathrm{ver}^\bullet$ are complete, i.e. the right $H$-coaction $\Delta_A\colon A\to A\otimes H$ extends to a morphism $\Delta_v^\bullet\colon\mathrm{ver}^\bullet\to\mathrm{ver}^\bullet\otimes\Omega^\bullet(H)$ of DGAs. Moreover, the diagram
\begin{equation}\label{diag:DeltaVer}
\begin{tikzcd}
\Omega^\bullet(A) \arrow{rr}{\pi_v} \arrow{d}[swap]{\Delta_A^\bullet}
& & \mathrm{ver}^\bullet \arrow{d}{\Delta_v^\bullet}\\
\Omega^\bullet(A)\otimes\Omega^\bullet(H) \arrow{rr}{\pi_v\otimes\mathrm{id}}
& & \mathrm{ver}^\bullet\otimes\Omega^\bullet(H)
\end{tikzcd}
\end{equation}
commutes.
\end{enumerate}
\end{proposition}
\begin{proof}
\begin{enumerate}
\item[i.)] By \eqref{piver} we define a graded map $\pi_v\colon\Omega^\bullet(A)\to\mathrm{ver}^\bullet$ extending the identity on $A$. On a $k$-form $a^0\mathrm{d}a^1\wedge\ldots\wedge\mathrm{d}a^k\in\Omega^k(A)$ this produces in fact \eqref{piverExpl}, as one easily verifies. We show that $\pi_v$ is a morphism of DGAs. Then, since it extends the identity, it will automatically be surjective.
Let $a^0\mathrm{d}a^1\wedge\ldots\wedge\mathrm{d}a^k\in\Omega^k(A)$ and $b^0\mathrm{d}b^1\wedge\ldots\wedge\mathrm{d}b^\ell\in\Omega^\ell(A)$. Then $\pi_v$ is compatible with the wedge products, since
{\small\begin{align*}
&\pi_v(a^0\mathrm{d}a^1\wedge\ldots\wedge\mathrm{d}a^k)\wedge_v\pi_v(b^0\mathrm{d}b^1\wedge\ldots\wedge\mathrm{d}b^\ell)\\
&=(a^0_0a^1_0\ldots a^k_0\otimes S(a^0_1a^1_1\ldots a^k_1)a^0_2\mathrm{d}a^1_2\wedge\ldots\wedge\mathrm{d}a^k_2)\wedge_v(b^0_0b^1_0\ldots b^\ell_0\otimes S(b^0_1b^1_1\ldots b^\ell_1)b^0_2\mathrm{d}b^1_2\wedge\ldots\wedge\mathrm{d}b^\ell_2)\\
&=a^0_0a^1_0\ldots a^k_0b^0_0b^1_0\ldots b^\ell_0\otimes((S(a^0_1a^1_1\ldots a^k_1)a^0_2\mathrm{d}a^1_2\wedge\ldots\wedge\mathrm{d}a^k_2)\leftharpoonup(b^0_1b^1_1\ldots b^\ell_1))\wedge S(b^0_2b^1_2\ldots b^\ell_2)b^0_3\mathrm{d}b^1_3\wedge\ldots\wedge\mathrm{d}b^\ell_3\\
&=a^0_0a^1_0\ldots a^k_0b^0_0b^1_0\ldots b^\ell_0\otimes S(a^0_1a^1_1\ldots a^k_1b^0_1b^1_1\ldots b^\ell_1)a^0_2\mathrm{d}a^1_2\wedge\ldots\wedge\mathrm{d}a^k_2\wedge b^0_2\mathrm{d}b^1_2\wedge\ldots\wedge\mathrm{d}b^\ell_2\\
&=\pi_v((a^0\mathrm{d}a^1\wedge\ldots\wedge\mathrm{d}a^k)\wedge(b^0\mathrm{d}b^1\wedge\ldots\wedge\mathrm{d}b^\ell))
\end{align*}} 
holds. Moreover, $\pi_v$ respects the differentials, as shown by
\allowdisplaybreaks
\begin{align*}
\mathrm{d}_v(\pi_v(a^0\mathrm{d}a^1\wedge\ldots\wedge\mathrm{d}a^k))
&=\mathrm{d}_v(a^0_0a^1_0\ldots a^k_0\otimes S(a^0_1a^1_1\ldots a^k_1)a^0_2\mathrm{d}a^1_2\wedge\ldots\wedge\mathrm{d}a^k_2)\\
&=a^0_0a^1_0\ldots a^k_0\otimes\mathrm{d}(S(a^0_1a^1_1\ldots a^k_1)a^0_2\mathrm{d}a^1_2\wedge\ldots\wedge\mathrm{d}a^k_2)\\
&\quad+a^0_0a^1_0\ldots a^k_0\otimes\varpi(\pi_\varepsilon(a^0_1a^1_1\ldots a^k_1))\wedge S(a^0_2a^1_2\ldots a^k_2)a^0_3\mathrm{d}a^1_3\wedge\ldots\wedge\mathrm{d}a^k_3\\
&=a^0_0a^1_0\ldots a^k_0\otimes\mathrm{d}(S(a^0_1a^1_1\ldots a^k_1))\wedge a^0_2\mathrm{d}a^1_2\wedge\ldots\wedge\mathrm{d}a^k_2\\
&\quad+a^0_0a^1_0\ldots a^k_0\otimes S(a^0_1a^1_1\ldots a^k_1)\mathrm{d}a^0_2\wedge\mathrm{d}a^1_2\wedge\ldots\wedge\mathrm{d}a^k_2)\\
&\quad+a^0_0a^1_0\ldots a^k_0\otimes S(a^0_1a^1_1\ldots a^k_1)\underbrace{\mathrm{d}(a^0_2a^1_2\ldots a^k_2)\wedge S(a^0_3a^1_3\ldots a^k_3)}_{=-a^0_2a^1_2\ldots a^k_2\mathrm{d}(S(a^0_3a^1_3\ldots a^k_3))}a^0_4\mathrm{d}a^1_4\wedge\ldots\wedge\mathrm{d}a^k_4\\
&=a^0_0a^1_0\ldots a^k_0\otimes S(a^0_1a^1_1\ldots a^k_1)\mathrm{d}a^0_2\wedge\mathrm{d}a^1_2\wedge\ldots\wedge\mathrm{d}a^k_2)\\
&=\pi_v(\mathrm{d}a^0\wedge\mathrm{d}a^1\wedge\ldots\wedge\mathrm{d}a^k)\\
&=\pi_v(\mathrm{d}(a^0\mathrm{d}a^1\wedge\ldots\wedge\mathrm{d}a^k)),
\end{align*}
where we used the Leibniz rule and $\mathrm{d}^2=0$ several times. This DGA morphism extends the identity $A\ni a\mapsto a\otimes 1\in A\otimes\Bbbk=A\otimes\Lambda^0$ on $A$ and we recover \eqref{piverExpl}, namely
\begin{align*}
\pi_v(a^0)\mathrm{d}_v(\pi_v(a^1))\wedge_v\ldots\wedge_v\mathrm{d}_v(\pi_v(a^k))
&=(a^0\otimes 1)\mathrm{d}_v(a^1\otimes 1)\wedge_v\ldots\wedge_v\mathrm{d}_v(a^k\otimes 1)\\
&=(a^0\otimes 1)\wedge_v(a^1_0\otimes S(a^1_1)\mathrm{d}a^1_2)\wedge_v\ldots\wedge_v(a^1_0\otimes S(a^k_1)\mathrm{d}a^k_2)\\
&=a^0a^1_0a^2_0\ldots a^k_0\otimes((S(a^1_1)\mathrm{d}a^1_2)\leftharpoonup(a^2_1\ldots a^k_1))\wedge\ldots\wedge S(a^k_k)\mathrm{d}a^k_{k+1}\\
&=a^0_0a^1_0\ldots a^k_0\otimes S(a^0_1a^1_1\ldots a^k_1)a^0_2\mathrm{d}a^1_2\wedge\ldots\wedge\mathrm{d}a^k_2\\
&=\pi_v(a^0\mathrm{d}a^1\wedge\ldots\wedge\mathrm{d}a^k)
\end{align*}
for all $a^0,a^1,\ldots,a^k\in A$.

\item[ii.)] We show that one can define $\Delta_v^\bullet$ by the commutativity of the diagram \eqref{diag:DeltaVer}. Then $\Delta_v^\bullet$ is automatically a DGA morphism extending the right $H$-coaction. 
We define $\widetilde{\Delta_v^\bullet}:=(\pi_v\otimes\mathrm{id})\circ\Delta_A^\bullet\colon\Omega^\bullet(A)\to\mathrm{ver}^\bullet\otimes\Omega^\bullet(H)$ and prove that $\widetilde{\Delta_v^\bullet}$ descends to the quotient $\mathrm{ver}^\bullet\cong\Omega^\bullet(A)/\ker\pi_v$. By definition the descended map will be $\Delta_v^\bullet$.
Let $a^0\mathrm{d}a^1\wedge\ldots\wedge\mathrm{d}a^k\in\Omega^k(A)$ such that 
\begin{equation}\label{eq:pivzero}
	0=\pi_v(a^0\mathrm{d}a^1\wedge\ldots\wedge\mathrm{d}a^k)=a^0_0a^1_0\ldots a^k_0\otimes S(a^0_1a^1_1\ldots a^k_1)a^0_2\mathrm{d}a^1_2\wedge\ldots\wedge\mathrm{d}a^k_2.
\end{equation}
Then, applying  $\Delta_A\otimes\Delta_{\Omega^k(A)}$ to the last equation, we obtain
\begin{align*}
0
&=a^0_0a^1_0\ldots a^k_0\otimes a^0_1a^1_1\ldots a^k_1\otimes\big(S(a^0_2a^1_2\ldots a^k_2)a^0_3\mathrm{d}a^1_3\wedge\ldots\wedge\mathrm{d}a^k_3\big)_0\otimes\big(S(a^0_2a^1_2\ldots a^k_2)a^0_3\mathrm{d}a^1_3\wedge\ldots\wedge\mathrm{d}a^k_3\big)_1\\
&=a^0_0a^1_0\ldots a^k_0\otimes a^0_1a^1_1\ldots a^k_1\otimes S(a^0_2a^1_2\ldots a^k_2)_1a^0_3\mathrm{d}a^1_3\wedge\ldots\wedge\mathrm{d}a^k_3\otimes S(a^0_2a^1_2\ldots a^k_2)_2a^0_4a^1_4\ldots a^k_4\\
&=a^0_0a^1_0\ldots a^k_0\otimes a^0_1a^1_1\ldots a^k_1\otimes S(a^0_3a^1_3\ldots a^k_3)_1a^0_4\mathrm{d}a^1_4\wedge\ldots\wedge\mathrm{d}a^k_4\otimes S(a^0_2a^1_2\ldots a^k_2)_2a^0_5a^1_5\ldots a^k_5.
\end{align*}
If we multiply the second and the last tensor component of the previous equation this implies
$$
0=a^0_0a^1_0\ldots a^k_0\otimes S(a^0_1a^1_1\ldots a^k_1)a^0_2\mathrm{d}a^1_2\wedge\ldots\wedge\mathrm{d}a^k_2\otimes a^0_3a^1_3\ldots a^k_3.
$$
Thus,
\begin{align*}
(\mathrm{id}\otimes\pi^k\otimes\mathrm{id})(\widetilde{\Delta_v^\bullet}(a^0\mathrm{d}a^1\wedge\ldots\wedge\mathrm{d}a^k))
=a^0_0a^1_0\ldots a^k_0\otimes S(a^0_1a^1_1\ldots a^k_1)a^0_2\mathrm{d}a^1_2\wedge\ldots\wedge\mathrm{d}a^k_2\otimes a^0_3a^1_3\ldots a^k_3
=0
\end{align*}
To see that also $(\mathrm{id}\otimes\pi^{k-\ell}\otimes\mathrm{id})(\widetilde{\Delta_v^\bullet}(a^0\mathrm{d}a^1\wedge\ldots\wedge\mathrm{d}a^k))
=0$ for $0<\ell\leq k$ one applies $\Delta_A\otimes(\pi^{k-\ell}\otimes\pi^\ell)\circ\Delta^\bullet$ rather than $\Delta_A\otimes\Delta_{\Omega^k(A)}$ to \eqref{eq:pivzero} and then multiplies the second and last tensor component of the result.
This then implies that $\widetilde{\Delta_v^\bullet}$ descends to the quotient $\mathrm{ver}^\bullet\cong\Omega^\bullet(A)/\ker\pi_v$ and consequently the descended map $\Delta_v^\bullet$ makes the diagram \eqref{diag:DeltaVer} commute, which means that $\Delta_v^\bullet$ is a DGA morphism extending the right $H$-coaction.
\end{enumerate}
\end{proof}
Given the DGA morphism $\Delta_A^\bullet\colon\Omega^\bullet(A)\to\Omega^\bullet(A)\otimes\Omega^\bullet(H)$ there is a canonical notion of horizontal forms, given by the total space forms which have trivial vertical components.
\begin{definition}[{\cite[Definition 3.2]{DurII}}]
For a complete calculus $\Omega^\bullet(A)$ on a QPB $A$ we define the \textbf{horizontal forms} as the preimage
\begin{equation}
	\mathrm{hor}^\bullet:=(\Delta_A^\bullet)^{-1}(\Omega^\bullet(A)\otimes H)
\end{equation} 
of $\Omega^\bullet(A)\otimes H$ under $\Delta_A^\bullet\colon\Omega^\bullet(A)\to\Omega^\bullet(A)\otimes\Omega^\bullet(H)$.
\end{definition}
The following result is a conceptualized version of \cite[Lemma 3.4]{DurII}.
\begin{corollary}
$\mathrm{hor}^\bullet$ is a graded right $H$-comodule algebra with right $H$-coaction given by the restriction
\begin{equation}
	\Delta_A^\bullet|_{\mathrm{hor}^\bullet}\colon\mathrm{hor}^\bullet\to\mathrm{hor}^\bullet\otimes H
\end{equation}
of $\Delta_A^\bullet$.
\end{corollary}
\begin{proof}
Let $\theta\in\mathrm{hor}^k$, $\theta'\in\mathrm{hor}^\ell$ with corresponding $\omega^i\in\Omega^k(A)$, $\eta^j\in\Omega^\ell(A)$ and $h^i,g^j\in H$ such that $\Delta_A^\bullet(\theta)=\omega^i\otimes h^i$ and $\Delta_A^\bullet(\theta')=\eta^j\otimes g^j$ (sum over $i$ and $j$ understood). Then $\Delta_A^\bullet(\theta\wedge\theta')=\Delta_A^\bullet(\theta)\wedge_\otimes\Delta_A^\bullet(\theta')=\omega^i\wedge\eta^j\otimes h^ig^j$ shows that $\mathrm{hor}^\bullet$ is a graded algebra. Furthermore, $\Delta_A^\bullet(\theta)\in\mathrm{hor}^k\otimes H$, since 
$$
(\Delta_A^\bullet\otimes\mathrm{id})(\Delta_A^\bullet(\theta))
=(\mathrm{id}\otimes\Delta^\bullet)(\Delta_A^\bullet(\theta))
=\omega^i\otimes\Delta^\bullet(h^i)
=\omega^i\otimes h^i_1\otimes h^i_2\in\Omega^k(A)\otimes H\otimes H,
$$
using the coaction axioms of $\Delta_A^\bullet$. Since the restriction of the comodule algebra morphism $\Delta_A^\bullet$ closes in $\mathrm{hor}^\bullet$ it is a comodule algebra morphism.
\end{proof}
Note that $\mathrm{hor}^\bullet$ is \textbf{not} a DGA, since for $\theta\in\mathrm{hor}^k$ with $\Delta_A^\bullet(\theta)=\omega^i\otimes h^i\in\Omega^k(A)\otimes H$ one obtains
\begin{equation}
	\Delta_A^\bullet(\mathrm{d}\theta)
	=\mathrm{d}_\otimes(\Delta_A^\bullet(\theta))
	=\mathrm{d}_\otimes(\omega^i\otimes h^i)
	=\mathrm{d}\omega^i\otimes h^i+(-1)^{|\omega^i|}\omega^i\otimes\mathrm{d}h^i,
\end{equation}
using that $\Delta_A^\bullet$ is a DGA morphism. Clearly, the above is in general not in $\Omega^k(A)\otimes H\otimes H$, i.e. the second term might not vanish.

We end this section by proving a decomposition result, expressing $\Delta_A^{k+\ell}\colon\Omega^{k+\ell}(A)\to\bigoplus_{m=0}^{k+\ell}\Omega^{k+\ell-m}(A)\otimes\Omega^m(H)$ in terms of its graded components
\begin{equation}
	\mathrm{ver}^{k,\ell}:=(\pi^k_A\otimes\pi^\ell_H)\circ\Delta_A^{k+\ell}\colon\Omega^{k+\ell}(A)\to\Omega^k(A)\otimes\Omega^\ell(H).
\end{equation}
We are colloquially referring to them as the \textbf{(higher) vertical maps}.
These will be particularly instrumental in the examples of Section \ref{sec:ex}.
\begin{proposition}
Given a complete calculus $\Omega^\bullet(A)$ the higher vertical maps on a wedge product decompose as
\begin{equation}\label{eq:higherVM}
	\mathrm{ver}^{k,\ell}(\omega\wedge\eta)
	=\sum_{m=0}^{|\omega|}\mathrm{ver}^{m,|\omega|-m}(\omega)\mathrm{ver}^{k-m,|\eta|-k+m}(\eta)
\end{equation}
for $\omega,\eta\in\Omega^\bullet(A)$ such that $|\omega|+|\eta|=k+\ell$, where in the above sum we set the terms with $\mathrm{ver}^{m,n}$ and $m<0$ or $n<0$ to zero. The product on the right hand side of \eqref{eq:higherVM} is the one of $\Omega^\bullet(A\otimes H)$.
\end{proposition}
\begin{proof}
Let $\omega=a^0\mathrm{d}a^1\wedge\ldots\wedge\mathrm{d}a^r$ and $\eta=c^0\mathrm{d}c^1\wedge\ldots\wedge\mathrm{d}c^s$ be arbitrary homogeneous elements of $\Omega^\bullet(A)$. The left hand side of \eqref{eq:higherVM} reads
\begin{align*}
	\mathrm{ver}^{k,\ell}(\omega\wedge\eta)
	&=(\pi^k_A\otimes\pi^\ell_H)(\Delta^\bullet_A(\omega\wedge\eta))\\
	&=(\pi^k_A\otimes\pi^\ell_H)(\Delta_A(a^0)\mathrm{d}_\otimes(\Delta_A(a^1))\wedge_\otimes\ldots\wedge_\otimes\mathrm{d}_\otimes(\Delta_A(a^r))\\
	&\qquad\wedge_\otimes\Delta_A(c^0)\mathrm{d}_\otimes(\Delta_A(c^1))\wedge_\otimes\ldots\wedge_\otimes\mathrm{d}_\otimes(\Delta_A(c^s))),
\end{align*}
while the right hand side of \eqref{eq:higherVM} is
\begin{align*}
	\sum_{m=0}^{|\omega|}\mathrm{ver}^{m,|\omega|-m}(\omega)\mathrm{ver}^{k-m,|\eta|-k+m}(\eta)
	&=\sum_{m=0}^{|\omega|}(\pi^m_A\otimes\pi^{|\omega|-m}_H)(\Delta_A(a^0)\mathrm{d}_\otimes(\Delta_A(a^1))\wedge_\otimes\ldots\wedge_\otimes\mathrm{d}_\otimes(\Delta_A(a^r))\\
	&\qquad\wedge_\otimes(\pi^{k-m}_A\otimes\pi^{|\eta|-k+m})(\Delta_A(c^0)\mathrm{d}_\otimes(\Delta_A(c^1))\wedge_\otimes\ldots\wedge_\otimes\mathrm{d}_\otimes(\Delta_A(c^s))).
\end{align*}
In the first expression we obtain the $(k,\ell)$-component of $\Delta_A^\bullet(\omega\wedge\eta)$, while the second expression computes the $\wedge_\otimes$-product of the $(m,|\omega|-m)$-component of $\Delta_A^\bullet(\omega)$ and the $(k-m,|\eta|-k+m)$-component of $\Delta_A^\bullet(\eta)$ for all $m=0,\ldots,|\omega|$. The latter exhausts all possible combinations to obtain a $(m,|\omega|-m)$-component in $\Delta_A^\bullet(\omega\wedge\eta)$ and thus the equality follows.
\end{proof}

\subsection{Base forms and the Atiyah sequence}\label{sec:baseforms}

After introducing vertical forms, total space forms and horizontal forms it remains to define base forms to complete the differential picture for QPBs.
\begin{definition}[{\cite[page 13]{DurII}}]
Let $\Omega^\bullet(A)$ be a complete calculus on a QPB $B=A^{\mathrm{co}H}\subseteq A$. The corresponding \textbf{base forms} are defined as the graded subspace $\Omega^\bullet(B)$ of $\Omega^\bullet(A)$ which is invariant under the right $\Omega^\bullet(H)$-coaction $\Delta_A^\bullet\colon\Omega^\bullet(A)\to\Omega^\bullet(A)\otimes\Omega^\bullet(H)$, i.e.
\begin{equation}
	\Omega^\bullet(B):=\{\omega\in\Omega^\bullet(A)~|~\Delta_A^\bullet(\omega)=\omega\otimes 1_H\}.
\end{equation}
\end{definition}
Since $\Delta_A^\bullet$ is a morphism of DGAs the following result is immediate.
\begin{lemma}
$\Omega^\bullet(B)\subseteq\Omega^\bullet(A)$ is a differential graded subalgebra.
\end{lemma}
Clearly $\Omega^0(B)=B$, however it is not obvious that $\Omega^\bullet(B)$ is generated by $B$. Thus, $\Omega^\bullet(B)$ is \textit{not} a DC on $B$ in general, only a differential graded subalgebra of $\Omega^\bullet(A)$. However, in all explicit examples which we encounter in this article $\Omega^\bullet(B)$ will be a DC on $B$.
\begin{proposition}\label{prop:baseintersection}
Base forms equal the intersection of horizontal and right $H$-coinvariant forms, i.e.
\begin{equation}
	\Omega^\bullet(B)=\mathrm{hor}^\bullet\cap\Omega^\bullet(A)^{\mathrm{co}H}.
\end{equation}
\end{proposition}
\begin{proof}
The inclusion $\subseteq$ is trivially satisfies. For the other inclusion we note that a horizontal form $\omega\in\mathrm{hor}^\bullet$ satisfies $\Delta_A^\bullet(\omega)=\Delta_{\Omega^\bullet(A)}(\omega)=\eta^i\otimes h^i$ for some $\eta^i\in\Omega^\bullet(A)$ and $h^i\in H$. If $\omega$ is right $H$-coinvariant, as well, this implies $\Delta_A^\bullet(\omega)=\omega\otimes 1$, i.e. $\omega\in\Omega^\bullet(B)$. 
\end{proof}
We conclude this section by proving one of the main results of this paper (compare to \cite[Lemma 3.7]{DurII}), namely exactness of the Atiyah sequence. It states that total space forms surject in vertical forms, with kernel given precisely by the horizontal forms. 
\begin{theorem}\label{thm:exseq}
For any complete calculus $\Omega^\bullet(A)$ on a QPB $B=A^{\mathrm{co}H}\subseteq A$ the \textbf{Atiyah sequence}
\begin{equation}\label{exactsequence}
	0\to\mathrm{hor}^1\hookrightarrow\Omega^1(A)\xrightarrow{\pi_v}\mathrm{ver}^1\to 0
\end{equation}
is exact in the category ${}_A\mathcal{M}_A^H$ of right $H$-covariant $A$-bimodules.
\end{theorem} 
\begin{proof}
By definition $\mathrm{hor}^1\subseteq\Omega^1(A)$ and by Proposition \ref{prop:piver} $i.)$ the map $\pi_v$ is surjective. Thus, it remains to prove $\ker\pi_v=\mathrm{hor}^1$. Let $\omega=a\mathrm{d}a'\in\Omega^1(A)$ be an arbitrary $1$-form for the moment. Then
\begin{equation}\label{eq:helpsequence}
\Delta_A^1(\omega)=a_0\mathrm{d}a'_0\otimes a_1a'_1
+a_0a'_0\otimes a_1\mathrm{d}a'_1.
\end{equation}
If $\omega\in\mathrm{hor}^1$, i.e. $\Delta_A^1(\omega)\in\Omega^1(A)\otimes H$, the equation \eqref{eq:helpsequence} implies $a_0a'_0\otimes a_1\mathrm{d}a'_1=0$. Applying $(\mathrm{id}\otimes S\otimes\mathrm{id})\circ(\Delta_A\otimes\mathrm{id})$ to the last equation gives
$a_0a'_0\otimes S(a_1a'_1)\otimes a_2\mathrm{d}a'_2=0$. Multiplying the last two tensor factors implies
$$
0
=a_0a'_0\otimes S(a_1a'_1)a_2\mathrm{d}a'_2
=\pi_v(a\mathrm{d}a')
=\pi_v(\omega)
$$
and thus $\mathrm{hor}^1\subseteq\ker\pi_v$. For the other inclusion assume that $\omega=a\mathrm{d}a'\in\ker\pi_v$, i.e. $a_0a'_0\otimes S(a_1a'_1)a_2\mathrm{d}a'_2=0$. Applying $\Delta_A\otimes\mathrm{id}$ to the last equation and multiplying the last two tensor factors gives
$a_0a'_0\otimes a_1\mathrm{d}a'_1=0$, which together with \eqref{eq:helpsequence} implies $\omega\in\mathrm{hor}^1$, i.e. $\ker\pi_v\subseteq\mathrm{hor}^1$. This concludes the proof of the theorem.
\end{proof}
In Section \ref{sec:connections} we are going to consider connections of QPBs as section of the Atiyah sequence. Furthermore, in the following section we compare the above Atiyah sequence of \DJ ur\dj evi\'c with the more established Atiyah sequence of Brzezi\'nski--Majid.

\subsection{First order and the Brzezi\'nski--Majid approach}\label{sec:compare}

The name "quantum principal bundle" has been given to many structures in the context of comodule algebras. We recall the ones which are relevant for us (not claiming to be exhaustive) and clarify their relation to each other. We start with the simplest and add more complexity.

In some contexts, "quantum principal bundle" is used synonymous to Hopf--Galois extension. This is justified by the fact that in (classical) differential geometry a $G$-principal bundles $P\to M$ gives rise a Hopf--Galois extension $B:=\mathscr{C}^\infty(M)\subseteq A:=\mathscr{C}^\infty(P)$, see for example \cite[Example 2.13]{AsBiPaSc}. Others \cite{AFLW} assume the Hopf--Galois extension to be faithfully flat in addition in order to apply Schneider's equivalence \cite{Schn90}. However, faithfully flat Hopf--Galois extensions are more commonly called \textbf{principal comodule algebras} \cite{DGH}. This is also the point of view of \DJ ur\dj evi\'c in \cite{DurHG} and implicitly in \cite{DurI,DurDS,DurII}.

Probably the most established notion of "quantum principal bundle" is the one introduced by Brzezi\'nski--Majid in \cite{BrzMaj}, see also \cite[Definition 5.39]{BegMaj}, which we are going to review in more detail in the following.
\begin{definition}\label{def:QPB-BrMa}
Given a right $H$-covariant FODC $(\Gamma_A,\mathrm{d}_A)$ on a right $H$-comodule algebra $A$ and a bicovariant FODC $(\Gamma_H,\mathrm{d}_H)$ on $H$ we have a \textbf{quantum principal bundle à la Brzezi\'nski--Majid} if the \textbf{vertical map}
\begin{equation}
	\mathrm{ver}_\mathrm{BM}\colon\Gamma_A\to A\otimes\Lambda^1,\qquad
	\mathrm{ver}_\mathrm{BM}(a\mathrm{d}_A(a'))=aa'_0\otimes\varpi(a'_1)=aa'_0\otimes S(a'_1)\mathrm{d}_H(a'_2)
\end{equation}
is well-defined and the sequence
\begin{equation}\label{sequenceBM}
	0\to A\mathrm{d}_A(B)A\hookrightarrow\Gamma_A\xrightarrow{\mathrm{ver}_\mathrm{BM}}A\otimes\Lambda^1\to 0
\end{equation}
is exact. We call $\Omega^1_\mathrm{hor}:=A\mathrm{d}_A(B)A$ the \textbf{horizontal $1$-forms à la Brzezi\'nski--Majid}.
\end{definition}

To distinguish a quantum principal bundle as in Definition \ref{def:QPB-BrMa} from a quantum principal bundle as in Definition \ref{def:QPB} we always refer to the former as a quantum principal bundle \textit{à la Brzezi\'nski--Majid} and we call the latter a quantum principal bundle (QPB) (or a quantum principal bundle \textit{à la \DJ ur\dj evi\'c} if we want to emphasize the context). In the rest of this section we show that in the context of faithfully flat Hopf--Galois extensions QPB à la Brzezi\'nski--Majid correspond to DCi which are "first order complete". Namely, up to a certain equality of horizontal forms, the Brzezi\'nski--Majid QPB approach can be understood as the first order of the \DJ ur\dj evi\'c approach.
\begin{definition}\label{def:FOcomplete}
A DC $\Omega^\bullet(A)$ on a quantum principal bundle à la \DJ ur\dj evi\'c is called \textbf{first order complete} if the right $H$-coaction $\Delta_A\colon A\to A\otimes H$ is $1$-differentiable, i.e. if there is a morphism of FODCi 
\begin{equation}
	\Delta_A^1\colon\Omega^1(A)\to\Omega^1(A\otimes H)=(\Omega^1(A)\otimes H)\oplus(A\otimes\Omega^1(H))
\end{equation}
extending $\Delta_A$:
\begin{equation}
\begin{tikzcd}
\Omega^1(A) \arrow{rr}{\Delta_A^1}
& & \Omega^1(A\otimes H)\\
A \arrow{u}{\mathrm{d}} \arrow{rr}{\Delta_A}
& & A\otimes H \arrow{u}[swap]{\mathrm{d}_\otimes}
\end{tikzcd}
\end{equation}
\end{definition}

In the following proposition we clarify the relation between quantum principal bundles à la Brzezi\'nski--Majid and differential calculi on quantum principal bundles à la \DJ ur\dj evi\'c which are first order complete. 
The context of Brzezi\'nski--Majid is more general in the sense that only a right $H$-comodule algebra is assumed, rather than a faithfully flat Hopf--Galois extension $B=A^{\mathrm{co}H}\subseteq A$ as in \DJ ur\dj evi\'c's approach. Thus, to compare the two we have to restrict our attention to faithfully flat Hopf--Galois extensions.
\begin{proposition}\label{prop:compare}
Let $B=A^{\mathrm{co}H}\subseteq A$ be a faithfully flat Hopf--Galois extension.
\begin{enumerate}
\item[i.)] If there is a quantum principal bundle $(\Gamma_H,\Gamma_A)$ à la Brzezi\'nski--Majid on $A$, then the maximal prolongation of $\Gamma_A$ is first order complete.

\item[ii.)] If there is a first order complete DC $\Omega^\bullet(A)$ on $A$ with corresponding structure Hopf algebra DC $\Omega^\bullet(H)$, then the truncation $(\Omega^1(H),\Omega^1(A))$ is a quantum principal bundle à la Brzezi\'nski--Majid if and only if the horizontal $1$-forms of Brzezi\'nski--Majid and \DJ ur\dj evi\'c coincide, i.e. if and only if $A\mathrm{d}(B)A=\mathrm{hor}^1$.
\end{enumerate}
\end{proposition}
\begin{proof}
\begin{enumerate}
\item[i.)] Let us assume a QPB à la Brzezi\'nski--Majid $(\Gamma_H,\Gamma_A)$. We show that the vertical maps
\begin{equation*}
\begin{split}
\Delta_{\Gamma_A}=\mathrm{ver}^{1,0}\colon\Gamma_A&\to\Gamma_A\otimes H\\
a\mathrm{d}_A(a')&\mapsto a_0\mathrm{d}_A(a'_0)\otimes a_1a'_1
\end{split}\qquad,\qquad
\begin{split}
\mathrm{ver}=\mathrm{ver}^{0,1}\colon\Gamma_A&\to A\otimes\Gamma_H\\
a\mathrm{d}_A(a')&\mapsto a_0a'_0\otimes a_1\mathrm{d}_H(a'_1)
\end{split}
\end{equation*}
are well-defined, which is precisely first order completeness of the maximal prolongation of $\Gamma_A$. By assumption $\Delta_{\Gamma_A}$ is well-defined, since $\Gamma_A$ is a right $H$-covariant FODC. Moreover, by assumption the map $\mathrm{ver}_\mathrm{BM}\colon\Gamma_A\to A\otimes\Gamma_H$, $a\mathrm{d}_A(a')\mapsto aa'_0\otimes S(a'_1)\mathrm{d}_H(a'_2)=a_0a'_0\otimes S(a_1a'_1)a_2\mathrm{d}_H(a'_2)$ is well-defined. Thus, $(\mathrm{id}\otimes \lambda_{\Gamma_H})\circ(\Delta_A\otimes\mathrm{id})\circ\mathrm{ver}_\mathrm{BM}\colon\Gamma_A\to H\otimes\Gamma_H$ is well-defined, where $\lambda_{\Gamma_H}\colon H\otimes\Gamma_H\to\Gamma_H$ is the left $H$-module action on $\Gamma_H$. The previous composition of maps coincides with $\mathrm{ver}$, since
\begin{align*}
(\mathrm{id}\otimes \lambda_{\Gamma_H})((\Delta_A\otimes\mathrm{id})(\mathrm{ver}_\mathrm{BM}(a\mathrm{d}_A(a'))))
&=a_0a'_0\otimes\lambda_{\Gamma_H}(a_1a'_1\otimes S(a_2a'_2)a_3\mathrm{d}_H(a'_3))\\
&=a_0a'_0\otimes a_1\mathrm{d}_H(a'_1)\\
&=\mathrm{ver}(a\mathrm{d}_A(a'))
\end{align*}
and thus $\mathrm{ver}$ is well-defined.

\item[ii.)] Consider a calculus $\Omega^\bullet(A)$ which is first order complete, with corresponding DC $\Omega^\bullet(H)$ on $H$. In particular, $\Omega^1(A)$ is right $H$-covariant and $\Omega^1(H)$ is bicovariant. Moreover, the vertical map $\mathrm{ver}=\mathrm{ver}^{0,1}\colon\Omega^1(A)\to A\otimes\Omega^1(H)$ is well-defined, implying that
\begin{equation}
	\mathrm{ver}_\mathrm{BM}=(\mathrm{id}\otimes \lambda_{\Gamma_H})\circ(\mathrm{id}\otimes S\otimes\mathrm{id})\circ(\Delta_A\otimes\mathrm{id})\circ\mathrm{ver}
\end{equation}
is well-defined. It remains to prove that \eqref{sequenceBM} is exact. From Theorem \ref{thm:exseq} we know that the sequence \eqref{exactsequence} is exact. From \eqref{piverExpl} we read that $\mathrm{ver}_\mathrm{BM}$ coincides with $\pi_v|_{\Omega^1(A)}\colon\Omega^1(A)\to A\otimes\Lambda^1$. Thus $\mathrm{ver}_\mathrm{BM}$ is surjective. Since $\Omega^1_\mathrm{hor}\hookrightarrow\Omega^1(A)$ is automatically injective it remains to prove that $\ker\mathrm{ver}_\mathrm{BM}=\Omega^1_\mathrm{hor}$. By the exactness of \eqref{exactsequence} we know that $\mathrm{hor}^1=\ker\pi_v=\ker\mathrm{ver}_\mathrm{BM}$. Thus, we obtain a QPB à la Brzezi\'nski--Majid if and only if $\mathrm{hor}^1=\Omega^1_\mathrm{hor}$.
\end{enumerate}
\end{proof}
In all explicit examples which we encounter in Section \ref{sec:ex} the condition $\mathrm{hor}^1=\Omega^1_\mathrm{hor}$ is satisfied and thus the complete calculus approach of this paper is an extension of the QPB approach of Brzezi\'nski--Majid. We would like to remark that \DJ ur\dj evi\'c briefly comments on the relation with the Brzezi\'nski--Majid approach in \cite[Section 4]{DurHG}.

It turns out that first order completeness is sufficient to prove completeness in \textit{all} degrees if the DC is the maximal prolongation. We prove this statement in the following proposition, following \cite[Lemma 4.29]{BegMaj}.
\begin{proposition}\label{prop:comlete}
Let $B=A^{\mathrm{co}H}\subseteq A$ be a faithfully flat Hopf--Galois extension and $\Omega^\bullet(A)$ a first-order complete calculus on $A$. If $\Omega^\bullet(A)$ is the maximal prolongation of $\Omega^1(A)$ then $\Omega^\bullet(A)$ is a complete DC (in all orders).
\end{proposition}
\begin{proof}
As discussed in and before Proposition \ref{prop:maxprol}, the maximal prolongation is given as the quotient of the tensor algebra with an ideal in degree $2$. The ideal consists of all $\mathrm{d}a^i\otimes_A\mathrm{d}b^i$ with $a^i,b^i\in A$ (sum understood), such that $a^i\mathrm{d}b^i=0$. It is thus sufficient to prove that $\Delta_A^1$ extends as an algebra morphism to degree $2$ in order to ensure that it extends to all degrees. Let $\mathrm{d}a^i\otimes_A\mathrm{d}b^i$ be in the ideal as before. We show that
\begin{equation}\label{eq:idealvanish}
\begin{split}
\Delta^1_A(\mathrm{d}a^i)\Delta_A^1(\mathrm{d}b^i)
&=(\mathrm{d}(a^i_0)\otimes a^i_1+a^i_0\otimes\mathrm{d}(a^i_1))(\mathrm{d}(b^i_0)\otimes b^i_1+b^i_0\otimes\mathrm{d}(b^i_1))\\
&=\mathrm{d}(a^i_0)\wedge\mathrm{d}(b^i_0)\otimes a^i_1b^i_1
-a^i_0\mathrm{d}(b^i_0)\otimes \mathrm{d}(a^i_1)b^i_1
+\mathrm{d}(a^i_0)b^i_0\otimes a^i_1\mathrm{d}(b^i_1)
+a^i_0b^i_0\otimes \mathrm{d}(a^i_1)\wedge\mathrm{d}(b^i_1)
\end{split}
\end{equation}
vanishes on $\Omega^\bullet(A)\otimes\Omega^\bullet(H)$. Now since $a^i\mathrm{d}b^i=0$ and $\Delta_A^1$ is well-defined by assumption obtain
$
0=\Delta_A^1(a^i\mathrm{d}b^i)
=a^i_0\mathrm{d}(b^i_0)\otimes a^i_1b^i_1
+a^i_0b^i_0\otimes a^i_1\mathrm{d}(b^i_1).
$
From the direct sum we obtain the two equations
\begin{align}
	0&=a^i_0\mathrm{d}(b^i_0)\otimes a^i_1b^i_1,\label{eq:idealvanish1}\\
	0&=a^i_0b^i_0\otimes a^i_1\mathrm{d}(b^i_1).\label{eq:idealvanish2}
\end{align}
Applying $\mathrm{d}\otimes\mathrm{id}$ to \eqref{eq:idealvanish1} shows that the first term of the right hand side of \eqref{eq:idealvanish} vanishes, while applying $\mathrm{id}\otimes\mathrm{d}$ to \eqref{eq:idealvanish2} shows that the last term of the right hand side of \eqref{eq:idealvanish} vanishes. Applying $\mathrm{id}\otimes\mathrm{d}$ to \eqref{eq:idealvanish1}, while applying $\mathrm{d}\otimes\mathrm{id}$ to \eqref{eq:idealvanish2} and subtracting the former from the latter implies that the second and third term on the right hand side of \eqref{eq:idealvanish} vanish, as well. Thus \eqref{eq:idealvanish} is zero and we conclude the proof.
\end{proof}
When facing examples one often encounters the maximal prolongation and thus Proposition \ref{prop:comlete} implies that it is sufficient to check completeness in first order, which itself is "morally" equivalent to a quantum principal bundle à la Brzezi\'nski--Majid according to Proposition \ref{prop:compare}. However, we would like to stress that the maximal prolongation is not always the recipe of choice, for example when working with higher order bicovariant calculi à la Woronowicz \cite{Wor89} or when considering higher order inner calculi on graphs \cite[Corollary 1.54]{BegMaj}. In such situations higher order completeness is not automatic and has to be checked case by case.

\section{Connections}\label{sec:connections}

We discuss the notion of connections on first order complete calculi. They appear as sections of the Atiyah sequence \eqref{exactsequence}. Most of the material presented here is standard (see for example \cite{BegMaj,DGH,Ma99} and references therein) eventhough adapted to the slightly more general setup of \DJ ur\dj evi\'c (recall that $\Omega^1_\mathrm{hor}\subsetneq\mathrm{hor}^1$ in general). When considering strong connections we are going to restrict our attention to the situation $\Omega^1_\mathrm{hor}=\mathrm{hor}^1$, namely for which the first order \DJ ur\dj evi\'c approach coincides with the one of Brzezi\'nski--Majid.

\subsection{Connections and connection $1$-forms}\label{C4.1}

In this section we fix a QPB $B=A^{\mathrm{co}H}\subseteq A$ and a DC $\Omega^\bullet(A)$ on $A$ which is first order complete (as in Definition \ref{def:FOcomplete}). The corresponding bicovariant FODC on $H$ is denoted by $\Gamma\cong H\otimes\Lambda^1\cong H\otimes H^+/I$.

\begin{definition}
A \textbf{connection} on $\Omega^1(A)$ is a morphism
\begin{equation}
	\Pi\colon\Omega^1(A)\to\Omega^1(A)
\end{equation}
in ${}_A\mathcal{M}^H$, such that
\begin{enumerate}
\item[i.)] $\Pi$ is a projector, i.e. $\Pi^2=\Pi$ and

\item[ii.)] $\ker\Pi=\mathrm{hor}^1$.
\end{enumerate}
\end{definition}

\begin{definition}\label{def:con}
A \textbf{connection $1$-form} on $\Omega^1(A)$ is a morphism
\begin{equation}
	s\colon\Lambda^1\to\Omega^1(A)
\end{equation}
in $\mathcal{M}^H$, where we endow $\Lambda^1$ with the adjoint coaction $\mathrm{Ad}\colon\Lambda^1\to\Lambda^1\otimes H$, $\varpi(\pi_\varepsilon(h))\mapsto\varpi(\pi_\varepsilon(h_2))\otimes S(h_1)h_3$, such that
\begin{equation}
	\pi_v\circ s=1_A\otimes\mathrm{id}_{\Lambda^1}
\end{equation}
is an equality of maps $\Lambda^1\to A\otimes\Lambda^1$.
\end{definition}
Clearly, a connection $1$-form $s\colon\Lambda^1\to\Omega^1(A)$ is equivalent to a splitting $\hat{s}\colon A\otimes\Lambda^1\to\Omega^1(A)$ in ${}_A\mathcal{M}^H$ of the exact sequence
\begin{equation}
\begin{tikzcd}
0 \arrow{r}{~}
& \mathrm{hor}^1 \arrow{r}{~}
& \Omega^1(A) \arrow{r}{\pi_v}
& A\otimes\Lambda^1 \arrow{r}{~} \arrow[bend left=33]{l}{\hat{s}}
& 0
\end{tikzcd}
\end{equation}
in ${}_A\mathcal{M}^H$. From $s$ we define $\hat{s}$ by $\hat{s}(a\otimes\vartheta):=as(\vartheta)$ and given $\hat{s}$ we determine $s$ via $s(\vartheta):=\hat{s}(1\otimes\vartheta)$ for all $a\in A$ and $\vartheta\in\Lambda^1$.

\begin{proposition}\label{prop:1:1connection}
There are $1:1$-correspondences between
\begin{equation}
\begin{Bmatrix}
\text{connections}\\
\text{on }\Omega^1(A)
\end{Bmatrix}\xleftrightarrow{1:1}
\begin{Bmatrix}
	\text{connection $1$-forms}\\
	\text{on }\Omega^1(A)
\end{Bmatrix}\xleftrightarrow{1:1}
\begin{Bmatrix}
	\kappa\colon\Omega^1(A)\xrightarrow{\cong}\mathrm{ver}^1\oplus\mathrm{hor}^1\\
	\text{in }{}_A\mathcal{M}^H
\end{Bmatrix}.
\end{equation}
Explicitly,
\begin{enumerate}
\item[i.)] given a connection $\Pi\colon\Omega^1(A)\to\Omega^1(A)$ one defines a connection $1$-form by
\begin{equation}
	s(\vartheta):=\Pi(\widetilde{\pi_v}^{-1}(1\otimes\vartheta))
\end{equation}
for all $\vartheta\in\Lambda^1$, where $\widetilde{\pi_v}\colon\Omega^1(A)/\ker\pi_v\xrightarrow{\cong}A\otimes\Lambda^1$, while given a connection $1$-form $s\colon\Lambda^1\to\Omega^1(A)$ one defines
\begin{equation}
	\Pi(a\mathrm{d}a'):=aa'_0s(\varpi(\pi_\varepsilon(a'_1)))
\end{equation} 
for all $a\mathrm{d}a'\in\Omega^1(A)$.

\item[ii.)] given a connection $\Pi\colon\Omega^1(A)\to\Omega^1(A)$ with corresponding connection $1$-form $s\colon\Lambda^1\to\Omega^1(A)$ we define
\begin{equation}
\begin{split}
	\kappa\colon\Omega^1(A)&\to\mathrm{ver}^1\oplus\mathrm{hor}^1,\\
	\omega&\mapsto\pi_v(\omega)+(\mathrm{id}-\Pi)(\omega)
\end{split}
\end{equation}
with inverse
\begin{equation}
\begin{split}
	\kappa^{-1}\colon\mathrm{ver}^1\oplus\mathrm{hor}^1&\to\Omega^1(A),\\
	a\otimes\vartheta+\omega&\mapsto as(\vartheta)+\omega.
\end{split}
\end{equation}
On the other hand, given an isomorphism $\kappa\colon\Omega^1(A)\xrightarrow{\cong}\mathrm{ver}^1\oplus\mathrm{hor}^1$ we obtain a connection via
\begin{equation}
	\Pi:=\mathrm{id}-\mathrm{pr}_{\mathrm{hor}^1}\circ\kappa\colon\Omega^1(A)\to\Omega^1(A),
\end{equation}
where $\mathrm{pr}_{\mathrm{hor}^1}\colon\mathrm{ver}^1\oplus\mathrm{hor}^1\to\mathrm{hor}^1$ denotes the canonical projection.
\end{enumerate} 
\end{proposition}
\begin{proof}
The proof is elementary and simply requires to verify all properties, since we explicitly listed all operations.
\end{proof}

\subsection{Strong connections}\label{C4.2}

In this section we briefly recall the notion of strong connection as introduced in \cite{DGH}, see also \cite[Section 6.3]{BrJaMa} and \cite[Section 5.4]{BegMaj}. We recall that the datum of a faithfully flat Hopf--Galois extension is equivalent to existence of a strong connection on the universal calculus and outline how to induce connections on arbitrary covariant calculi on QPBs.
\begin{definition}
A connection $\Pi\colon\Omega^1(A)\to\Omega^1(A)$ on $\Omega^1(A)$ is called \textbf{strong} if
\begin{equation}
	(\mathrm{id}-\Pi)(\mathrm{d}A)\subseteq\Omega^1(B)A.
\end{equation}
\end{definition}
The following existence theorem for strong connections is taken from \cite[Theorem 6.19 and Theorem 6.20]{BrJaMa}.
\begin{theorem}\label{thm:Hajac}
Let $A$ be a right $H$-comodule algebra and $B=A^{\mathrm{co}H}$.
The following are equivalent.
\begin{enumerate}
\item[i.)] $B\subseteq A$ is a QPB (faithfully flat Hopf--Galois extension).

\item[ii.)] There is a strong connection on the universal calculus $\Omega^1_u(A)$.

\item[iii.)] There is a $\Bbbk$-linear map $\ell\colon H\to A\otimes A$ such that, using the short notation
\begin{equation}
	\ell(h)=:h^{(1)}\otimes h^{(2)}
\end{equation}
for $h\in H$, the following equalities are satisfied for all $h\in H$.
\begin{align}
	\ell(1_H)
	&=1_A\otimes 1_A\\
	h^{(1)}(h^{(2)})_0\otimes(h^{(2)})_1
	&=1_A\otimes h\text{ (equivalently }h^{(1)}h^{(2)}=\varepsilon(h)1_A)\\
	h^{(1)}\otimes(h^{(2)})_0\otimes(h^{(2)})_1
	&=(h_1)^{(1)}\otimes(h_1)^{(2)}\otimes h_2\\
	(h^{(1)})_0\otimes(h^{(1)})_1\otimes h^{(2)}
	&=(h_2)^{(1)}\otimes S(h_1)\otimes(h_2)^{(2)}
\end{align}
\end{enumerate}
\end{theorem}
Let $B=A^{\mathrm{co}H}\subseteq A$ a QPB and consider the canonical surjection
\begin{equation}
	\pi_B\colon A\otimes A\to A\otimes_BA
\end{equation}
with its kernel $\ker\pi_B\cong A\Omega^1_u(B)A$ (see \cite[Lemma 5.7]{BegMaj}). Given a map $\ell\colon H\to A\otimes A$ as in Theorem \ref{thm:Hajac} the translation map reads $\tau=\pi_B\circ\ell$. Define $\chi'\colon A\otimes A\to A\otimes H$, $\chi'(a\otimes a'):=aa'_0\otimes a'_1$, which is clearly surjective with kernel $\ker\pi_B$. We apply Proposition \ref{prop:1:1connection} to obtain the following composition of the universal FODC.
\begin{corollary}\label{cor:decUniCal}
For every QPB $B=A^{\mathrm{co}H}\subseteq A$ there is a decomposition isomorphism
\begin{equation}
\begin{split}
	\kappa\colon\Omega^1_u(A)&\to(A\otimes H^+)\oplus\ker\pi_B\\
	a\otimes a'&\mapsto\chi'(a\otimes a')+a\otimes a'-aa'_0\ell(a'_1)
\end{split}
\end{equation}
in ${}_A\mathcal{M}^H$ with inverse
\begin{equation}
\begin{split}
	\kappa^{-1}\colon(A\otimes H^+)\oplus\ker\pi_B&\to\Omega^1_u(A)\\
	a\otimes h+a'\otimes a''&\mapsto a\ell(h)+a'\otimes a''.
\end{split}
\end{equation}
\end{corollary}

We now extend this approach from the universal calculus to arbitrary right $H$-covariant calculi on QPBs.
Recall that any right $H$-covariant FODC $\Omega^1(A)$ admits an isomorphism
\begin{equation}\label{quot:A}
\begin{split}
	\Omega^1(A)&\to\Omega^1_u(A)/N\\
	a\mathrm{d}a'&\mapsto[a\otimes a'-aa'\otimes 1],\\
	a\mathrm{d}a'&\mapsfrom[a\otimes a'],
\end{split}
\end{equation}
where $N\subseteq\Omega^1_u(A)$ is a subobject in ${}_A\mathcal{M}_A^H$.
Moreover, for any bicovariant FODC $\Omega^1(H)$ there is a right ideal $I\subseteq H^+$ with $\mathrm{Ad}(I)\subseteq I\otimes H$ such that
\begin{equation}\label{quot:H}
\begin{split}
	\Omega^1(H)&\to H\otimes H^+/I\\
	h\varpi(h')&\mapsto h\otimes[h'],\\
	h\varpi(h')&\mapsfrom h\otimes[h'].
\end{split}
\end{equation}
Assuming a QPB $B\subseteq A$ with corresponding map $\ell\colon H\to A\otimes A$ induced by Theorem \ref{thm:Hajac} one easily shows that the restriction of $\ell$ to $H^+$ maps to the kernel of the multiplication, i.e. $\ell|_{H^+}\colon H^+\to\Omega^1_u(A)$.
It is natural to ask if this map descends to the quotients $\Lambda^1\to\Omega^1(A)$. For this, we first note that the restriction of $\chi'\colon A\otimes A\to A\otimes H$ to $\Omega^1_u(A)$ is a surjective map $\chi'|_{\Omega^1_u(A)}\colon\Omega^1_u(A)\to A\otimes H^+$. Then, we obtain the following result (compare to \cite[Proposition 5.41]{BegMaj}).
\begin{lemma}\label{lem:ell}
Let $B=A^{\mathrm{co}H}\subseteq A$ be a QPB. As before, consider a right $H$-covariant FODC $\Omega^1(A)$ on $A$ and a bicovariant FODC $\Omega^1(H)$ on $H$.
Then the vertical map $\pi_v\colon\Omega^1(A)\to A\otimes\Lambda^1$ is well-defined if and only if $\chi'(N)\subseteq A\otimes I$.

Moreover, $\ell\colon H^+\to\Omega^1_u(A)$ descends to a well-defined map $\ell\colon H^+/I\to\Omega^1_u(A)/N$ if and only if $\chi'(N)=A\otimes I$.
\end{lemma}
\begin{proof}
If $\chi'(N)\subseteq A\otimes I$ then the vertical map descends from $\chi'$ (or rather the restriction of $\chi'$ to $\Omega^1_u(A)$) via the quotient procedure
\begin{equation}\label{diag:chi'}
\begin{tikzcd}
a\otimes a'
& \Omega^1_u(A) \arrow{r}{\chi'}
\arrow{d}{~}
& A\otimes H^+ \arrow{d}{~}
& aa'_0\otimes a'_1\\
{[a\otimes a']}
& \Omega^1(A) \arrow{r}{\pi_v}
& A\otimes\Lambda^1
& aa'_0\otimes\varpi(a'_1)
\end{tikzcd}
\end{equation}
and thus $\pi_v$ is well-defined. If, on the other hand, $\pi_v$ is well-defined, the diagram \eqref{diag:chi'} commutes, forcing $\chi'(N)\subseteq A\otimes I$.
For the second statement, assume first that $\chi'(N)=A\otimes I$. Then, applying $\mathrm{id}\otimes\ell$ we obtain $\ell(I)\subseteq N$ and thus $\ell$ descends to the quotient
\begin{equation}
\begin{tikzcd}
h
& H^+ \arrow{r}{\ell}
\arrow{d}{~}
& \Omega^1_u(A) \arrow{d}{~}
& h^{(1)}\otimes h^{(2)}\\
\varpi(h)
& \Lambda^1 \arrow{r}{\ell}
& \Omega^1(A)
& h^{(1)}\mathrm{d}h^{(2)}
\end{tikzcd}
\end{equation}
As before, the above diagram forces $\chi'(N)=A\otimes I$ to hold. 
\end{proof}
It turns out that the condition $\chi'(N)=A\otimes I$ forces the equality of the horizontal forms of \DJ ur\dj evi\'c and Brzezi\'nski--Majid $\mathrm{hor}^1=\Omega^1_\mathrm{hor}$. And in this case the corresponding Atiyah sequence admits a splitting, i.e. we can find a connection.
\begin{proposition}
Let $B=A^{\mathrm{co}H}\subseteq A$ be a QPB and $\Omega^\bullet(A)$ a DC on $A$ which is first order complete such that $\mathrm{hor}^1=\Omega^1_\mathrm{hor}$.
Then, the exact sequence
\begin{equation}
	\begin{tikzcd}
		0 \arrow{r}{~}
		& \mathrm{hor}^1 \arrow{r}{~}
		& \Omega^1(A) \arrow{r}{\pi_v}
		& A\otimes\Lambda^1 \arrow{r}{~} \arrow[bend left=33]{l}{\hat{\ell}}
		& 0
	\end{tikzcd}
\end{equation}
admits the splitting
\begin{equation}
\hat{\ell}\colon A\otimes\Lambda^1\to\Omega^1(A),\qquad
\hat{\ell}(a\otimes\varpi(h)):=ah^{(1)}\mathrm{d}h^{(2)}.
\end{equation}
\end{proposition}
\begin{proof}
This is a direct consequence of Lemma \ref{lem:ell} and \cite[Lemma 5.40]{BegMaj}.
\end{proof}
In particular, applying Proposition \ref{prop:1:1connection} to the previous result we obtain a decomposition of $\Omega^1(A)$ into vertical and horizontal forms, just as for the universal calculus in Corollary \ref{cor:decUniCal}.

\section{The \DJ ur\dj evi\'c braiding and graded Hopf--Galois extension}\label{C5}

In this section we are building on the concepts of quantum principal bundle and complete calculus introduced in the previous sections and elaborate on the structures they are accompanied by. It turns out that the total space algebra $A$ of any QPB $B=A^{\mathrm{co}H}\subseteq A$ is canonically braided-commutative with respect to the \DJ ur\dj evi\'c braiding. We explain in Section \ref{sec:BCA} how this braided structure naturally arises from pulling back the tensor product multiplication of $A\otimes H$ to $A\otimes_BA$ via the Hopf--Galois map. In this way, $A\otimes_BA$ becomes an associative algebra and $\chi,\tau$ can be understood as algebra morphisms. In Section \ref{sec:GradedHG} the considered objects are extended to graded maps, leading to a braiding operation on total space forms, such that $\Omega^\bullet(A)$ becomes graded braided-commutative. Furthermore, the extended Hopf--Galois map is bijective, leading to a graded Hopf--Galois extension. This can be seen as an a posteriori justification why to consider complete calculi: braidings are essential tools (for example in the context of bimodule connections) and obtaining a graded Hopf--Galois extension elegantly prolongs the original Hopf--Galois extension $B=A^{\mathrm{co}H}\subseteq A$. Our main references are the preprints \cite{DurGauge,DurHG}.

\subsection{Braided-commutative algebra}\label{sec:BCA}

The canonical map $\chi\colon A\otimes_BA\to A\otimes H$ of a Hopf--Galois extension $B=A^{\mathrm{co}H}\subseteq A$ is an isomorphism of vector spaces\footnote{$\chi$ is even an isomorphism of left $A$-modules and right $H$-comodules if endowed with the obvious left $A$-actions and the diagonal right $H$-coactions, where $H$ is endowed with the adjoint coaction $\mathrm{Ad}\colon H\to H\otimes H$, $\mathrm{Ad}(h)=h_2\otimes S(h_1)h_3$}. While $A\otimes H$ has a natural algebra structure, namely the tensor product one $(a\otimes h)(a'\otimes h')=aa'\otimes hh'$, there is no obvious product on the balanced tensor product $A\otimes_BA$. However, using the vector space isomorphism $\chi$ we induce an associative multiplication $m_{A\otimes_BA}$ on $A\otimes_BA$ from the one $m_{A\otimes H}$ on $A\otimes H$. Explicitly,
\begin{equation}\label{diag:sigma}
\begin{tikzcd}
(aa'_0\otimes a'_1)\otimes(cc'_0\otimes c'_1)
& (A\otimes H)\otimes(A\otimes H) \arrow{rr}{m_{A\otimes H}}
& & A\otimes H \arrow{d}{\chi^{-1}}
& aa'_0cc'_0\otimes a'_1c'_1\\
(a\otimes_Ba')\otimes(c\otimes_Bc')
& (A\otimes_BA)\otimes(A\otimes_BA) \arrow{u}{\chi\otimes\chi}
\arrow{rr}{m_{A\otimes_BA}}
& & A\otimes_BA
& aa'_0cc'_0\tau(a'_1c'_1)
\end{tikzcd}
\end{equation}
Namely, $m_{A\otimes_BA}\colon(A\otimes_BA)\otimes(A\otimes_BA)\to A\otimes_BA$ is given by
\begin{align*}
m_{A\otimes_BA}((a\otimes_Ba')\otimes(c\otimes_Bc'))
:&=aa'_0cc'_0\tau(a'_1c'_1)\\
&=aa'_0cc'_0(c'_1)^{\langle 1\rangle}(a'_1)^{\langle 1\rangle}\otimes_B(a'_1)^{\langle 2\rangle}(c'_1)^{\langle 2\rangle}\\
&=aa'_0c(a'_1)^{\langle 1\rangle}\otimes_B(a'_1)^{\langle 2\rangle}c'\\
&=a\sigma(a'\otimes_Bc)c',
\end{align*}
where we used \eqref{tau2} and \eqref{tau5}. Above we also invoked the $\Bbbk$-linear map $\sigma\colon A\otimes_BA\to A\otimes_BA$ defined by
\begin{equation}
	\sigma(a\otimes_Ba'):=a_0a'\tau(a_1)
	=a_0a'(a_1)^{\langle 1\rangle}\otimes_B(a_1)^{\langle 2\rangle},
\end{equation}
to which we colloquially refer to as the \textbf{\DJ ur\dj evi\'c braiding}.
\begin{proposition}[{\cite[Proposition 2.1]{DurGauge}}]\label{prop:sigma}
Given $B=A^{\mathrm{co}H}\subseteq A$ a QPB the following hold true.
\begin{enumerate}
\item[i.)] $(A\otimes_BA,m_{A\otimes_BA})$ is an associative algebra with unit element $1\otimes_B1$. The canonical map $\chi\colon A\otimes_BA\to A\otimes H$ and the translation map $\tau\colon H\to A\otimes_BA$ become algebra morphisms if $A\otimes_BA$ is endowed with $m_{A\otimes_BA}$.
	
\item[ii.)] The map $\sigma\colon A\otimes_BA\to A\otimes_BA$ is an isomorphism in ${}_B\mathcal{M}_B$. Its inverse is
\begin{equation}
	\sigma^{-1}(a\otimes_Ba'):=\tau(S^{-1}(a'_1))aa'_0
	=S^{-1}(a'_1)^{\langle 1\rangle}\otimes_BS^{-1}(a'_1)^{\langle 2\rangle}aa'_0.
\end{equation}

\item[iii.)] $\sigma$ satisfies the braid equation
\begin{equation}\label{QYB}
	(\sigma\otimes_B\mathrm{id})\circ(\mathrm{id}\otimes_B\sigma)\circ(\sigma\otimes_B\mathrm{id})
	=(\mathrm{id}\otimes_B\sigma)\circ(\sigma\otimes_B\mathrm{id})\circ(\mathrm{id}\otimes_B\sigma).
\end{equation}

\item[iv.)] $A$ is braided-commutative with respect to $\sigma$, i.e.
\begin{equation}
	m_A\circ\sigma=m_A,
\end{equation}
where $m_A$ denotes the multiplication $A\otimes_BA\to A$.

\item[v.)] 
\begin{align}
	\sigma\circ(m_A\otimes_B\mathrm{id})
	&=(\mathrm{id}\otimes_Bm_A)\circ(\sigma\otimes_B\mathrm{id})\circ(\mathrm{id}\otimes_B\sigma)\label{hexagon1}\\
	\sigma\circ(\mathrm{id}\otimes_Bm_A)
	&=(m_A\otimes_B\mathrm{id})\circ(\mathrm{id}\otimes_B\sigma)\circ(\sigma\otimes_B\mathrm{id})\label{hexagon2}
\end{align}
hold as equations $A\otimes_BA\otimes_BA\to A\otimes_BA$.
\end{enumerate}
\end{proposition}
\begin{proof}
\begin{enumerate}
\item[i.)] This is an immediate consequence of the commutativity of the diagram \eqref{diag:sigma}.
	
\item[ii.)] Observe that $\sigma$ and $\sigma^{-1}$ are well-defined on the balanced tensor product. They are $B$-bilinear thanks to \eqref{tau7}. Furthermore, they are inverse to each other, since
\begin{align*}
	\sigma(\sigma^{-1}(a\otimes_Ba'))
	&=\sigma(S^{-1}(a'_1)^{\langle 1\rangle}\otimes_BS^{-1}(a'_1)^{\langle 2\rangle}aa'_0)\\
	&=(S^{-1}(a'_1)^{\langle 1\rangle})_0S^{-1}(a'_1)^{\langle 2\rangle}aa'_0((S^{-1}(a'_1)^{\langle 1\rangle})_1)^{\langle 1\rangle}\otimes_B((S^{-1}(a'_1)^{\langle 1\rangle})_1)^{\langle 2\rangle}\\
	&\overset{\eqref{tau4}}{=}aa'_0(a'_1)^{\langle 1\rangle}\otimes_B(a'_1)^{\langle 2\rangle}\\
	&\overset{\eqref{tau5}}{=}a\otimes_Ba',
\end{align*}
and
\begin{align*}
	\sigma^{-1}(\sigma(a\otimes_Ba'))
	&=\sigma^{-1}(a_0a'(a_1)^{\langle 1\rangle}\otimes_B(a_1)^{\langle 2\rangle})\\
	&=S^{-1}(((a_1)^{\langle 2\rangle})_1)^{\langle 1\rangle}\otimes_BS^{-1}(((a_1)^{\langle 2\rangle})_1)^{\langle 2\rangle}a_0a'(a_1)^{\langle 1\rangle}((a_1)^{\langle 2\rangle})_0\\
	&\overset{\eqref{tau6}}{=}S^{-1}(a_1)^{\langle 1\rangle}\otimes_BS^{-1}(a_1)^{\langle 2\rangle}a_0a'\\
	&=S^{-1}(a_1)^{\langle 1\rangle}S^{-1}(a_1)^{\langle 2\rangle}a_0\otimes_Ba'\\
	&\overset{\eqref{tau1}}{=}a\otimes_Ba',
\end{align*}
where we used in the second to last equation that $S^{-1}(a_1)^{\langle 1\rangle}\otimes_BS^{-1}(a_1)^{\langle 2\rangle}a_0\in A\otimes_BB\subseteq A\otimes_BA$. This is the case, since
\begin{align*}
	S^{-1}(a_2)^{\langle 1\rangle}\otimes_B(S^{-1}(a_2)^{\langle 2\rangle})_0a_0\otimes(S^{-1}(a_2)^{\langle 2\rangle})_1a_1
	&\overset{\eqref{tau3}}{=}(S^{-1}(a_2)_1)^{\langle 1\rangle}\otimes_B(S^{-1}(a_2)_1)^{\langle 2\rangle}a_0\otimes S^{-1}(a_2)_2a_1\\
	&=S^{-1}(a_3)^{\langle 1\rangle}\otimes_BS^{-1}(a_3)^{\langle 2\rangle}a_0\otimes S^{-1}(a_2)a_1\\
	&=S^{-1}(a_1)^{\langle 1\rangle}\otimes_BS^{-1}(a_1)^{\langle 2\rangle}a_0\otimes 1.
\end{align*}

\item[iii.)] On elements $x,y,z\in A$ the left hand side of \eqref{QYB} reads
\allowdisplaybreaks
{\small\begin{align*}
	&(\sigma\otimes_B\mathrm{id})(x_0y(x_1)^{\langle 1\rangle}\otimes_B\sigma((x_1)^{\langle 2\rangle}\otimes_Bz))\\
	&=\sigma\big(x_0y(x_1)^{\langle 1\rangle}\otimes_B((x_1)^{\langle 2\rangle})_0z(((x_1)^{\langle 2\rangle})_1)^{\langle 1\rangle}\big)\otimes_B(((x_1)^{\langle 2\rangle})_1)^{\langle 2\rangle}\\
	&=x_0y_0((x_2)^{\langle 1\rangle})_0((x_2)^{\langle 2\rangle})_0z(((x_2)^{\langle 2\rangle})_1)^{\langle 1\rangle}\big(x_1y_1((x_2)^{\langle 1\rangle})_1\big)^{\langle 1\rangle}
	\otimes_B\big(x_1y_1((x_2)^{\langle 1\rangle})_1\big)^{\langle 2\rangle}
	\otimes_B(((x_2)^{\langle 2\rangle})_1)^{\langle 2\rangle}\\
	&\overset{\eqref{tau2}}{=}x_0y_0((x_2)^{\langle 1\rangle})_0((x_2)^{\langle 2\rangle})_0z(((x_2)^{\langle 2\rangle})_1)^{\langle 1\rangle}
	(((x_2)^{\langle 1\rangle})_1)^{\langle 1\rangle}(y_1)^{\langle 1\rangle}(x_1)^{\langle 1\rangle}
	\otimes_B(x_1)^{\langle 2\rangle}(y_1)^{\langle 2\rangle}(((x_2)^{\langle 1\rangle})_1)^{\langle 2\rangle}
	\otimes_B(((x_2)^{\langle 2\rangle})_1)^{\langle 2\rangle}\\
	&=x_0y_0((x_2)^{\langle 1\rangle})_0((x_2)^{\langle 2\rangle})_0z(((x_2)^{\langle 2\rangle})_1)^{\langle 1\rangle}
	(((x_2)^{\langle 1\rangle})_1)^{\langle 1\rangle}(y_1)^{\langle 1\rangle}(x_1)^{\langle 1\rangle}
	\otimes_B(x_1)^{\langle 2\rangle}(y_1)^{\langle 2\rangle}(((x_2)^{\langle 1\rangle})_1)^{\langle 2\rangle}
	\otimes_B(((x_2)^{\langle 2\rangle})_1)^{\langle 2\rangle}\\
	&\overset{\eqref{tau4}}{=}x_0y_0(x_3)^{\langle 1\rangle}((x_3)^{\langle 2\rangle})_0z(((x_3)^{\langle 2\rangle})_1)^{\langle 1\rangle}
	S(x_2)^{\langle 1\rangle}(y_1)^{\langle 1\rangle}(x_1)^{\langle 1\rangle}
	\otimes_B(x_1)^{\langle 2\rangle}(y_1)^{\langle 2\rangle}S(x_2)^{\langle 2\rangle}
	\otimes_B(((x_3)^{\langle 2\rangle})_1)^{\langle 2\rangle}\\
	&\overset{\eqref{tau6}}{=}x_0y_0z(x_3)^{\langle 1\rangle}
	S(x_2)^{\langle 1\rangle}(y_1)^{\langle 1\rangle}(x_1)^{\langle 1\rangle}
	\otimes_B(x_1)^{\langle 2\rangle}(y_1)^{\langle 2\rangle}S(x_2)^{\langle 2\rangle}
	\otimes_B(x_3)^{\langle 2\rangle}\\
	&\overset{\eqref{tau4}}{=}x_0y_0z((x_2)^{\langle 1\rangle})_0
	(((x_2)^{\langle 1\rangle})_1)^{\langle 1\rangle}(y_1)^{\langle 1\rangle}(x_1)^{\langle 1\rangle}
	\otimes_B(x_1)^{\langle 2\rangle}(y_1)^{\langle 2\rangle}(((x_2)^{\langle 1\rangle})_1)^{\langle 2\rangle}
	\otimes_B(x_2)^{\langle 2\rangle}\\
	&\overset{\eqref{tau5}}{=}x_0y_0z(y_1)^{\langle 1\rangle}(x_1)^{\langle 1\rangle}
	\otimes_B(x_1)^{\langle 2\rangle}(y_1)^{\langle 2\rangle}(x_2)^{\langle 1\rangle}
	\otimes_B(x_2)^{\langle 2\rangle},
\end{align*}}
while the right hand side of \eqref{QYB} reads
{\small\begin{align*}
	(\mathrm{id}\otimes_B\sigma)(\sigma(x\otimes_By_0z(y_1)^{\langle 1\rangle})\otimes_B(y_1)^{\langle 2\rangle})
	&=x_0y_0z(y_1)^{\langle 1\rangle}(x_1)^{\langle 1\rangle}\otimes_B\sigma((x_1)^{\langle 2\rangle}\otimes_B(y_1)^{\langle 2\rangle})\\
	&=x_0y_0z(y_1)^{\langle 1\rangle}(x_1)^{\langle 1\rangle}
	\otimes_B((x_1)^{\langle 2\rangle})_0(y_1)^{\langle 2\rangle}(((x_1)^{\langle 2\rangle})_1)^{\langle 1\rangle}
	\otimes_B(((x_1)^{\langle 2\rangle})_1)^{\langle 2\rangle}\\
	&\overset{\eqref{tau3}}{=}x_0y_0z(y_1)^{\langle 1\rangle}(x_1)^{\langle 1\rangle}
	\otimes_B(x_1)^{\langle 2\rangle}(y_1)^{\langle 2\rangle}(x_2)^{\langle 1\rangle}
	\otimes_B(x_2)^{\langle 2\rangle}
\end{align*}}
Thus, the two sides coincide and \eqref{QYB} is proven.

\item[iv.)] This follows immediately from \eqref{tau1}.

\item[v.)] On elements $x\otimes_By\otimes_Bz$ the left hand side of \eqref{hexagon1} reads
\begin{align*}
	\sigma(xy\otimes_Bz)
	=(xy)_0z((xy)_1)^{\langle 1\rangle}\otimes_B((xy)_1)^{\langle 2\rangle}
	\overset{\eqref{tau2}}{=}x_0y_0z(y_1)^{\langle 1\rangle}(x_1)^{\langle 1\rangle}\otimes_B(x_1)^{\langle 2\rangle}(y_1)^{\langle 2\rangle},
\end{align*}
which coincides with the right hand side of \eqref{hexagon1}, since in fact
\begin{align*}
	(\mathrm{id}\otimes_Bm_A)((\sigma\otimes_B\mathrm{id})(x\otimes_B\sigma(y\otimes_Bz)))
	&=(\mathrm{id}\otimes_Bm_A)(\sigma(x\otimes_By_0z(y_1)^{\langle 1\rangle})\otimes_B(y_1)^{\langle 2\rangle})\\
	&=x_0y_0z(y_1)^{\langle 1\rangle}(x_1)^{\langle 1\rangle}\otimes_Bm_A((x_1)^{\langle 2\rangle}\otimes_B(y_1)^{\langle 2\rangle}).
\end{align*}
Similarly, for the left hand side of \eqref{hexagon2} we obtain
$
\sigma(x\otimes_Byz)
=x_0yz(x_1)^{\langle 1\rangle}\otimes_B(x_1)^{\langle 2\rangle},
$
which coincides with the right hand side of \eqref{hexagon2}, since
\begin{align*}
	(m_A\otimes_B\mathrm{id})((\mathrm{id}\otimes_B\sigma)((\sigma(x\otimes_By)\otimes_Bz)))
	&=(m_A\otimes_B\mathrm{id})(x_0y(x_1)^{\langle 1\rangle}\otimes_B\sigma((x_1)^{\langle 2\rangle}\otimes_Bz))\\
	&=m_A(x_0y(x_1)^{\langle 1\rangle}\otimes_B((x_1)^{\langle 2\rangle})_0z(((x_1)^{\langle 2\rangle})_1)^{\langle 1\rangle})\otimes_B(((x_1)^{\langle 2\rangle})_1)^{\langle 2\rangle}\\
	&=x_0y(x_1)^{\langle 1\rangle}((x_1)^{\langle 2\rangle})_0z(((x_1)^{\langle 2\rangle})_1)^{\langle 1\rangle}\otimes_B(((x_1)^{\langle 2\rangle})_1)^{\langle 2\rangle}\\
	&\overset{\eqref{tau6}}{=}x_0yz(x_1)^{\langle 1\rangle}\otimes_B(x_1)^{\langle 2\rangle}.
\end{align*}
\end{enumerate}
This completes the proof of the proposition.
\end{proof}

\subsection{Graded Hopf--Galois extension}\label{sec:GradedHG}

Let $\Omega^\bullet(A)$ be a complete calculus on a QPB $B=A^{\mathrm{co}H}\subseteq A$. As before, we denote the base forms by $\Omega^\bullet(B)$, the bicovariant FODC on the corresponding structure Hopf algebra by $\Gamma\cong H\otimes\Lambda^1\cong H\otimes H^+/I$ and its maximal prolongation by $\Omega^\bullet(H)$. Since $\Omega^\bullet(A)$ is a DGA and $\Omega^\bullet(B)$ a differential graded subalgebra it follows that $\Omega^\bullet(A)$ is an $\Omega^\bullet(B)$-bimodule. In particular
\begin{equation}\label{GradedHGLeft}
	\Omega^\bullet(A\otimes_BA):=\Omega^\bullet(A)\otimes_{\Omega^\bullet(B)}\Omega^\bullet(A)
\end{equation}
is well-defined as an $\Omega^\bullet(A)$-bimodule with the obvious left and right actions on the first and second tensor factor. It becomes a graded $\Omega^\bullet(A)$-bimodule with degree induced by the tensor product DGA $\Omega^\bullet(A)\otimes\Omega^\bullet(A)$. In this section we show that \eqref{GradedHGLeft} is a DGA and even a DC over $A\otimes_BA$. We will show that
\begin{equation}\label{eq:chibullet}
\begin{split}
	\chi^\bullet\colon\Omega^\bullet(A\otimes_BA)&\to\Omega^\bullet(A)\otimes\Omega^\bullet(H)\\
	\omega\otimes_{\Omega^\bullet(B)}\eta&\mapsto\omega\Delta_A^\bullet(\eta)=\omega\wedge\eta_{[0]}\otimes\eta_{[1]}
\end{split}
\end{equation}
is an isomorphism of graded vector spaces, or, in other words, that $\Omega^\bullet(B)=\Omega^\bullet(A)^{\Omega^\bullet(H)}\subseteq\Omega^\bullet(A)$ is a graded Hopf--Galois extension. Then we pull back the natural tensor product DGA structure of $\Omega^\bullet(A)\otimes\Omega^\bullet(H)$ to $\Omega^\bullet(A\otimes_BA)$ using the graded vector space isomorphism \eqref{eq:chibullet}, making $\chi^\bullet$ into an isomorphism of DGAs by construction. We further obtain an extension of the \DJ ur\dj evi\'c braiding to all degrees of $\Omega^\bullet(A\otimes_BA)$. The main ingredient of this approach is the extension of the translation map to differential forms, which is obtained in the following proposition.  
\begin{proposition}\label{lem:taubullet}
	There is a well-defined graded map
	\begin{equation}\label{taubullet}
		\tau^\bullet\colon\Omega^\bullet(H)\to\Omega^\bullet(A)\otimes_{\Omega^\bullet(B)}\Omega^\bullet(A)
	\end{equation}
	extending $\tau$. On $1$-forms and $2$-forms it reads
	\begin{equation}\label{eq:tau1}
		\begin{split}
			\tau^1\colon\Omega^1(H)&\to\Omega^\bullet(A)\otimes_{\Omega^\bullet(B)}\Omega^\bullet(A)\\
			h\mathrm{d}g&\mapsto\mathrm{d}(g^{\langle 1\rangle})h^{\langle 1\rangle}\otimes_{\Omega^\bullet(B)}h^{\langle 2\rangle}g^{\langle 2\rangle}
			+g^{\langle 1\rangle}h^{\langle 1\rangle}\otimes_{\Omega^\bullet(B)}h^{\langle 2\rangle}\mathrm{d}(g^{\langle 2\rangle})
		\end{split}
	\end{equation}
	and
	\begin{equation}\label{tau2form}
		\begin{split}
			\tau^2\colon\Omega^2(H)&\to\Omega^\bullet(A)\otimes_{\Omega^\bullet(B)}\Omega^\bullet(A)\\
			h\mathrm{d}g\wedge\mathrm{d}k&\mapsto-\mathrm{d}(k^{\langle 1\rangle})\wedge\mathrm{d}(g^{\langle 1\rangle})h^{\langle 1\rangle}\otimes_{\Omega^\bullet(B)}h^{\langle 2\rangle}g^{\langle 2\rangle}k^{\langle 2\rangle}\\
			&\qquad-\mathrm{d}(k^{\langle 1\rangle})g^{\langle 1\rangle}h^{\langle 1\rangle}\otimes_{\Omega^\bullet(B)}h^{\langle 2\rangle}\mathrm{d}(g^{\langle 2\rangle})k^{\langle 2\rangle}\\
			&\qquad+k^{\langle 1\rangle}\mathrm{d}(g^{\langle 1\rangle})h^{\langle 1\rangle}\otimes_{\Omega^\bullet(B)}h^{\langle 2\rangle}g^{\langle 2\rangle}\mathrm{d}(k^{\langle 2\rangle})\\
			&\qquad+k^{\langle 1\rangle}g^{\langle 1\rangle}h^{\langle 1\rangle}\otimes_{\Omega^\bullet(B)}h^{\langle 2\rangle}\mathrm{d}(g^{\langle 2\rangle})\wedge\mathrm{d}(k^{\langle 2\rangle}),
		\end{split}
	\end{equation}
	respectively.
	For higher forms one keeps the "inside-out" structure, distributes the differentials and obtains a minus sign when any tensor leg of $\tau(g)$ and the first tensor leg of $\tau(k)$ admit a differential for some arbitrary $k$-form $h\mathrm{d}h'\wedge\ldots\wedge\mathrm{d}g\wedge\ldots\wedge\mathrm{d}k\wedge\ldots\in\Omega^k(H)$.
\end{proposition}
The longish and technical proof of Proposition \ref{lem:taubullet} is shifted to Appendix \ref{app:taubullet} for the convenience of the reader. 

We are now prepared to prove that there is a graded Hopf--Galois extension in the context of complete calculi. Note that in \cite[{Proposition 2}]{DurHG} the proof of this statement invokes the representation theory of compact quantum groups, while we give a fully algebraic proof.
\begin{theorem}\label{thm:gradedHG}
Let $\Omega^\bullet(A)$ be a complete calculus on a QPB $B=A^{\mathrm{co}H}\subseteq A$ with corresponding complete calculus $\Omega^\bullet(H)$ on the structure Hopf algebra $H$. Then
\begin{equation}
\Omega^\bullet(B)=\Omega^\bullet(A)^{\Omega^\bullet(H)}\subseteq\Omega^\bullet(A)
\end{equation}
is a graded Hopf--Galois extension, meaning that $\chi^\bullet\colon\Omega^\bullet(A)\otimes_{\Omega^\bullet(B)}\Omega^\bullet(A)\to\Omega^\bullet(A)\otimes\Omega^\bullet(H)$ is a bijection of graded vector spaces. Its inverse is given by
\begin{equation}
(\chi^\bullet)^{-1}\colon\Omega^\bullet(A)\otimes\Omega^\bullet(H)\to\Omega^\bullet(A)\otimes_{\Omega^\bullet(B)}\Omega^\bullet(A),\qquad
\omega\otimes\theta\mapsto\omega\wedge\tau^\bullet(\theta),
\end{equation}
with $\tau^\bullet\colon\Omega^\bullet(H)\to\Omega^\bullet(A)\otimes_{\Omega^\bullet(B)}\Omega^\bullet(A)$ defined in Proposition \ref{lem:taubullet}.
\end{theorem}
\begin{proof}
Since $\chi^\bullet$ and $(\chi^\bullet)^{-1}$ are left $\Omega^\bullet(A)$-linear it is sufficient to prove that $(\chi^\bullet)^{-1}(\chi^\bullet(1\otimes_{\Omega^\bullet(B)}\omega))=1\otimes_{\Omega^\bullet(B)}\omega$ and $\chi^\bullet(\tau^\bullet(\theta))=1\otimes\theta$ for all $\omega\in\Omega^k(A)$ and $\theta\in\Omega^k(H)$. For $k=0$ this is the case since $B=A^{\mathrm{co}H}\subseteq A$ is a Hopf--Galois extension. In the rest of this proof we write $\otimes_B$ instead of $\otimes_{\Omega^\bullet(B)}$ in order to ease the notation.
\begin{enumerate}
\item[$k=1$:] Let $\omega=a\mathrm{d}a'$. Then $\chi^\bullet(1\otimes_B\omega)
=\omega_{[0]}\otimes\omega_{[1]}
=a_0\mathrm{d}a'_0\otimes a_1a'_1+a_0a'_0\otimes a_1\mathrm{d}a'_1$ and with \eqref{eq:tau1} we obtain
\begin{align*}
	(\chi^\bullet)^{-1}(\chi^\bullet(1\otimes_B\omega))
	&=(\chi^\bullet)^{-1}(a_0\mathrm{d}a'_0\otimes a_1a'_1+a_0a'_0\otimes a_1\mathrm{d}a'_1)\\
	&\overset{\eqref{tau2}}{=}a_0\mathrm{d}(a'_0)(a'_1)^{\langle 1\rangle}(a_1)^{\langle 1\rangle}\otimes_B(a_1)^{\langle 2\rangle}(a'_1)^{\langle 2\rangle}
	+a_0a'_0\mathrm{d}((a'_1)^{\langle 1\rangle})(a_1)^{\langle 1\rangle}\otimes_B(a_1)^{\langle 2\rangle}(a'_1)^{\langle 2\rangle}\\
	&\quad+a_0a'_0(a'_1)^{\langle 1\rangle}(a_1)^{\langle 1\rangle}\otimes_B(a_1)^{\langle 2\rangle}\mathrm{d}((a'_1)^{\langle 2\rangle})\\
	&\overset{\eqref{tau5}}{=}a_0\mathrm{d}(a'_0(a'_1)^{\langle 1\rangle})(a_1)^{\langle 1\rangle}\otimes_B(a_1)^{\langle 2\rangle}(a'_1)^{\langle 2\rangle}
	+1\otimes_Ba\mathrm{d}a'\\
	&\overset{\eqref{tau5}}{=}1\otimes_B\omega.
\end{align*}
For $\theta=h\mathrm{d}g$ we similarly get
\begin{align*}
	\chi^\bullet(\tau^\bullet(\theta))
	&=\chi^\bullet(\mathrm{d}(g^{\langle 1\rangle})h^{\langle 1\rangle}\otimes_Bh^{\langle 2\rangle}g^{\langle 2\rangle}
	+g^{\langle 1\rangle}h^{\langle 1\rangle}\otimes_Bh^{\langle 2\rangle}\mathrm{d}(g^{\langle 2\rangle}))\\
	&=\mathrm{d}(g^{\langle 1\rangle})h^{\langle 1\rangle}(h^{\langle 2\rangle})_0(g^{\langle 2\rangle})_0\otimes(h^{\langle 2\rangle})_1(g^{\langle 2\rangle})_1
	+g^{\langle 1\rangle}h^{\langle 1\rangle}(h^{\langle 2\rangle})_0\mathrm{d}((g^{\langle 2\rangle})_0)\otimes(h^{\langle 2\rangle})_1(g^{\langle 2\rangle})_1\\
	&\quad+g^{\langle 1\rangle}h^{\langle 1\rangle}(h^{\langle 2\rangle})_0(g^{\langle 2\rangle})_0\otimes(h^{\langle 2\rangle})_1\mathrm{d}((g^{\langle 2\rangle})_1)\\
	&\overset{\eqref{tau6}}{=}\mathrm{d}(g^{\langle 1\rangle})(g^{\langle 2\rangle})_0\otimes h(g^{\langle 2\rangle})_1
	+g^{\langle 1\rangle}\mathrm{d}((g^{\langle 2\rangle})_0)\otimes h(g^{\langle 2\rangle})_1
	+1\otimes h\mathrm{d}g\\
	&\overset{\eqref{tau6}}{=}1\otimes\theta.
\end{align*}

\item[$k=2$:] Let $\omega=a\mathrm{d}a'\wedge\mathrm{d}a''$. Then 
\begin{align*}
	\chi^\bullet(1\otimes_B\omega)
	&=\omega_{[0]}\otimes\omega_{[1]}\\
	&=a_0\mathrm{d}(a'_0)\wedge\mathrm{d}(a''_0)\otimes a_1a'_1a''_1
	+a_0\mathrm{d}(a'_0)a''_0\otimes a_1a'_1\mathrm{d}(a''_1)\\
	&\quad-a_0a'_0\mathrm{d}(a''_0)\otimes a_1\mathrm{d}(a'_1)a''_1
	+a_0a'_0a''_0\otimes a_1\mathrm{d}(a'_1)\wedge\mathrm{d}(a''_1)
\end{align*}
and applying \eqref{eq:tau1} and \eqref{tau2form} we obtain
\allowdisplaybreaks
\begin{align*}
	(\chi^\bullet)^{-1}(\chi^\bullet(1\otimes_B\omega))
	&=a_0\mathrm{d}(a'_0)\wedge\mathrm{d}(a''_0)\tau(a_1a'_1a''_1)
	+a_0\mathrm{d}(a'_0)a''_0\tau^1(a_1a'_1\mathrm{d}(a''_1))\\
	&\quad-a_0a'_0\mathrm{d}(a''_0)\tau^1(a_1\mathrm{d}(a'_1)a''_1)
	+a_0a'_0a''_0\tau^2(a_1\mathrm{d}(a'_1)\wedge\mathrm{d}(a''_1))\\
	&\overset{\eqref{tau2}}{=}
	a_0\mathrm{d}(a'_0)\wedge\mathrm{d}(a''_0)(a''_1)^{\langle 1\rangle}(a'_1)^{\langle 1\rangle}(a_1)^{\langle 1\rangle}\otimes_B(a_1)^{\langle 2\rangle}(a'_1)^{\langle 2\rangle}(a''_1)^{\langle 2\rangle}\\
	&\quad+a_0\mathrm{d}(a'_0)a''_0\mathrm{d}((a''_1)^{\langle 1\rangle})(a'_1)^{\langle 1\rangle}(a_1)^{\langle 1\rangle}\otimes_B(a_1)^{\langle 2\rangle}(a'_1)^{\langle 2\rangle}(a''_1)^{\langle 2\rangle}\\
	&\quad+a_0\mathrm{d}(a'_0)a''_0(a''_1)^{\langle 1\rangle}(a'_1)^{\langle 1\rangle}(a_1)^{\langle 1\rangle}\otimes_B(a_1)^{\langle 2\rangle}(a'_1)^{\langle 2\rangle}\mathrm{d}((a''_1)^{\langle 2\rangle})\\
	&\quad-a_0a'_0\mathrm{d}(a''_0)(a''_1)^{\langle 1\rangle}\wedge\mathrm{d}((a'_1)^{\langle 1\rangle})(a_1)^{\langle 1\rangle}\otimes_B(a_1)^{\langle 2\rangle}(a'_1)^{\langle 2\rangle}(a''_1)^{\langle 2\rangle}\\
	&\quad-a_0a'_0\mathrm{d}(a''_0)(a''_1)^{\langle 1\rangle}(a'_1)^{\langle 1\rangle}(a_1)^{\langle 1\rangle}\otimes_B(a_1)^{\langle 2\rangle}\mathrm{d}((a'_1)^{\langle 2\rangle})(a''_1)^{\langle 2\rangle}\\
	&\quad-a_0a'_0a''_0\mathrm{d}((a''_1)^{\langle 1\rangle})\wedge\mathrm{d}((a'_1)^{\langle 1\rangle})(a_1)^{\langle 1\rangle}\otimes_B(a_1)^{\langle 2\rangle}(a'_1)^{\langle 2\rangle}(a''_1)^{\langle 2\rangle}\\
	&\quad-a_0a'_0a''_0\mathrm{d}((a''_1)^{\langle 1\rangle})(a'_1)^{\langle 1\rangle}(a_1)^{\langle 1\rangle}\otimes_B(a_1)^{\langle 2\rangle}\mathrm{d}((a'_1)^{\langle 2\rangle})(a''_1)^{\langle 2\rangle}\\
	&\quad+a_0a'_0a''_0(a''_1)^{\langle 1\rangle}\mathrm{d}((a'_1)^{\langle 1\rangle})(a_1)^{\langle 1\rangle}\otimes_B(a_1)^{\langle 2\rangle}(a'_1)^{\langle 2\rangle}\mathrm{d}((a''_1)^{\langle 2\rangle})\\
	&\quad+a_0a'_0a''_0(a''_1)^{\langle 1\rangle}(a'_1)^{\langle 1\rangle}(a_1)^{\langle 1\rangle}\otimes_B(a_1)^{\langle 2\rangle}\mathrm{d}((a'_1)^{\langle 2\rangle})\wedge\mathrm{d}((a''_1)^{\langle 2\rangle})\\
	&\overset{\eqref{tau5}}{=}1\otimes_B\omega,
\end{align*}
where only the last line remains and the following lines cancel each other: 1 and 2, 3 and 8, 4 and 6, 5 and 7.
On the other hand, for $\theta=h\mathrm{d}g\wedge\mathrm{d}k$ we obtain by similar arguments
\allowdisplaybreaks
\begin{align*}
	\chi^\bullet(\tau^\bullet(\theta))
	&=\chi^\bullet(-\mathrm{d}(k^{\langle 1\rangle})\wedge\mathrm{d}(g^{\langle 1\rangle})h^{\langle 1\rangle}\otimes_Bh^{\langle 2\rangle}g^{\langle 2\rangle}k^{\langle 2\rangle}\\
	&\qquad-\mathrm{d}(k^{\langle 1\rangle})g^{\langle 1\rangle}h^{\langle 1\rangle}\otimes_Bh^{\langle 2\rangle}\mathrm{d}(g^{\langle 2\rangle})k^{\langle 2\rangle}\\
	&\qquad+k^{\langle 1\rangle}\mathrm{d}(g^{\langle 1\rangle})h^{\langle 1\rangle}\otimes_Bh^{\langle 2\rangle}g^{\langle 2\rangle}\mathrm{d}(k^{\langle 2\rangle})\\
	&\qquad+k^{\langle 1\rangle}g^{\langle 1\rangle}h^{\langle 1\rangle}\otimes_Bh^{\langle 2\rangle}\mathrm{d}(g^{\langle 2\rangle})\wedge\mathrm{d}(k^{\langle 2\rangle}))\\
	&=-\mathrm{d}(k^{\langle 1\rangle})\wedge\mathrm{d}(g^{\langle 1\rangle})h^{\langle 1\rangle}(h^{\langle 2\rangle})_0(g^{\langle 2\rangle})_0(k^{\langle 2\rangle})_0\otimes(h^{\langle 2\rangle})_1(g^{\langle 2\rangle})_1(k^{\langle 2\rangle})_1\\
	&\quad-\mathrm{d}(k^{\langle 1\rangle})g^{\langle 1\rangle}h^{\langle 1\rangle}(h^{\langle 2\rangle})_0\wedge\mathrm{d}((g^{\langle 2\rangle})_0)(k^{\langle 2\rangle})_0\otimes(h^{\langle 2\rangle})_1(g^{\langle 2\rangle})_1(k^{\langle 2\rangle})_1\\
	&\quad-\mathrm{d}(k^{\langle 1\rangle})g^{\langle 1\rangle}h^{\langle 1\rangle}(h^{\langle 2\rangle})_0(g^{\langle 2\rangle})_0(k^{\langle 2\rangle})_0\otimes(h^{\langle 2\rangle})_1\mathrm{d}((g^{\langle 2\rangle})_1)(k^{\langle 2\rangle})_1\\
	&\quad+k^{\langle 1\rangle}\mathrm{d}(g^{\langle 1\rangle})h^{\langle 1\rangle}(h^{\langle 2\rangle})_0(g^{\langle 2\rangle})_0\wedge\mathrm{d}((k^{\langle 2\rangle})_0)\otimes(h^{\langle 2\rangle})_1(g^{\langle 2\rangle})_1(k^{\langle 2\rangle})_1\\
	&\quad+k^{\langle 1\rangle}\mathrm{d}(g^{\langle 1\rangle})h^{\langle 1\rangle}(h^{\langle 2\rangle})_0(g^{\langle 2\rangle})_0(k^{\langle 2\rangle})_0)\otimes(h^{\langle 2\rangle})_1(g^{\langle 2\rangle})_1\mathrm{d}((k^{\langle 2\rangle})_1)\\
	&\quad+k^{\langle 1\rangle}g^{\langle 1\rangle}h^{\langle 1\rangle}(h^{\langle 2\rangle})_0\mathrm{d}((g^{\langle 2\rangle})_0)\wedge\mathrm{d}((k^{\langle 2\rangle})_0)\otimes(h^{\langle 2\rangle})_1(g^{\langle 2\rangle})_1(k^{\langle 2\rangle})_1\\
	&\quad+k^{\langle 1\rangle}g^{\langle 1\rangle}h^{\langle 1\rangle}(h^{\langle 2\rangle})_0\mathrm{d}((g^{\langle 2\rangle})_0)(k^{\langle 2\rangle})_0\otimes(h^{\langle 2\rangle})_1(g^{\langle 2\rangle})_1\mathrm{d}((k^{\langle 2\rangle})_1)\\
	&\quad-k^{\langle 1\rangle}g^{\langle 1\rangle}h^{\langle 1\rangle}(h^{\langle 2\rangle})_0(g^{\langle 2\rangle})_0\mathrm{d}((k^{\langle 2\rangle})_0)\otimes(h^{\langle 2\rangle})_1\mathrm{d}((g^{\langle 2\rangle})_1)(k^{\langle 2\rangle})_1\\
	&\quad+k^{\langle 1\rangle}g^{\langle 1\rangle}h^{\langle 1\rangle}(h^{\langle 2\rangle})_0(g^{\langle 2\rangle})_0(k^{\langle 2\rangle})_0\otimes(h^{\langle 2\rangle})_1\mathrm{d}((g^{\langle 2\rangle})_1)\wedge\mathrm{d}((k^{\langle 2\rangle})_1)\\
	&\overset{\eqref{tau6}}{=}
	1\otimes\theta,
\end{align*}
where only the last term remains and the following terms cancel each other: 1 and 2, 3 and 8, 4 and 6, 5 and 7.

\item[$k>2$:] Invertibility for higher forms follows inductively from the same principles, making use of the inside-out structure of $\tau^\bullet$.
\end{enumerate}
\end{proof}

We would like to remark that Theorem \ref{thm:gradedHG} is an extension of the first order result in \cite[Theorem 3.23]{AFLW}.

The extension $\tau^\bullet\colon\Omega^\bullet(H)\to\Omega^\bullet(A\otimes_BA)$ of the translation map satisfies graded analogs of \eqref{tau6}-\eqref{tau7}, which reduce to the latter in zero degree.
\begin{proposition}\label{prop:gradSigma}
	For a complete calculus we have
	\begin{align}
		\theta^{\langle 1\rangle}\wedge(\theta^{\langle 2\rangle})_{[0]}\otimes(\theta^{\langle 2\rangle})_{[1]}
		&=1\otimes\theta\label{TauBul1}\\
		\omega_{[0]}\wedge(\omega_{[1]})^{\langle 1\rangle}\otimes_{\Omega^\bullet(B)}(\omega_{[1]})^{\langle 2\rangle}
		&=1\otimes_{\Omega^\bullet(B)}\omega\label{TauBul2}\\
		\tau^\bullet(\theta\wedge\xi)
		&=(-1)^{|\theta||\xi^{\langle 1\rangle}|}\xi^{\langle 1\rangle}\wedge\theta^{\langle 1\rangle}\otimes_{\Omega^\bullet(B)}\theta^{\langle 2\rangle}\wedge\xi^{\langle 2\rangle}\label{TauBul3}\\
		\theta^{\langle 1\rangle}\wedge\theta^{\langle 2\rangle}
		&=\varepsilon(\theta)1\label{TauBul4}\\
		\theta^{\langle 1\rangle}\otimes_{\Omega^\bullet(B)}(\theta^{\langle 2\rangle})_{[0]}\otimes(\theta^{\langle 2\rangle})_{[1]}
		&=(\theta_{[1]})^{\langle 1\rangle}\otimes_{\Omega^\bullet(B)}(\theta_{[1]})^{\langle 2\rangle}\otimes\theta_{[2]}\label{TauBul5}\\
		(\theta^{\langle 1\rangle})_{[0]}\otimes_{\Omega^\bullet(B)}\theta^{\langle 2\rangle}\otimes(\theta^{\langle 1\rangle})_{[1]}
		&=(\theta_{[2]})^{\langle 1\rangle}\otimes_{\Omega^\bullet(B)}(\theta_{[2]})^{\langle 2\rangle}\otimes S^\bullet(\theta_{[1]})\label{TauBul6}
	\end{align}
	for all $\theta,\xi\in\Omega^\bullet(H)$ and $\omega\in\Omega^\bullet(A)$. Moreover, the centrality property
	\begin{equation}\label{gradedcentral}
		\eta\wedge\tau^\bullet(\theta)=(-1)^{|\eta||\theta|}\tau^\bullet(\theta)\wedge\eta
	\end{equation}
	holds for all $\eta\in\Omega^\bullet(B)$ and $\theta\in\Omega^\bullet(H)$.
\end{proposition}
\begin{proof}
	We use that $\chi^\bullet(\omega\otimes_B\eta)=\omega\wedge\eta_{[0]}\otimes\eta_{[1]}$ is an isomorphism of graded vector spaces with inverse $(\chi^\bullet)^{-1}(\omega\otimes\theta)=\omega\wedge\tau^\bullet(\theta)$. Then 
	$
	\theta^{\langle 1\rangle}\wedge(\theta^{\langle 2\rangle})_{[0]}\otimes(\theta^{\langle 2\rangle})_{[1]}
	=\chi^\bullet((\chi^\bullet)^{-1}(1\otimes\theta))
	=1\otimes\theta,
	$
	which proves \eqref{TauBul1}, while
	$
	\omega_{[0]}\wedge(\omega_{[1]})^{\langle 1\rangle}\otimes_{\Omega^\bullet(B)}(\omega_{[1]})^{\langle 2\rangle}
	=(\chi^\bullet)^{-1}(\chi^\bullet(1\otimes_B\omega))=1\otimes_B\omega
	$
	shows that \eqref{TauBul2} holds. For \eqref{TauBul3} we note that
	\begin{align*}
		(-1)^{|\theta||\xi^{\langle 1\rangle}|}\chi^\bullet(\xi^{\langle 1\rangle}\wedge\theta^{\langle 1\rangle}\otimes_{\Omega^\bullet(B)}\theta^{\langle 2\rangle}\wedge\xi^{\langle 2\rangle})
		&=(-1)^{|\theta||\xi^{\langle 1\rangle}|}\xi^{\langle 1\rangle}\wedge\theta^{\langle 1\rangle}\wedge\Delta_A^\bullet(\theta^{\langle 2\rangle})\wedge_\otimes\Delta_A^\bullet(\xi^{\langle 2\rangle})\\
		&=(-1)^{|\theta||\xi^{\langle 1\rangle}|}(\xi^{\langle 1\rangle}\otimes\theta)\wedge_\otimes\Delta_A^\bullet(\xi^{\langle 2\rangle})\\
		&=(1\otimes\theta)\wedge_\otimes(\xi^{\langle 1\rangle}\otimes 1)\wedge_\otimes\Delta_A^\bullet(\xi^{\langle 2\rangle})\\
		&=1\otimes\theta\wedge\xi,
	\end{align*}
	using \eqref{TauBul1} twice, the DGA structure of $\Omega^\bullet(A)\otimes\Omega^\bullet(H)$ and that $\Delta_A^\bullet$ is a morphism of DGAs.
	Since $\chi^\bullet$ is bijective this implies that \eqref{TauBul3} holds. For \eqref{TauBul4} we apply $\mathrm{id}\otimes\varepsilon^\bullet$ to \eqref{TauBul1} and use that $\Delta_A^\bullet$ is a right $\Omega^\bullet(H)$-coaction. Next we prove \eqref{TauBul5}. Since $\Delta_A^\bullet$ is a coaction we have $\omega\wedge\eta_{[0]}\otimes(\eta_{[1]})_{[1]}\otimes(\eta_{[1]})_{[2]}=\omega\wedge(\eta_{[0]})_{[0]}\otimes(\eta_{[0]})_{[1]}\otimes\eta_{[1]}$,
	which in terms of maps reads
	$(\mathrm{id}\otimes\Delta^\bullet)\circ\chi^\bullet=(\chi^\bullet\otimes\mathrm{id})\circ(\mathrm{id}\otimes_B\Delta_A^\bullet)$. Concatenating $((\chi^\bullet)^{-1}\otimes\mathrm{id})$ from the left and $\tau^\bullet$ from the right gives
	$
	(\tau^\bullet\otimes\mathrm{id})\circ\Delta_A^\bullet=((\chi^\bullet)^{-1}\otimes\mathrm{id})\circ(1\otimes\Delta_{A}^\bullet)=(\mathrm{id}\otimes_B\Delta_A^\bullet)\circ\tau^\bullet,
	$
	which reads \eqref{TauBul5} applied to $\theta\in\Omega^\bullet(H)$.
	For \eqref{TauBul6} we note that
	\begin{equation}\label{eq:last}
		(\mathrm{id}\otimes\chi^\bullet)\circ(\Delta_{A}^{\bullet^\mathrm{op}}\otimes_{\Omega^\bullet(B)}\mathrm{id})=\nu\circ\chi^\bullet
	\end{equation}
	as maps $\Omega^\bullet(A\otimes_BA)\to\Omega^\bullet(H)\otimes\Omega^\bullet(A)\otimes\Omega^\bullet(H)$, where $\nu\colon\Omega^\bullet(A)\otimes\Omega^\bullet(H)\to\Omega^\bullet(H)\otimes\Omega^\bullet(A)\otimes\Omega^\bullet(H)$ is defined by $\nu(\omega\otimes\theta):=\omega_{[1]}\wedge S(\theta_{[1]})\otimes\omega_{[0]}\otimes\theta_{[2]}$ and $\Delta_{A}^{\bullet^\mathrm{op}}(\omega):=\omega_{[1]}\otimes\omega_{[0]}$. In fact, \eqref{eq:last} holds, since
	\begin{align*}
		\omega_{[1]}\otimes\omega_{[0]}\wedge\eta_{[0]}\otimes\eta_{[1]}
		=\omega_{[1]}\wedge\eta_{[1]}\wedge S(\eta_{[2]})\otimes\omega_{[0]}\wedge\eta_{[0]}\otimes\eta_{[3]}.
	\end{align*}
	Concatenating \eqref{eq:last} with $\tau^\bullet$ from the right and $\mathrm{id}\otimes(\chi^\bullet)^{-1}$ from the left gives
	$$
	(\Delta_{A}^{\bullet^\mathrm{op}}\otimes_{\Omega^\bullet(B)}\mathrm{id})\circ\tau^\bullet
	=(\mathrm{id}\otimes(\chi^\bullet)^{-1})\circ\nu\circ(1_A\otimes\mathrm{id})
	\colon\Omega^\bullet(H)\to\Omega^\bullet(H)\otimes\Omega^\bullet(A\otimes_BA),
	$$ 
	which reads $(\theta^{\langle 1\rangle})_{[1]}\otimes(\theta^{\langle 1\rangle})_{[0]}\otimes_{\Omega^\bullet(B)}\theta^{\langle 2\rangle}
	=S(\theta_{[1]})\otimes(\theta_{[1]})^{\langle 1\rangle}\otimes_{\Omega^\bullet(B)}(\theta_{[1]})^{\langle 2\rangle}$ on elements. A tensor flip then gives \eqref{TauBul6}. 
	Finally, for \eqref{gradedcentral} we observe that for $\eta\in\Omega^\bullet(B)$ and $\theta\in\Omega^\bullet(H)$ the term $\chi^\bullet(\eta\wedge\tau^\bullet(\theta))=\eta\otimes\theta$ coincides with
	\begin{align*}
		(-1)^{|\eta||\theta|}\chi^\bullet(\tau^\bullet\wedge\eta)
		&=(-1)^{|\eta||\theta|}\theta^{\langle 1\rangle}\wedge\Delta_A^\bullet(\theta^{\langle 2\rangle}\wedge\eta)\\
		&=(-1)^{|\eta||\theta|}(\theta^{\langle 1\rangle}\wedge(\theta^{\langle 2\rangle})_{[0]}\otimes(\theta^{\langle 2\rangle})_{[1]})\wedge_\otimes(\eta\otimes 1)\\
		&\overset{\eqref{TauBul1}}{=}(-1)^{|\eta||\theta|}(1\otimes\theta)\wedge_\otimes(\eta\otimes 1)\\
		&=\eta\otimes\theta,
	\end{align*}
	using that $\Delta_A^\bullet$ is an algebra morphism and $\Delta_A^\bullet(\eta)=\eta\otimes 1$.
	This concludes the proof since $\chi^\bullet$ is an isomorphism.
\end{proof}

Just like in Section \ref{sec:BCA} we use the isomorphism of (graded) vector spaces $\chi^\bullet\colon\Omega^\bullet(A\otimes_BA)\to\Omega^\bullet(A\otimes H)$ to induce an associative algebra structure on $\Omega^\bullet(A\otimes_BA)$ from the wedge product of $\Omega^\bullet(A\otimes H)$. Namely, we define $\wedge_{\otimes_B}\colon\Omega^\bullet(A\otimes_BA)\otimes\Omega^\bullet(A\otimes_BA)\to\Omega^\bullet(A\otimes_BA)$ via the commutative diagram
\begin{equation}\label{diag:wedgeB}
\begin{tikzcd}
\Omega^\bullet(A\otimes H)\otimes\Omega^\bullet(A\otimes H) \arrow{rr}{\wedge_\otimes}
& & \Omega^\bullet(A\otimes H) \arrow{d}{(\chi^\bullet)^{-1}}\\
\Omega^\bullet(A\otimes_BA)\otimes\Omega^\bullet(A\otimes_BA)
\arrow{u}{\chi^\bullet\otimes\chi^\bullet}
\arrow{rr}{\wedge_{\otimes_B}}
& & \Omega^\bullet(A\otimes_BA)
\end{tikzcd}
\end{equation}
Explicitly,
\begin{equation}\label{wedgeB}
\begin{split}
	(\omega\otimes_{\Omega^\bullet(B)}\omega')\wedge_{\otimes_B}(\eta\otimes_{\Omega^\bullet(B)}\eta')
	&:=(-1)^{|\omega'_{[1]}||\eta||\eta'_{[0]}|}\omega\wedge\omega'_{[0]}\wedge\eta\wedge\eta'_{[0]}\wedge\tau^\bullet(\omega'_{[1]}\wedge\eta'_{[1]})\\
	&=(-1)^{|\omega'_{[1]}||\eta|}\omega\wedge\omega'_{[0]}\wedge\eta\wedge\tau^\bullet(\omega'_{[1]})\wedge\eta'
\end{split}
\end{equation}
for all $\omega,\omega',\eta,\eta'\in\Omega^\bullet(A)$, where we used \eqref{TauBul3} and \eqref{TauBul2} in the second equation. Defining the graded map
\begin{equation}\label{sigmabullet}
\begin{split}
	\sigma^\bullet\colon\Omega^\bullet(A\otimes_BA)&\to\Omega^\bullet(A\otimes_BA)\\
	\omega\otimes_{\Omega^\bullet(B)}\eta&\mapsto\sigma^\bullet(\omega\otimes_{\Omega^\bullet(B)}\eta)
	:=(-1)^{|\omega_{[1]}||\eta|}\omega_{[0]}\wedge\eta\wedge\tau^\bullet(\omega_{[1]})
\end{split}
\end{equation}
we can rewrite the product $\wedge_{\otimes_B}$ as
$$
(\omega\otimes_{\Omega^\bullet(B)}\omega')\wedge_{\otimes_B}(\eta\otimes_{\Omega^\bullet(B)}\eta')=\omega\wedge\sigma^\bullet(\omega'\otimes_{\Omega^\bullet(B)}\eta)\wedge\eta'.
$$
We colloquially refer to \eqref{sigmabullet} as the \textbf{extended \DJ ur\dj evi\'c braiding}.
\begin{proposition}\label{prop:DGAotimesB}
The graded vector space $\Omega^\bullet(A\otimes_BA)=\Omega^\bullet(A)\otimes_{\Omega^\bullet(B)}\Omega^\bullet(A)$ is a DC on $A\otimes_BA$ with wedge product $\wedge_{\otimes_B}$ defined in \eqref{wedgeB} and differential
\begin{equation}\label{d_otimes_B}
	\mathrm{d}_{\otimes_B}(\omega\otimes_{\Omega^\bullet(B)}\eta)
	:=\mathrm{d}\omega\otimes_{\Omega^\bullet(B)}\eta+(-1)^{|\omega|}\omega\otimes_{\Omega^\bullet(B)}\mathrm{d}\eta.
\end{equation}
for all $\omega,\eta\in\Omega^\bullet(A)$.
\end{proposition}
\begin{proof}
This follows immediately from the construction: the DGA structure of $\Omega^\bullet(A\otimes_BA)$ is obtained by pulling back the DGA structure of $\Omega^\bullet(A\otimes H)$ via the isomorphism of graded vector spaces $\chi^\bullet\colon\Omega^\bullet(A\otimes_BA)\to\Omega^\bullet(A\otimes H)$. We verify that \eqref{d_otimes_B} is in fact the pullback of $\mathrm{d}_\otimes$, i.e. $\mathrm{d}_{\otimes_B}=(\chi^\bullet)^{-1}\circ\mathrm{d}_\otimes\circ\chi^\bullet$ or equivalently $\chi^\bullet\circ\mathrm{d}_{\otimes_B}=\mathrm{d}_\otimes\circ\chi^\bullet$.
For $\omega,\eta\in\Omega^\bullet(A)$ we obtain
\begin{align*}
	\mathrm{d}_\otimes(\chi^\bullet(\omega\otimes_{\Omega^\bullet(B)}\eta))
	=&\mathrm{d}_\otimes(\omega\wedge\eta_{[0]}\otimes\eta_{[1]})\\
	&=\mathrm{d}(\omega\wedge\eta_{[0]})\otimes\eta_{[1]}+(-1)^{|\omega||\eta_{[0]}|}\omega\wedge\eta_{[0]}\otimes\mathrm{d}(\eta_{[1]})\\
	&=\mathrm{d}(\omega)\wedge\eta_{[0]}\otimes\eta_{[1]}
	+(-1)^{|\omega|}\omega\wedge\mathrm{d}(\eta_{[0]})\otimes\eta_{[1]}
	+(-1)^{|\omega||\eta_{[0]}|}\omega\wedge\eta_{[0]}\otimes\mathrm{d}(\eta_{[1]}),
\end{align*}
which coincides with
\begin{align*}
	\chi^\bullet(\mathrm{d}_{\otimes_B}(\omega\otimes_{\Omega^\bullet(B)}\eta))
	&=\chi^\bullet(\mathrm{d}\omega\otimes_{\Omega^\bullet(B)}\eta+(-1)^{|\omega|}\omega\otimes_{\Omega^\bullet(B)}\mathrm{d}\eta)\\
	&=\mathrm{d}(\omega)\wedge\eta_{[0]}\otimes\eta_{[1]}
	+(-1)^{|\omega|}\omega\wedge(\mathrm{d}\eta)_{[0]}\otimes(\mathrm{d}\eta)_{[1]}\\
	&=\mathrm{d}(\omega)\wedge\eta_{[0]}\otimes\eta_{[1]}
	+(-1)^{|\omega|}\omega\wedge\mathrm{d}(\eta_{[0]})\otimes\eta_{[1]}
	+(-1)^{|\omega||\eta_{[0]}|}\omega\wedge\eta_{[0]}\otimes\mathrm{d}(\eta_{[1]}),
\end{align*}
using that $\Delta_A^\bullet$ is a DGA morphism.
\end{proof}

We also prove the graded extension of Proposition \ref{prop:sigma}.
\begin{proposition}[{\cite[Proposition 3.1]{DurGauge}}]
As before, consider a complete calculus $\Omega^\bullet(A)$ with $\Omega^\bullet(H)$ the corresponding calculus on the structure Hopf algebra.
\begin{enumerate}
\item[i.)] The canonical map $\chi^\bullet\colon\Omega^\bullet(A\otimes_BA)\to\Omega^\bullet(A\otimes H)$ and the translation map $\tau\colon\Omega^\bullet(H)\to\Omega^\bullet(A\otimes_BA)$ are DGA morphisms.

\item[ii.)] The map $\sigma^\bullet\colon\Omega^\bullet(A\otimes_BA)\to\Omega^\bullet(A\otimes_BA)$ is an isomorphism in ${}_{\Omega^\bullet(B)}\mathcal{M}_{\Omega^\bullet(B)}$. Its inverse is
\begin{equation}
	(\sigma^\bullet)^{-1}(\omega\otimes_B\eta)
	:=(-1)^{(|\omega|+|\eta_{[0]}|)|\eta_{[1]}|}\tau^\bullet((S^\bullet)^{-1}(\eta_{[1]}))\wedge\omega\wedge\eta_{[0]}.
\end{equation}

\item[iii.)] $\sigma^\bullet$ satisfies the braid equation
\begin{equation}
	(\sigma^\bullet\otimes_{\Omega^\bullet(B)}\mathrm{id})\circ(\mathrm{id}\otimes_{\Omega^\bullet(B)}\sigma^\bullet)\circ(\sigma^\bullet\otimes_{\Omega^\bullet(B)}\mathrm{id})
	=(\mathrm{id}\otimes_{\Omega^\bullet(B)}\sigma^\bullet)\circ(\sigma^\bullet\otimes_{\Omega^\bullet(B)}\mathrm{id})\circ(\mathrm{id}\otimes_{\Omega^\bullet(B)}\sigma^\bullet).
\end{equation}

\item[iv.)] $\Omega^\bullet(A)$ is graded braided-commutative with respect to $\sigma^\bullet$, i.e.
\begin{equation}
	\wedge\circ\sigma^\bullet=\wedge,
\end{equation}
where $\wedge$ denotes the wedge product $\Omega^\bullet(A)\otimes_{\Omega^\bullet(B)}\Omega^\bullet(A)\to\Omega^\bullet(A)$.

\item[v.)] 
\begin{align}
	\sigma^\bullet\circ(\wedge\otimes_{\Omega^\bullet(B)}\mathrm{id})
	&=(\mathrm{id}\otimes_{\Omega^\bullet(B)}\wedge)\circ(\sigma^\bullet\otimes_{\Omega^\bullet(B)}\mathrm{id})\circ(\mathrm{id}\otimes_{\Omega^\bullet(B)}\sigma^\bullet)\label{eq:hex1}\\
	\sigma^\bullet\circ(\mathrm{id}\otimes_{\Omega^\bullet(B)}\wedge)
	&=(\wedge\otimes_{\Omega^\bullet(B)}\mathrm{id})\circ(\mathrm{id}\otimes_{\Omega^\bullet(B)}\sigma^\bullet)\circ(\sigma^\bullet\otimes_{\Omega^\bullet(B)}\mathrm{id})\label{eq:hex2}
\end{align}
hold as equations $\Omega^\bullet(A)\otimes_{\Omega^\bullet(B)}\Omega^\bullet(A)\otimes_{\Omega^\bullet(B)}\Omega^\bullet(A)\to \Omega^\bullet(A)\otimes_{\Omega^\bullet(B)}\Omega^\bullet(A)$.
\end{enumerate}
\end{proposition}
\begin{proof}
\begin{enumerate}
\item[i.)] This is the case since the DGA structure of $\Omega^\bullet(A\otimes_BA)$ is defined as the pullback via $\chi^\bullet$ according to Proposition \ref{prop:DGAotimesB}.

\item[ii.)] Note that $S^\bullet$ is invertible since $S\colon H\to H$ is invertible by assumption. The inverse $(S^\bullet)^{-1}$ is given by the formula \eqref{gradedS} with $S$ replaced by $S^{-1}$ everywhere. This is the case since $S^\bullet\colon\Omega^\bullet(H)\to\Omega^\bullet(H^\mathrm{op})$ is a DGA morphism. Then the claim follows in complete analogy to the one of Proposition \ref{prop:sigma} ii.), taking into account the correct signs and using Proposition \ref{prop:gradSigma}.

\item[iii.)] One performs the same steps as in Proposition \ref{prop:sigma} iii.), using Proposition \ref{prop:gradSigma}.

\item[iv.)] We have $\wedge(\sigma^\bullet(\omega\otimes_{\Omega^\bullet(B)}\eta))
=(-1)^{|\omega_{[1]}||\eta|}\omega_{[0]}\wedge\eta\wedge(\omega_{[1]})^{\langle 1\rangle}\wedge(\omega_{[1]})^{\langle 2\rangle}
\overset{\eqref{TauBul4}}{=}(-1)^{|\omega_{[1]}||\eta|}\omega_{[0]}\wedge\eta\varepsilon(\omega_{[1]})
=(-1)^{0\cdot|\eta|}\omega\wedge\eta
=\omega\wedge\eta$, since $\varepsilon$ vanishes on degree $>0$ forms.

\item[v.)] One performs the same steps as in Proposition \ref{prop:sigma} v.), using Proposition \ref{prop:gradSigma}.
\end{enumerate}
\end{proof}

\section{Examples of complete calculi}\label{sec:ex}

As explained in \cite{DurII} the motivating example of a complete calculus is the de Rham calculus on a principal bundle in classical differential geometry. The corresponding (extended) \DJ ur\dj evi\'c braiding is the (graded) flip. In this section we provide explicit examples of complete calculi and (extended) \DJ ur\dj evi\'c braidings which go beyond differential geometry. They are original to this article.
Among them are the maximal prolongation of bicovariant calculi in Section \ref{sec:exBC}, the noncommutative $2$-torus with $U(1)$-symmetry in Section \ref{sec:ex2tor}, the quantum Hopf fibration in Section \ref{sec:exHopfFib} and calculi on crossed product algebras in Section \ref{ex:CP}.

\subsection{Bicovariant calculi}\label{sec:exBC}

Let $H$ be a Hopf algebra (with invertible antipode). Consider a bicovariant FODC $\Gamma$ on $H$ with maximal prolongation $\Omega^\bullet(H)$ and $\Lambda^\bullet:={}^{\mathrm{co}H}\Omega^\bullet(H)$. We have proven in Lemma \ref{lem:Hcomplete} that $\Omega^\bullet(H)$ is complete. The corresponding vertical forms $H\otimes\Lambda^\bullet$ are isomorphic to $\Omega^\bullet(H)$ according to Proposition \ref{prop:gradedYD} ii.). The horizontal forms are trivial, since for any $h^0\mathrm{d}h^1\wedge\ldots\wedge\mathrm{d}h^k\in\mathrm{hor}^k$ we have $\Delta^\bullet(h^0\mathrm{d}h^1\wedge\ldots\wedge\mathrm{d}h^k\in\mathrm{hor}^k)\in\Omega^k(H)\otimes H$, which in particular implies $h^0_1h^1_1\ldots h^k_1\otimes h^0_2\mathrm{d}h^1_2\wedge\ldots\wedge\mathrm{d}h^k_2=0$ and the latter gives $h^0\mathrm{d}h^1\wedge\ldots\wedge\mathrm{d}h^k=0$ after applying $\varepsilon\otimes\mathrm{id}$. In particular, the corresponding base forms vanish. The graded Hopf--Galois map
\begin{equation}
	\chi^\bullet\colon\Omega^\bullet(H)\otimes\Omega^\bullet(H)\to\Omega^\bullet(H)\otimes\Omega^\bullet(H),\qquad
	\omega\otimes\eta\mapsto\omega\wedge\eta_{[1]}\otimes\eta_{[2]}
\end{equation}
is invertible with inverse $(\chi^\bullet)^{-1}\colon\Omega^\bullet(H)\otimes\Omega^\bullet(H)\to\Omega^\bullet(H)\otimes\Omega^\bullet(H)$, $\omega\otimes\eta\mapsto\omega\wedge\tau^\bullet(\eta)=\omega\wedge S^\bullet(\eta_{[1]})\otimes \eta_{[2]}$. Furthermore, we write explicitly the \DJ ur\dj evi\'c braiding and its extension
\begin{equation}
\begin{split}
	\sigma\colon H\otimes H\to H\otimes H,\qquad&
	\sigma(h\otimes g)=h_1gS(h_2)\otimes h_3,\\
	\sigma^\bullet\colon\Omega^\bullet(H)\otimes\Omega^\bullet(H)\to\Omega^\bullet(H)\otimes\Omega^\bullet(H),\qquad&
	\sigma^\bullet(\omega\otimes\eta)=(-1)^{(|\omega_{[2]}|+|\omega_{[3]}|)|\eta|}\omega_{[1]}\wedge\eta\wedge S^\bullet(\omega_{[2]})\otimes\omega_{[3]}.
\end{split}
\end{equation}
One notices that $\sigma$ coincides with the left-left Yetter-Drinfel'd braiding (see \cite{SchYD}), where $H$ is endowed with the coproduct and the adjoint left action on itself, while $\sigma^\bullet$ is the corresponding Yetter-Drinfel'd braiding for the graded Hopf algebra $\Omega^\bullet(H)$.

Let us discuss an example of this situation in the following subsection.

\subsubsection*{A complete calculus on $\mathcal{O}(U(1))$}\label{ex:CZgroupalgebra}

Consider the algebra $H:=\mathcal{O}(U(1))$ generated by an invertible element $t$, namely $H=\mathbb{C}[t,t^{-1}]$. Recall that $H$ is a Hopf algebra when equipped with coproduct, counit and antipode given by $\Delta(t^n)=t^n\otimes t^n$, $\varepsilon(t^n)=1$ and $S(t^n)=t^{-n}$, where $n\in\mathbb{Z}$. 
Fix now a non-zero complex number $\mathbb{C}\ni q\neq 0,\pm1$ and define $\Gamma$ as the free left $H$-module generated by one element $\mathrm{d}t$ such that

\begin{equation}
	\mathrm{d}t\, t^n:=q^nt^n\mathrm{d}t
\end{equation}
for all $n\in\mathbb{Z}$. Define $\mathrm{d}\colon H\to\Gamma$  as
\begin{equation}
	\mathrm{d}(f(t)):=\frac{f(qt)-f(t)}{t(q-1)}\mathrm{d}t
\end{equation}
for any rational polynomial $f$ in $t$.
One easily verifies that $(\Gamma,\mathrm{d})$ is a FODC on $H$ (see \cite[Example 1.11]{BegMaj}). It is further bicovariant with $H$-coactions $\Delta_\Gamma\colon\Gamma\to\Gamma\otimes H$ and ${}_\Gamma\Delta\colon\Gamma\to H\otimes\Gamma$ determined by
\begin{equation}
	\mathrm{ver}^{1,0}(t^n\mathrm{d}t):=\Delta_\Gamma(t^n\mathrm{d}t)=t^n\mathrm{d}t\otimes t^{n+1}\qquad\text{and}\qquad
	\mathrm{ver}^{0,1}(t^n\mathrm{d}t):={}_\Gamma\Delta(t^n\mathrm{d}t)=t^{n+1}\otimes t^n\mathrm{d}t,
\end{equation}
respectively. 
Note that the maximal prolongation $\Omega^\bullet(H)$ of $(\Gamma,\mathrm{d})$ is $\Omega^k(H)=\{0\}$ for $k>1$. This is because $\mathrm{d}t\otimes_H\mathrm{d}t$ vanishes on $\Omega^2(H)$, since $\mathrm{d}(t^{-1})+q^{-1}t^{-2}\mathrm{d}t=0$ and 
$$
	\mathrm{d}(1)\otimes_H\mathrm{d}(t^{-1})+q^{-1}\mathrm{d}(t^{-2})\otimes_H\mathrm{d}t
	=-q^{-1}(t^{-1}\mathrm{d}(t)t^{-2}+t^{-2}\mathrm{d}(t)t^{-1})\otimes_H\mathrm{d}t
	=-q^{-2}(1+q^{-1})t^{-3}\mathrm{d}t\otimes_H\mathrm{d}t \,.
$$
As a consequence of the previous computations and Lemma \ref{lem:BicExt} it follows that the calculus $\Omega^\bullet(H)=H\oplus\Omega^1(H)$ on $H=\mathcal{O}(U(1))$ is complete. All higher vertical maps vanish since $\Omega^k(H)=0$ for $k>1$.

We now discuss the braiding for $H=\mathcal{O}(U(1))$. As introduced in Section \ref{sec:BCA}, given a faithfully flat Hopf--Galois extension we consider the map $\sigma\colon H\otimes H\rightarrow H\otimes H$ defined as $\sigma(h\otimes h')=h_{1}h'\tau(h_{2}).$  In the present example the translation map reads $\tau(t^n)=S((t^n)_{1})\otimes (t^n)_{2}=t^{-n}\otimes t^n$, for any $n\in\mathbb{Z}$. Therefore we find that $\sigma$ corresponds to the flip, as expected for a commutative algebra, namely
\begin{equation}
		\sigma(t^{n}\otimes t^{m})=t^{n}t^{m}\tau(t^{n}) =t^{n}t^{m}t^{-n}\otimes t^{n} = t^{m}\otimes t^{n}.
\end{equation}
As discussed before, at the level of differential forms we consider $\sigma^{\bullet}\colon\Omega^{\bullet}(H)\otimes\Omega^{\bullet}(H) \rightarrow \Omega^{\bullet}(H)\otimes\Omega^{\bullet}(H)$, $\sigma^{\bullet}(\omega\otimes\eta)=\omega_{[0]}\wedge \eta \wedge \tau^{\bullet}(\omega_{[1]})$. Explicitly
\begin{align*}
	\sigma^\bullet(t^m\otimes t^n\mathrm{d}t)
	&=q^{-m}t^n\mathrm{d}t\otimes t^m,\\
	\sigma^\bullet(t^n\mathrm{d}t\otimes t^m)
	&=\frac{q^m-1}{q^{n+1}}t^{m-1}\mathrm{d}t\otimes t^{n+1}
	+t^m\otimes t^n\mathrm{d}t,\\
	\sigma^\bullet(t^n\mathrm{d}t\otimes t^m\mathrm{d}t)
	&=-q^{-(n+1)}t^m\mathrm{d}t\otimes t^n\mathrm{d}t.
\end{align*}
Note that, eventhough $H$ is commutative, $\sigma^\bullet$ is \textit{not} symmetric, i.e. it does not square to the identity.
In the classical limit $q\to 1$ we see that $\sigma^\bullet$ becomes the graded flip.

\subsection{The noncommutative algebraic 2-torus}\label{sec:ex2tor}

We now investigate the differential calculus on the QPB given by the noncommutative 2-torus, described in the Example \ref{ex:HG} ii.).
In this case $H=\mathcal{O}(U(1))$ and we consider a bicovariant first order differential calculus $\Omega^{1}(H):=\text{span}_{H}\{\mathrm{d}t\}$ with right $H$-action (generalizing the one in the previous section) defined by
\begin{equation} \label{eq: DC on U(1)}
	\mathrm{d}t\,t = q^{-\alpha} t \mathrm{d} t, \qquad 
	\mathrm{d}t\,t^{-1}
	= q^{\alpha} t^{-1}\mathrm{d}t	\,,
\end{equation} 
with $\alpha\in \mathbb{R}$ and $q\neq 0$ not a root of unity. Note again that $\Omega^2(H)$ is trivial since the previous relation implies that $\mathrm{d}t\otimes_H \mathrm{d}t$ has to vanish on the quotient of the maximal prolongation. For the rest of this subsection we choose $\alpha=0$, which implies that $\Omega^\bullet(H)$ is the classical calculus on $H$. In the subsequent section we choose a different value of $\alpha$. 

On $A=\mathcal{O}_{\theta}(\mathbb{T}^{2})$ we define a DC by $\Omega^{1}(A):=\text{span}_{A}\{\mathrm{d}u,\mathrm{d}v\}$ and $\Omega^{2}(A):=\text{span}_{A}\{\mathrm{d}u\wedge\mathrm{d}v\}$, equipped with the relations
\begin{equation}
		\mathrm{d}u\,u=u\mathrm{d}u, \quad \mathrm{d}v\,v=v\mathrm{d}v, \quad \mathrm{d}u\,v=e^{-i\theta}v\mathrm{d}u,\quad \mathrm{d}v\,u=e^{i\theta}u\mathrm{d}v\,,
\end{equation}
and
\begin{equation}
		\mathrm{d}u\wedge \mathrm{d}v+ e^{-i\theta} \mathrm{d}v\wedge \mathrm{d}u = 0,\qquad \mathrm{d}u\wedge \mathrm{d}u=0,\qquad  \mathrm{d}v \wedge \mathrm{d}v=0.
\end{equation} 
Moreover, $\Omega^k(A)=0$ for $k>2$. It is known that $\Omega^\bullet(A)$ is right $H$-covariant and a QPB à la Brzezi\'nski--Majid, see e.g. \cite{KhaLanVS} and \cite[E5.10]{BegMaj}. Consequently, the right $H$-coaction $\Delta_{A}\colon A\rightarrow A\otimes H$ lifts to  
\begin{equation}
	\Delta_{A}^{1}=\Delta_{\Omega^{1}(A)}+\mathrm{ver}\colon\Omega^{1}(A)\rightarrow(A\otimes \Omega^{1}(H))\oplus (\Omega^{1}(A)\otimes H)\,,	
\end{equation}
and $\Delta_{\Omega^2(A)}\colon\Omega^2(A)\to\Omega^2(A)\otimes H$ is well-defined. On generators we have
\begin{equation}
\begin{split}
	\Delta_{\Omega^{1}(A)}(\mathrm{d}u)=\mathrm{d}u\otimes t,&\qquad
	\Delta_{\Omega^{1}(A)}(\mathrm{d}v)=\mathrm{d}v\otimes t^{-1},\qquad
	\Delta_{\Omega^{2}(A)}(\mathrm{d}u\wedge \mathrm{d}v)=\mathrm{d}u\wedge \mathrm{d}v\otimes 1,\\
	&\mathrm{ver}(\mathrm{d}u)=u\otimes \mathrm{d}t, \qquad \mathrm{ver}(\mathrm{d}v)=v\otimes \mathrm{d}t^{-1}
\end{split}
\end{equation}
and these expressions extend by $A$-linearity.
Moreover, one shows that the left $A$-linear extension of
\begin{equation}\label{ver11}
	\mathrm{ver}^{1,1}(\mathrm{d}u\wedge \mathrm{d}v):= \mathrm{d}(u)v\otimes t\mathrm{d}(t^{-1}) - u \mathrm{d} v\otimes \mathrm{d} (t)t^{-1}
\end{equation}
is well-defined, i.e. consistent with the right $A$-module structure and commutation relation in $\Omega^2(A)$. The explicit computations can be found in \cite{AntonioThesis}. Note that $\mathrm{ver}^{0,2}$ is the zero map, since $\Omega^2(H)=0$. This leads to the extension
\begin{equation}
	\Delta_{A}^{2}=\Delta_{\Omega^2(A)}+\mathrm{ver}^{1,1}\colon\Omega^{2}(A)\rightarrow (\Omega^{2}(A)\otimes H) \oplus (\Omega^{1}(A)\otimes \Omega^{1}(H))
\end{equation}
of $\Delta_A$ to $\Omega^2(A)$. We obtain the following result.
\begin{theorem}\label{thm:torus}
The differential calculus $\Omega^{\bullet}(A)=A\oplus \Omega^{1}(A) \oplus \Omega^{2}(A)$ on the noncommutative algebraic 2-torus is complete. Moreover,
\begin{enumerate}
\item[i.)] $\Omega^\bullet(B)=B\oplus\Omega^1(B)$ is generated by $B:=\mathcal{O}_{\theta}(\mathbb{T}^{2})^{\mathrm{co}\mathcal{O}(U(1))}$, i.e. $\Omega^\bullet(B)\subseteq\Omega^\bullet(A)$ is the pullback DC on $B$.

\item[ii.)] $\mathrm{ver}^\bullet=\mathbb{C}\oplus\Lambda^1$ with $\Lambda^1=\mathrm{span}_\mathbb{C}\{t^{-1}\mathrm{d}t\}$ and $\mathrm{hor}^\bullet=A\Omega^\bullet(B)A$.

\item[iii.)] the map
$\Lambda^1\to\Omega^1(A)$, $t^{-1}\mathrm{d}t\mapsto u^{-1}\mathrm{d}u$
gives a connection $1$-form on $\Omega^1(A)$.
\end{enumerate}
\begin{proof}
Completeness of the calculus follows from the previous computations and it remains to prove $i.)-iii.)$.
\begin{enumerate}
\item[i.)] We show that $\Omega^\bullet(B)\subseteq\Omega^\bullet(A)$ is the pullback calculus. It is straightforward to check that $\Omega^{1}(B) \supseteq B\mathrm{d}B $. The non-trivial statement is then that $\Omega^{1}(B)\subseteq B\mathrm{d}B$. A generic element in $\omega\in \Omega^{1}(B)$ can be written as $\omega= \alpha_{k\ell}u^{k}v^{\ell}\mathrm{d}u  + \beta_{mn}u^{m}v^{n}\mathrm{d}v,$ for $\alpha_{k\ell},\beta_{mn}\in \mathbb{C}$. By $A$-linearity we have that 
\begin{align*}
	\Delta_{A}^\bullet(\alpha_{k\ell}u^{k}v^{\ell}\mathrm{d}u + \beta_{mn}u^{m}v^{n}\mathrm{d}v)
	&=\alpha_{k\ell}\Delta_{A}(u^{k}v^{\ell})\Delta_{A}^{1}(\mathrm{d}u) +  \beta_{mn}\Delta_{A}(u^{m}v^{n})\Delta_{A}^{1}(\mathrm{d}v) \\
	&=\alpha_{k\ell}u^{k}v^{\ell}\mathrm{d}u \otimes t^{k-\ell +1} + \beta_{mn}u^{m}v^{n} \mathrm{d}v\otimes t^{m-n -1} \\ & +\alpha_{k\ell}u^{k}v^{\ell}u \otimes t^{k-\ell}\mathrm{d}t - \beta_{mn}u^{m}v^{n}v \otimes t^{m-n-2}\mathrm{d}t
\end{align*}
equals $(\alpha_{k\ell}u^{k}v^{\ell}\mathrm{d}u  + \beta_{mn}u^{m}v^{n}\mathrm{d}v) \otimes 1$ 
if and only if 
\begin{equation}\label{VerTorus}
	\begin{aligned} \nonumber 
		\alpha_{k\ell}u^{k}v^{\ell}\mathrm{d}u \otimes t^{k-\ell +1} + \beta_{mn}u^{m}v^{n} \mathrm{d}v\otimes t^{m-n -1} & = (\alpha_{k\ell}u^{k}v^{\ell}\mathrm{d}u  + \beta_{mn}u^{m}v^{n}\mathrm{d}v) \otimes 1, \\  \alpha_{k\ell}u^{k}v^{\ell}u \otimes t^{k-\ell}\mathrm{d}t - \beta_{mn}u^{m}v^{n}v \otimes t^{m-n-2}\mathrm{d}t & =0.
	\end{aligned}
\end{equation}  
From the first equation we obtain $k-\ell +1 =0$ (or $\beta_{mn}=0$) and $m-n-1=0$ (or $\alpha_{k\ell}=0$), which leads to \begin{equation}
	\begin{aligned}\nonumber
		\alpha_{k\ell}u^{k}v^{\ell}\mathrm{d}u \otimes t^{k-\ell +1} + \beta_{mn}u^{m}v^{n} \mathrm{d}v\otimes t^{m-n -1} = \alpha_{k,k+1}u^{k}v^{k+1}\mathrm{d}u\otimes 1 + \beta_{n,n+1} u^{n+1}v^{n} \mathrm{d}v \otimes 1 .
	\end{aligned}
\end{equation} 
The second equation reads 
\begin{equation}
	\begin{aligned} \nonumber
		\alpha_{k\ell}u^{k}v^{\ell}u \otimes t^{k-\ell}\mathrm{d}t & - \beta_{mn}u^{m}v^{n}v \otimes t^{m-n-2}\mathrm{d}t  \\ & = \alpha_{k,k+1}u^{k}v^{k+1}u \otimes t^{-1}\mathrm{d}t - \beta_{n+1,n}u^{n+1}v^{n}v\otimes t^{-1}\mathrm{d}t \\ & =(\alpha_{k,k+1}e^{i(k+1)\theta}u^{k+1}v^{k+1}- \beta_{n,n+1}u^{n+1}v^{n+1})\otimes t^{-1}\mathrm{d}t,
	\end{aligned}
\end{equation} 
which is zero if and only if \begin{equation}\nonumber\alpha_{k,k+1}e^{i(k+1)\theta}u^{k+1}v^{k+1}- \beta_{n,n+1}u^{n+1}v^{n+1}=0,\end{equation}  so we must have $n=k$, and accordingly $\alpha_{k,k+1}e^{i(k+1)\theta}=\beta_{k,k+1}$.  Therefore a general element $\omega\in \Omega^{1}(B)$ reads  \begin{equation}\begin{aligned}\nonumber\alpha_{k\ell}u^{k}v^{\ell}\mathrm{d}(u)  + \beta_{mn}u^{m}v^{n}\mathrm{d}(v) & = \alpha_{k,k+1} u^{k}v^{k+1}\mathrm{d}(u) + \beta_{k,k+1}u^{k+1}v^{k}\mathrm{d}(v) \\ & = \alpha_{k,k+1}(u^{k}v^{k}v \mathrm{d}(u)   +e^{i(k+1)\theta} u^{k}u v^{k}\mathrm{d}(v) )\\ & =  \alpha_{k,k+1}u^{k}v^{k}(v\mathrm{d}(u) + e^{i(k+1)\theta}u^{k}v^{k}\mathrm{d}(v) u e^{-i(k+1)\theta}) \\ & =\alpha_{k,k+1} u^{k}v^{k}(v\mathrm{d}(u) + \mathrm{d}(v) u) \\ & = \alpha_{k,k+1} (uv)^{k} \mathrm{d}(uv),\end{aligned}\end{equation} for any $k\in \mathbb{Z}$. Since $B:=A^{coH}=\mathrm{span}_{\mathbb{C}}\{(uv)^{k}| \ k\in \mathbb{Z}\}$ the claim of the theorem follows.  

\item[ii.)] One easily shows that the only left coinvariant forms on $H$ are scalar multiples of $t^{-1}\mathrm{d}t$. For the horizontal forms we know from \cite[E5.10]{BegMaj} that $(\Omega^1(A),\Omega^1(H))$ is a QPB à la Brzezi\'nski--Majid. Thus, by Proposition \ref{prop:compare} $ii.)$ it follows that $\mathrm{hor}^1=A\mathrm{d}(B)A$, which equals $A\Omega^1(B)A$ by $i.)$ of the current proposition. Because of \eqref{ver11} there are no higher order horizontal forms and thus the claim follows.

\item[iii.)] Using \eqref{VerTorus} it follows that the displayed map satisfies the conditions of Definition \ref{def:con}, i.e. that it is a connection $1$-form, see also \cite[E5.10]{BegMaj}.
\end{enumerate}

\end{proof} 
\end{theorem}
					
We now discuss the braiding of $A=\mathcal{O}_{\theta}(\mathbb{T}^{2})$ and the differential forms. Since the extension $B\subseteq A$ is cleft, the translation map reads $\tau(h)=j^{-1}(h_{1})\otimes_{B}j(h_{2})$, where $j:H\rightarrow A$ is the cleaving map and $j^{-1}:H\rightarrow A$ the corresponding convolution inverse.  On the generators of $\mathcal{O}(U(1))$ we have 
\begin{equation}
	\tau(t)=u^{-1}\otimes_B u, \quad \tau(t^{-1})=v^{-1}\otimes_B v.
\end{equation}
Therefore $\sigma\colon A\otimes_{B} A\rightarrow A\otimes_{B}A$ reads 
\begin{equation}\label{eq: braiding on the 2-torus, algebra level}
\begin{aligned}
	\sigma(u\otimes_Ba)& =ua\tau(t)=ua(u^{-1}\otimes_B u)=uau^{-1}\otimes_B u, \\ 	\sigma(v\otimes_B a) & = v a \tau(t^{-1}) = v a (v^{-1}\otimes_B v) = vav^{-1}\otimes_B v, 
\end{aligned}
\end{equation}
for any $a\in A$ and the above expressions easily extends to arbitrary elements of $A$. We now focus on the braiding between generators of the algebra and corresponding differential calculi. The equations \eqref{eq: braiding on the 2-torus, algebra level} immediately generalize to 
\begin{equation}\nonumber
\begin{aligned}
		\sigma(u\otimes_{\Omega^{\bullet}(B)} \omega)  = u\omega u^{-1}\otimes_{\Omega^{\bullet}(B)} u,\qquad 	
		\sigma(v\otimes_{\Omega^{\bullet}(B)} \omega) = v\omega v^{-1}\otimes_{\Omega^{\bullet}(B)} v, 
\end{aligned} 
\end{equation} 
for any $\omega\in \Omega^{\leq 2}(A)$.  In particular we find 
\begin{equation}
\begin{aligned}
		\sigma^{\bullet}(u\otimes_{\Omega^{\bullet}(B)}\mathrm{d}u) & = \mathrm{d}u\otimes_{\Omega^{\bullet}(B)} u,\qquad
		\sigma^{\bullet}(u\otimes_{\Omega^{\bullet}(B)}\mathrm{d}v) = e^{-i\theta} \mathrm{d}v\otimes_{\Omega^{\bullet}(B)} u, \\ 
		\sigma^{\bullet}(v\otimes_{\Omega^{\bullet}(B)}\mathrm{d}v) & = \mathrm{d}v\otimes_{\Omega^{\bullet}(B)} v,\qquad  
		\sigma^{\bullet}(v\otimes_{\Omega^{\bullet}(B)}\mathrm{d}u) = e^{i\theta} \mathrm{d} u \otimes_{\Omega^{\bullet}(B)} v,  
\end{aligned}
\end{equation}
and moreover 
\begin{equation}
	\begin{aligned}
		\sigma^{\bullet}(u\otimes_{\Omega^{\bullet}(B)}\mathrm{d}u\wedge\mathrm{d}v) & = u\mathrm{d}u\wedge\mathrm{d}(v)u^{-1}\otimes_{\Omega^{\bullet}(B)}u = e^{-i\theta}(\mathrm{d}u\wedge\mathrm{d}v)\otimes_{\Omega^{\bullet}(B)}u,\\ 
		\sigma^{\bullet}(v\otimes_{\Omega^{\bullet}(B)}\mathrm{d}u\wedge\mathrm{d}v) & = v\mathrm{d}u\wedge\mathrm{d}(v)v^{-1}\otimes_{\Omega^{\bullet}(B)}v = e^{i\theta}(\mathrm{d}u\wedge\mathrm{d}v)\otimes_{\Omega^{\bullet}(B)}v.  
	\end{aligned}
\end{equation}
We now consider elements of the form $\Omega^{1}(A)\otimes_{\Omega^{\bullet}(B)} (A\oplus\Omega^{1}(A)\oplus \Omega^{2}(A))$. The braiding reads 
\begin{equation}
	\begin{aligned}
		\sigma^\bullet(\mathrm{d}u\otimes_{\Omega^\bullet(B)}\omega) & = (-1)^{|\omega||\mathrm{d}u_{[1]}|}(\mathrm{d}u)_{[0]} \wedge \omega \wedge \tau^\bullet(\mathrm{d}u_{[1]}) \\ & = \mathrm{d}u\wedge \omega \wedge \tau(t) + (-1)^{|\omega|}u \wedge \omega \wedge \tau^{1}(\mathrm{d}t) \\ & = (\mathrm{d}u\wedge \omega u^{-1})\otimes_{\Omega^\bullet(B)} u + (-1)^{|\omega|}(u\omega \wedge \mathrm{d}(u^{-1}))\otimes_{\Omega^\bullet(B)} u \\ 
		& \quad + (-1)^{|\omega|}(u \omega u^{-1}) \otimes_{\Omega^\bullet(B)} \mathrm{d}u,
	\end{aligned}
\end{equation}
and similarly
\begin{equation}
	\begin{aligned}
		\sigma^\bullet(\mathrm{d}v\otimes_{\Omega^\bullet(B)} \omega) & = (\mathrm{d}v\wedge \omega v^{-1})\otimes_{\Omega^\bullet(B)} v + (-1)^{|\omega|}(v\omega \wedge \mathrm{d}(v^{-1}))\otimes_{\Omega^\bullet(B)} v\\ 
		& \quad  + (-1)^{|\omega|}(v \omega v^{-1}) \otimes_{\Omega^\bullet(B)} \mathrm{d}v.
\end{aligned}
\end{equation}
Explicitly, for generators we obtain 
\begin{equation*}
\begin{split}
	\sigma^{\bullet}(\mathrm{d}u\otimes_{\Omega^{\bullet}(B)}u)
	&=\mathrm{u}\otimes_{\Omega^{\bullet}(B)}\mathrm{d}u,\\
	\sigma^{\bullet}(\mathrm{d}u\otimes_{\Omega^{\bullet}(B)}v)
	&=e^{-i\theta}v\otimes_{\Omega^{\bullet}(B)}\mathrm{d}u,
\end{split}\qquad
\begin{split}
	\sigma^{\bullet}( \mathrm{d}v\otimes_{\Omega^{\bullet}(B)}v)
	&=v\otimes_{\Omega^{\bullet}(B)}\mathrm{d}v,\\
	\sigma^{\bullet}(\mathrm{d}v\otimes_{\Omega^{\bullet}(B)}u)
	&=e^{i\theta}u\otimes_{\Omega^{\bullet}(B)}\mathrm{d}v,
\end{split}
\end{equation*}
\begin{equation*}
\begin{split}
	\sigma^{\bullet}(\mathrm{d}u\otimes_{\Omega^{\bullet}(B)}\mathrm{d}u)
	&=-\mathrm{d}u\otimes_{\Omega^{\bullet}(B)}  \mathrm{d}u,\\
	\sigma^{\bullet}(\mathrm{d}u\otimes_{\Omega^{\bullet}(B)} \mathrm{d}v)
	&=-e^{-i\theta} \mathrm{d}v \otimes_{\Omega^{\bullet}(B)} \mathrm{d}u,
\end{split}\qquad
\begin{split}
	\sigma^{\bullet}(\mathrm{d}v\otimes_{\Omega^{\bullet}(B)}\mathrm{d}v)
	&=-\mathrm{d}v\otimes_{\Omega^{\bullet}(B)}\mathrm{d}v,\\
	\sigma^{\bullet}(\mathrm{d}v\otimes_{\Omega^{\bullet}(B)}\mathrm{d}u)
	&=-e^{i\theta}\mathrm{d}u\otimes_{\Omega^{\bullet}(B)}\mathrm{d}v
\end{split}
\end{equation*}
%\begin{equation}\nonumber
%	\begin{aligned}
%		\sigma^{\bullet}(\mathrm{d}u\otimes_{\Omega^{\bullet}(B)}\mathrm{d}u) 
%		& = \mathrm{d}u\wedge\mathrm{d}(u)u^{-1}\otimes_{\Omega^{\bullet}(B)} u - u\mathrm{d}(u)\wedge\mathrm{d}(u^{-1})\otimes_{\Omega^{\bullet}(B)} u - u\mathrm{d}(u)u^{-1}\otimes_{\Omega^{\bullet}(B)} \mathrm{d}u  \\ 
%		& = -\mathrm{d}u\otimes_{\Omega^{\bullet}(B)}  \mathrm{d}u\\ 
%		\sigma^{\bullet}(\mathrm{d}u\otimes_{\Omega^{\bullet}(B)} \mathrm{d}v) & = \mathrm{d}u\wedge\mathrm{d}(v)u^{-1}\otimes_{\Omega^{\bullet}(B)} u - u \mathrm{d}v\wedge\mathrm{d}(u^{-1})\otimes_{\Omega^{\bullet}(B)} u - u \mathrm{d}(v)u^{-1} \otimes_{\Omega^{\bullet}(B)} \mathrm{d}u  \\ 
%		& = (\mathrm{d}u\wedge \mathrm{d}v + e^{-i\theta} \mathrm{d}v\wedge\mathrm{d}u)u^{-1}\otimes_{\Omega^{\bullet}(B)} u - u \mathrm{d}(v) u^{-1} \otimes_{\Omega^{\bullet}(B)} \mathrm{d}u  \\ 
%		& = -e^{-i\theta} \mathrm{d}v \otimes_{\Omega^{\bullet}(B)} \mathrm{d}u, \\ 
%		\sigma^{\bullet}(\mathrm{d}v\otimes_{\Omega^{\bullet}(B)}\mathrm{d}v) & = -\mathrm{d}v\otimes_{\Omega^{\bullet}(B)}\mathrm{d}v, \\
%		\sigma^{\bullet}(\mathrm{d}v\otimes_{\Omega^{\bullet}(B)}\mathrm{d}u) & =  \mathrm{d}v\wedge\mathrm{d}(u)v^{-1}\otimes_{\Omega^{\bullet}(B)}v -v\mathrm{d}u\wedge\mathrm{d}(v^{-1}) \otimes_{\Omega^{\bullet}(B)}v -v\mathrm{d}(u)v^{-1}\otimes_{\Omega^{\bullet}(B)}\mathrm{d}v \\ 
%		& =  -e^{i\theta}\mathrm{d}u\otimes_{\Omega^{\bullet}(B)}\mathrm{d}v.  
%	\end{aligned}
%\end{equation}
and
\begin{equation}\nonumber
	\begin{aligned}
		\sigma^{\bullet}(\mathrm{d}u\otimes_{\Omega^{\bullet}(B)}\mathrm{d}u\wedge\mathrm{d}v) & = u(\mathrm{d}u\wedge\mathrm{d}v)u^{-1}\otimes_{\Omega^{\bullet}(B)}\mathrm{d}u = e^{-i\theta}(\mathrm{d}u\wedge\mathrm{d}v)\otimes_{\Omega^{\bullet}(B)}\mathrm{d}u, \\ 
		\sigma^{\bullet}( \mathrm{d}v\otimes_{\Omega^{\bullet}(B)}\mathrm{d}u\wedge\mathrm{d}v) & = v(\mathrm{d}u\wedge\mathrm{d}v)v^{-1}\otimes_{\Omega^{\bullet}(B)}\mathrm{d}v = e^{i\theta}(\mathrm{d}u\wedge\mathrm{d}v)\otimes_{\Omega^{\bullet}(B)}\mathrm{d}v. 	
	\end{aligned}
\end{equation}
On elements of the form $\Omega^{2}(A)\otimes_{\Omega^{\bullet}(B)}(A\oplus\Omega^{1}(A)\oplus\Omega^{2}(A))$ the braiding reads 
\begin{equation*}
\begin{aligned}
		\sigma^{\bullet}(\mathrm{d}u\wedge\mathrm{d}v\otimes_{\Omega^{\bullet}(B)}\omega) & = \mathrm{d}u\wedge\mathrm{d}v \wedge\omega \otimes_{\Omega^{\bullet}(B)}1 -(-1)^{|\omega|}u\mathrm{d}(v)\wedge \omega\wedge \tau^{\bullet}(\mathrm{d}(t)t^{-1}) \\ 
		& \quad + (-1)^{|\omega|}\mathrm{d}(u)v\wedge\omega\wedge\tau^{\bullet}(t\mathrm{d}t^{-1}) + uv \omega \wedge\tau^{\bullet}(\mathrm{d}t\wedge\mathrm{d}t^{-1}) \\
		& = \mathrm{d}u\wedge\mathrm{d}v\wedge\omega \otimes_{\Omega^{\bullet}(B)}1 -(-1)^{|\omega|}u\mathrm{d}(v)\wedge\omega \wedge \tau^{\bullet}(t^{-1}\mathrm{d}t) - (-1)^{|\omega|}\mathrm{d}(u)v\wedge\omega\wedge\tau^{\ }(t^{-1}\mathrm{d}t) \\ 
		& = \mathrm{d}u\wedge\mathrm{d}v\wedge\omega\otimes_{\Omega^{\bullet}(B)}1 - (-1)^{|\omega|}\mathrm{d}(uv)\wedge\omega \wedge \tau^{\bullet}(t^{-1}\mathrm{d}t) \\ 
		& = \mathrm{d}u\wedge \mathrm{d}v\wedge\omega\otimes_{\Omega^{\bullet}(B)}1  -(-1)^{|\omega|}\mathrm{d}(uv)\wedge\omega\wedge(\mathrm{d}(u^{-1})u\otimes_{\Omega^{\bullet}(B)}1 + 1\otimes_{\Omega^{\bullet}(B)}u^{-1}\mathrm{d}u),   
\end{aligned}
\end{equation*} 
which gives on generators
\begin{equation*}
\begin{split}
	\sigma^{\bullet}(\mathrm{d}u\wedge\mathrm{d}v\otimes_{\Omega^{\bullet}(B)}u)
	=e^{i\theta} u\otimes_{\Omega^{\bullet}(B)}\mathrm{d}u\wedge\mathrm{d}v,\qquad
	\sigma^{\bullet}(\mathrm{d}u\wedge\mathrm{d}v\otimes_{\Omega^{\bullet}(B)}v)
	=e^{-i\theta}v\otimes_{\Omega^{\bullet}(B)}\mathrm{d}u\wedge\mathrm{d}v
\end{split}
\end{equation*}
and
\begin{equation}\nonumber
\begin{aligned}
		\sigma^{\bullet}(\mathrm{d}u\wedge\mathrm{d}v\otimes_{\Omega^{\bullet}(B)}\mathrm{d}u) &
		%=\mathrm{d}(uv)\wedge\mathrm{d}(u)\otimes_{\Omega^{\bullet}(B)}u^{-1}\mathrm{d}(u)=\mathrm{d}u\otimes_{\Omega^{\bullet}(B)}\mathrm{d}(vu)u^{-1}\mathrm{d}u 
		=e^{i\theta}\mathrm{d}u\otimes_{\Omega^{\bullet}(B)}\mathrm{d}u\wedge\mathrm{d}v, \\ 
		\sigma^{\bullet}(\mathrm{d}u\wedge\mathrm{d}v\otimes_{\Omega^{\bullet}(B)}\mathrm{d}v) & 
		%= \mathrm{d}(uv)\wedge\mathrm{d}(v)\otimes_{\Omega^{\bullet}(B)}u^{-1}\mathrm{d}u = -e^{i\theta} \mathrm{d}v \otimes_{\Omega^{\bullet}(B)}\mathrm{d}(uv)\wedge u^{-1}\mathrm{d}u
		=e^{-i\theta}\mathrm{d}v\otimes_{\Omega^{\bullet}(B)}\mathrm{d}u\wedge\mathrm{d}v.
\end{aligned}
\end{equation}
Note that $\mathrm{d}u\wedge\mathrm{d}v\otimes_{\Omega^{\bullet}(B)}\mathrm{d}u\wedge\mathrm{d}v=\mathrm{d}u\wedge\mathrm{d}v\otimes_{\Omega^{\bullet}(B)}\mathrm{d}(uv)\wedge v^{-1}\mathrm{d}v
=\mathrm{d}u\wedge\mathrm{d}v\wedge\mathrm{d}(uv)\otimes_{\Omega^{\bullet}(B)}v^{-1}\mathrm{d}v
=0$, i.e. the highest degree of $\Omega^\bullet(A\otimes_BA)$ is $3$.
\begin{remark}
According to the calculations above we conclude that $\sigma\colon A\otimes_BA\to A\otimes_BA$ squares to the identity in this example. Furthermore, $\sigma^{\bullet}\colon\Omega^{\bullet}(A)\otimes_{\Omega^{\bullet}(B)}\Omega^{\bullet}(A)\rightarrow\Omega^{\bullet}(A)\otimes_{\Omega^{\bullet}(B)}\Omega^{\bullet}(A)$ squares to the identity on generators and maps generators to generators. Thus, $\sigma^{\bullet}\colon\Omega^{\bullet}(A)\otimes_{\Omega^{\bullet}(B)}\Omega^{\bullet}(A)\rightarrow\Omega^{\bullet}(A)\otimes_{\Omega^{\bullet}(B)}\Omega^{\bullet}(A)$ squares to the identity by  \eqref{eq:hex1} and \eqref{eq:hex2}.
\end{remark}

\subsection{Quantum Hopf fibration and the Podleś sphere}\label{sec:exHopfFib}
					
Let $H=\mathcal{O}(U(1))$ and $A=\mathcal{O}_{q}(\mathrm{SL}_2(\mathbb{C}))$ be as in the Example \ref{ex:HG} iii.). The subalgebra of coinvariant elements $B$ is the Podleś sphere and we already pointed that $B\subseteq A$ is a QPB. We choose $\Omega^\bullet(H)=H\oplus\Omega^1(H)$ with right $H$ action \ref{eq: DC on U(1)} for $\alpha=2$. 
We further indicate by $|\cdot|$ the degree of an element in $A$ that is defined by the $H$-coaction on a general element  $f\in \{\alpha,\beta,\gamma,\delta\},$ by  $\Delta_{A}(f)=f\otimes t^{|f|}$. So $|\alpha|=|\gamma|=1$, whereas $|\beta|=|\delta|=-1$. We define the FODC $\Omega^{1}(A)$ as the free left $A$-module generated by 
\begin{equation}
	e^{+}=q^{-1}\alpha\mathrm{d}\gamma-q^{-2}\gamma\mathrm{d} (\alpha), \quad e^{-}=\delta\mathrm{d} (\beta) - q\beta \mathrm{d} (\delta), \quad e^{0}=\delta\mathrm{d} (\alpha)-q\beta \mathrm{d} (\gamma),
\end{equation} 
with commutation relations 
\begin{equation}
	e^{\pm}f=q^{|f|}fe^{\pm},\quad e^{0}f=q^{2|f|}fe^{0}\,.
\end{equation} 
It is well-known that $\Omega^1(A)$ is right $H$-covariant and that the vertical map is $\mathrm{ver}\colon\Omega^1(A)\to A\otimes\Omega^1(H)$ is well-defined, see e.g. \cite[Example 5.51]{BegMaj}. Thus $\Delta_{A}\colon A\rightarrow A\otimes H$ extends to a well-defined map  
\begin{equation}
\nonumber \Delta_{A}^{1}=\Delta_{\Omega^{1}(A)}+\mathrm{ver} :\Omega^{1}(A)\rightarrow (\Omega^{1}(A)\otimes H)\oplus(A\otimes \Omega^{1}(H)),
\end{equation} 
which reads 
\begin{equation}\label{VerSL}
\begin{split}
	\Delta_{\Omega^{1}(A)}(e^{\pm})&=e^{\pm}\otimes t^{\pm 2},\\
	\mathrm{ver}(e^{0})&=1\otimes t^{-1}\mathrm{d}(t),
\end{split}\qquad
\begin{split}
	\Delta_{\Omega^{1}(A)}(e^{0})&=e^{0}\otimes 1,\\
	\mathrm{ver}(e^{\pm})&=0
\end{split}
\end{equation}
on generators and extends by left $A$-linearity.
The second order differential calculus $(\Omega^{2}(A),\wedge,\mathrm{d})$ on $A$ is the free left $A$-module generated as $\Omega^{2}(A)=\mathrm{span}_{A}\{ e^{\pm}\wedge e^{0},e^{+}\wedge e^{-}\}$ with commutation relations \cite[Example 2.32]{BegMaj}
\begin{equation}
\begin{aligned}\label{ord2}
	e^{+}\wedge e^{-}& =-q^{-2}e^{-}\wedge e^{+}, \quad e^{\pm}\wedge e^{0}=-q^{\mp 4}e^{0}\wedge e^{\pm}, \quad \mathrm{d} (e^{0})=q^{3}e^{+}\wedge e^{-}, \\ &  \mathrm{d} (e^{\pm})=\mp q^{\pm 2}[2]_{q^{-2}}e^{\pm}\wedge e^{0}, \quad e^{\pm}\wedge e^{\pm}=e^{0}\wedge e^{0}=0,
\end{aligned}
\end{equation}  
where $[2]_{q}=(1-q^{2})/(1-q)$. On generators the right $H$-coaction on $\Omega^2(A)$ reads
\begin{equation}
\begin{aligned} 
	\Delta_{\Omega^{2}(A)}(e^{+}\wedge e^{-})& := e^{+}\wedge e^{-}\otimes 1, \quad \Delta_{\Omega^{2}(A)}(e^{+}\wedge e^{0}) := e^{+}\wedge e^{0}\otimes t^{2}, \\ & \Delta_{\Omega^{2}(A)}(e^{-}\wedge e^{0}):= e^{-}\wedge e^{0}\otimes t^{-2}\,.
\end{aligned}
\end{equation} 
We define the vertical map $\mathrm{ver}^{1,1}\colon\Omega^2(A)\to\Omega^1(A)\otimes\Omega^1(H)$ on generators by
\begin{equation}\label{eq:ver11} 
\begin{aligned}
	\mathrm{ver}^{1,1}(e^{+}\wedge e^{-}):=0, \quad \mathrm{ver}^{1,1}(e^{\pm}\wedge e^{0}):= e^{\pm}\otimes t^{\pm 2 -1}\mathrm{d} (t)\,
\end{aligned}
\end{equation}
and extend it by left $A$-linearity, while we are forced to set $\mathrm{ver}^{0,2}$ to zero since $\Omega^2(H)=0$. These maps are well-defined, since they respect the relations \eqref{ord2}, as one easily verifies, leading to
\begin{equation}
	\Delta^2_A=\Delta_{\Omega^2(A)}+\mathrm{ver}^{1,1}\colon\Omega^2(A)\to(\Omega^2(A)\otimes H)\oplus(\Omega^1(A)\oplus\Omega^1(H)).
\end{equation}
The third order differential calculus $\Omega^{3}(A)$ is defined as the free left $A$-module $\Omega^{3}(A)=\mathrm{span}_{A}\{e^{+}\wedge e^{-}\wedge e^{0}\},$ with relations \eqref{ord2} and 
\begin{equation}\label{rel3}
	\mathrm{d} (e^{+}\wedge e^{-})=-q^{-2}\mathrm{d} (e^{-}\wedge e^{+}), \quad \mathrm{d} (e^{\pm}\wedge e^{0}) =-q^{\mp 4}\mathrm{d}(e^{0}\wedge e^{\pm}).
\end{equation} 
The coaction is the $A$-linear extension of $\Delta_{\Omega^{3}(A)}(e^{+}\wedge e^{-}\wedge e^{0})= e^{+}\wedge e^{-}\wedge e^{0}\otimes 1$.
For degree reasons $\mathrm{ver}^{0,3}$ and $\mathrm{ver}^{1,2}$ are trivial, while the left $A$-linear extension of
\begin{equation}\label{ver21}
	\mathrm{ver}^{2,1}(e^{+}\wedge e^{-}\wedge e^{0}) := (e^{+}\wedge e^{-}\wedge e^{0}) \otimes t^{-1}\mathrm{d}(t)\,,
\end{equation}									
induces a well-defined vertical map $\mathrm{ver}^{2,1}\colon\Omega^3(A)\to\Omega^2(A)\otimes\Omega^1(H)$, since relations \eqref{rel3} are respected. Thus, leading to
\begin{equation}
	\Delta^3_A=\Delta_{\Omega^3(A)}+\mathrm{ver}^{2,1}\colon\Omega^3(A)\to(\Omega^3(A)\otimes H)\oplus(\Omega^2(A)\oplus\Omega^1(H)).
\end{equation}
More details can be found in the thesis \cite{AntonioThesis}.					
\begin{theorem} 
The DC											$\Omega^{\bullet}(A):= A \oplus \Omega^{1}(A) \oplus \Omega^{2}(A) \oplus \Omega^{3}(A) $ on $A=\mathcal{O}_{q}(\mathrm{SL}_2(\mathbb{C}))$ is a complete. Moreover,
\begin{enumerate}
\item[i.)] $\Omega^{\bullet}(B)\subseteq\Omega^\bullet(A)$ is the pullback DC generated by $B:=\mathcal{O}_{q}(\mathrm{SL}_2(\mathbb{C}))^{\mathrm{co}\mathcal{O}_{q}(U(1))}$. 

\item[ii.)] $\mathrm{ver}^\bullet=\mathbb{C}\oplus\Lambda^1$, where $\Lambda^1=\mathrm{span}_\mathbb{C}\{t^{-1}\mathrm{d}t\}$ and $\mathrm{hor}^\bullet=A\Omega^\bullet(B)A$.

\item[iii.)] the map $\Lambda^1\to\Omega^1(A)$, $t^{-1}\mathrm{d}t\mapsto e^0$ gives a connection $1$-form on $\Omega^1(A)$.
\end{enumerate} 
\begin{proof}  
Completeness of the calculus follows from the previous computations and constructions. It remains to prove $i.)-iii.)$.
\begin{enumerate}
\item[i.)] We show that the base calculus coincides with the usual pullback calculus. Let us consider a generic element $\omega\in \Omega^{1}(B)$, which we write as
\begin{equation} 
	\omega= ae^{+}+be^{-}+ce^{0}, \quad a,b,c\in A,
\end{equation}
and since by assumption $\Delta_{A}^{1}(\omega)=\omega \otimes 1$ we find 
\begin{equation}
	\begin{aligned} \nonumber
		\Delta_{A}^{1}(ae^{+}+be^{-}+ce^{0}) & = \Delta_{A}(a)\Delta_{A}^1(e^{+}) + \Delta_{A}(b)\Delta_{A}^1(e^{-})+\Delta_{A}(c)\Delta_{A}^1(e^{0}) \\ & = (a\otimes t^{|a|})(e^{+}\otimes t^{2}) + (b\otimes t^{|b|})(e^{-}\otimes t^{-2}) \\ & \quad + (c\otimes t^{|c|})(e^{0}\otimes 1 +1\otimes t^{-1}\mathrm{d} t) \\ & = ae^{+}\otimes t^{|a|}t^{2} + be^{-}\otimes t^{-2}t^{|b|} + ce^{+}\otimes t^{|c|} + c\otimes t^{|c|}t^{-1}\mathrm{d} (t)  \\ & = (a e^{+} + b e^{-} + ce^{0})\otimes 1, 
	\end{aligned}
\end{equation} 
where the last equality holds true if and only if we have one the following combinations 
\begin{equation}
	\begin{aligned}\nonumber
		|a|=-2,|b|=2,c=0,\quad  a=c=0,|b|=2,\quad  b=c=0,|a|=-2, \quad a=b=c=0.														
	\end{aligned}
\end{equation}
Accordingly, every element in $\Omega^{1}(B)$ must be of the form $\omega=a e^{+} + b e^{-}$ with the stated prescriptions. 
We list the possible elements of degree $2,-2$ in terms of generators: 
$$ 
\text{degree 2}: \alpha^{2},\gamma^{2},\alpha\gamma; \quad \text{degree -2}:\delta^{2},\beta^{2},\beta\gamma.
$$  
We now show that every possible element of the form $a e^{+}+b e^{-}$ with $a$ of degree $-2$ and $b$ of degree $2$ gives rise to an element of the form $B\mathrm{d}B$. For elements of the form $be^{-}$ we find 
\begin{equation} \nonumber
	\begin{aligned} 
		\alpha^{2} e^{-} & = \alpha (q^{2}\beta \mathrm{d} (x) - q^{-1} \delta \mathrm{d} (z)) = q^{2}\alpha \beta \mathrm{d} (x) - q^{-2} \alpha\delta \mathrm{d} (z), \\  \gamma^{2}e^{-} & = \gamma (-\delta \mathrm{d} (x) - q \beta \mathrm{d} (z)) = - \gamma\delta \mathrm{d} (x)- q \gamma \beta \mathrm{d} (z),  \\ 
		\alpha\gamma e^{-} & = \alpha(-\delta \mathrm{d} (x)- q \beta \mathrm{d} (z)) = -\alpha \delta \mathrm{d} (x) - q \alpha \beta \mathrm{d} (z).
	\end{aligned}
\end{equation} 
Each of the above expressions defines elements of the form $B\mathrm{d} B$ as we are always pairing degree $1$ and degree $-1$ elements. By the same reasoning we describe elements of the form $ae^{+}$, for which 
\begin{equation}
	\begin{aligned} \nonumber
		\beta^{2}e^{+} & = \beta (q^{-2} \gamma \mathrm{d} (z) - q \alpha \mathrm{d} (x)) = q^{-2}\beta \gamma \mathrm{d} (z)- q \beta \alpha \mathrm{d} (x), \\ 
		\delta^{2}e^{+} & = \delta (\alpha \mathrm{d} (z) + q^{-1}\gamma \mathrm{d} (x)) = \delta \alpha \mathrm{d} (z) + q^{-1}\delta \gamma \mathrm{d} (x), \\
		\delta \beta e^{+} & = \delta (q^{-2}\gamma \mathrm{d} (z) -q\alpha \mathrm{d} (x)) = q^{-2} \delta \gamma \mathrm{d} (z) - q \delta \alpha \mathrm{d} (x). 
	\end{aligned}
\end{equation} 
Moreover, in \cite{BegMaj} (Proposition 2.35, page 113) it is stated that the volume form can be expressed in terms of elements in $B$ and $\Omega^{1}(B)$. Therefore $\Omega^{\bullet}(B)\subseteq\Omega^\bullet(A)$ is a DC.

\item[ii.)] The equality for horizontal $1$-forms follows via the same argumentation used in Theorem \ref{thm:torus} $ii.)$, using the result that $(\Omega^1(A),\Omega^1(H))$ is a QPB à la Brzezi\'nski--Majid according to \cite[Example 5.51]{BegMaj}.
For horizontal $2$-forms the equality holds as well, using \eqref{eq:ver11} and the fact that $e^+\wedge e^-$ is a base form.
There are no horizontal $3$-forms by \eqref{ver21}.

\item[iii.)] This is immediate, using \eqref{VerSL}.
\end{enumerate}
\end{proof}
\end{theorem}
We discuss the braiding of algebra elements and differential forms. In \cite[Example 6.26]{BrJaMa} an explicit formula for the strong connection $\ell\colon H\rightarrow A\otimes A$ is provided. By composing with the projection to the balanced tensor product over $B$ we get the translation map $\tau:=\pi_{B}\circ \ell\colon H \rightarrow A\otimes_{B}A$, reading \begin{equation}\label{translation map sphere}
	\begin{aligned}
		\tau(t^{n})& =\sum_{k=0}^{n}\begin{pmatrix}
			n \\ k 
		\end{pmatrix}_{q^{2}}(-1)^{k}q^{k}\beta^{k}\delta^{n-k}\otimes_{B} \alpha^{n-k}\gamma^{k},  \\\tau(t^{-n})& =\sum_{k=0}^{n}
		\begin{pmatrix}
			n \\ k 
		\end{pmatrix}_{q^{2}}
		(-1)^{-k}q^{-k}\alpha^{n-k}\gamma^{k}\otimes_{B} \beta^{k}\delta^{n-k}, 
	\end{aligned} 
\end{equation} 
where 
\begin{equation}
	\begin{pmatrix}
		n \\ k 
	\end{pmatrix}_{q^{2}}=\frac{(q^{2n}-1)(q^{2n-2}-1)\ldots(q^{2k+2}-1)}{(q^{2n-2k}-1)(q^{2n-2k-2}-1)\ldots(q^{2}-1)}
\end{equation}
are the deformed binomial coefficients.
In particular, we have 
\begin{equation}
	\begin{aligned}
		\tau(t)  =  \delta\otimes_{B}\alpha - q\beta\otimes_{B} \gamma, \quad  
		\tau(t^{-1})  = \alpha\otimes_{B}\delta- q^{-1}\gamma\otimes_{B}\beta.
	\end{aligned}
\end{equation}
Therefore we find, for every $h\in \Omega^{\leq 3}(A)=A\oplus\Omega^{1}(A)\oplus\Omega^{2}(A)\oplus\Omega^{3}(A)$, $f\in \{\alpha,\gamma\}$ and $g\in \{\beta,\delta\}$  \begin{equation}
	\begin{aligned}
		\sigma(f\otimes_{B} h)& = fh(\delta\otimes_{B} \alpha-q\beta\otimes_{B}\gamma), \\ \sigma(g\otimes_{B}h)&=gh(\alpha\otimes_{B}\delta- q^{-1}\gamma\otimes_{B}\beta).
	\end{aligned}
\end{equation} 
We provide the explicit expressions of the above for the generators of the algebra and the corresponding differential calculus.
At the algebra level we find that the braiding mimics the commutation relations of the generators, namely 
\begin{equation}
	\begin{aligned}
		&\sigma(\alpha\otimes_{B} \alpha)=\alpha\otimes_{B} \alpha,  &&\sigma( \alpha\otimes_{B} \beta) =q^{-1}\beta\otimes_{B} \alpha, \\
		&\sigma(\alpha\otimes_{B}\gamma)=q^{-1}\gamma\otimes_{B} \alpha, &&\sigma(\alpha\otimes_{B}\delta)=\delta\otimes_{B}\alpha +(q^{-1}-q)\beta\otimes_{B} \gamma,\\
		&\sigma(\beta\otimes_{B}\beta)=\beta\otimes_{B}\beta, 
		&&\sigma(\beta\otimes_{B}\gamma)=\gamma\otimes_{B}\beta, \\ 
		&\sigma(\beta\otimes_{B}\delta)=q^{-1}\delta\otimes_{B}\beta,
		&&\sigma(\gamma\otimes_{B}\gamma)=\gamma\otimes_{B} \gamma, \\
		&\sigma(\gamma\otimes_{B} \delta)=q^{-1}\delta\otimes_{B} \gamma, &&\sigma(\delta\otimes_{B} \delta) =\delta\otimes_{B}\delta. 
	\end{aligned}
\end{equation}
Next we consider the braiding for differential 1-forms and elements of the algebra.
Let $f\in \{\alpha,\beta,\gamma,\delta\}$. For elements of the form $A\otimes_{\Omega^{\bullet}(B)}\Omega^{1}(A)$ we have  
\begin{equation}
	\begin{aligned}
		\sigma^{\bullet}(f\otimes_{\Omega^{\bullet}(B)} e^{0}) & = q^{-2|f|}e^{0}\otimes_{\Omega^{\bullet}(B)} f, \quad \sigma^{\bullet}(f\otimes_{\Omega^{\bullet}(B)} e^{\pm})& = q^{-f}e^{\pm}\otimes_{\Omega^{\bullet}(B)} f.
	\end{aligned}
\end{equation}
On the other hand, for elements the form $\Omega^{1}(A)\otimes_{\Omega^{\bullet}(B)} A$, we have
\begin{equation}
	\begin{aligned}\nonumber
		\sigma^{\bullet}(e^{0}\otimes_{\Omega^{\bullet}(B)} f) & = e^{0}f\otimes_{\Omega^{\bullet}(B)} 1 + f \tau^{\bullet}(t^{-1}\mathrm{d}t) 
		= e^{0}f + f\otimes_{\Omega^{\bullet}(B)} e^{0}-f e^{0}\otimes_{\Omega^{\bullet}(B)} 1 \\ 
		& = (q^{2|f|}-1)fe^{0}\otimes_{\Omega^{\bullet}(B)} 1 +f\otimes_{\Omega^{\bullet}(B)} e^{0},  
	\end{aligned}
\end{equation}
where we used that $\tau^{\bullet}(t^{-1}\mathrm{d}t) = 1\otimes e^{0}-e^{0}\otimes 1$. The other combinations read 
\begin{equation}
	\sigma^{\bullet}(e^{\pm}\otimes_{\Omega^{\bullet}(B)} f) = q^{|f|}f\otimes_{\Omega^{\bullet}(B)} e^{\pm}.
\end{equation} 
For the braiding of differential $1$-forms on both tensor factors, namely elements of the form $\Omega^{1}(A)\otimes_{\Omega^{\bullet}(B)} \Omega^{1}(A)$, we find 
\begin{equation*}
	\begin{aligned}
		&\sigma^{\bullet}(e^{+}\otimes_{\Omega^{\bullet}(B)} e^{+}) =\sigma^{\bullet}(e^{-}\otimes_{\Omega^{\bullet}(B)} e^{-}) =0, \quad &&\sigma^{\bullet}(e^{0}\otimes_{\Omega^{\bullet}(B)} e^{0})=-e^{0}\otimes_{\Omega^{\bullet}(B)} e^{0}, \\ 
		&\sigma^{\bullet}(e^{+}\otimes_{\Omega^{\bullet}(B)} e^{-})  = -q^{-2}e^{-}\otimes_{\Omega^{\bullet}(B)} e^{+}, \quad &&\sigma^{\bullet}(e^{-}\otimes_{\Omega^{\bullet}(B)} e^{+}) = -q^{2}e^{+}\otimes_{\Omega^{\bullet}(B)} e^{-}, \\ 
		&\sigma^{\bullet}(e^{\pm}\otimes_{\Omega^{\bullet}(B)} e^{0})  = -q^{\mp 4}e^{0}\otimes_{\Omega^{\bullet}(B)} e^{\pm}, \quad &&\sigma^{\bullet}(e^{0}\otimes_{\Omega^{\bullet}(B)} e^{\pm}) = -e^{\pm}\otimes_{\Omega^{\bullet}(B)} e^{0}+ (1-q^{\mp 4})e^{0}\wedge e^{\pm}\otimes_{\Omega^{\bullet}(B)} 1  
	\end{aligned}
\end{equation*}
Next, we consider the braiding of differential $2$-forms and differential $1$-forms. For elements of the form $\Omega^{1}(A)\otimes_{\Omega^{\bullet}(B)}\Omega^{2}(A)$ we find 
\begin{equation*}
\begin{split}
	\sigma^{\bullet}(e^{\pm}\otimes_{\Omega^{\bullet}(B)}(e^{\pm}\wedge e^{0}))& =0, \\
	\sigma^{\bullet}(e^{-}\otimes_{\Omega^{\bullet}(B)}(e^{+}\wedge e^{0}))& = q^{2}(e^{+}\wedge e^{0})\otimes_{\Omega^{\bullet}(B)} e^{-},\\
	\sigma^{\bullet}(e^{0}\otimes_{\Omega^{\bullet}(B)} (e^{-}\wedge e^{0}))& =(e^{-}\wedge e^{0})\otimes_{\Omega^{\bullet}(B)} e^{0},
\end{split}\quad
\begin{split}
	\sigma^{\bullet}(e^{+}\otimes_{\Omega^{\bullet}(B)}(e^{-}\wedge e^{0}))& = q^{-2}(e^{-}\wedge e^{0})\otimes_{\Omega^{\bullet}(B)} e^{+},\\
	\sigma^{\bullet}(e^{0}\otimes_{\Omega^{\bullet}(B)} (e^{+}\wedge e^{0}))& =(e^{+}\wedge e^{0})\otimes_{\Omega^{\bullet}(B)} e^{0}, \\
	\sigma^{\bullet}(e^{0}\otimes_{\Omega^{\bullet}(B)} (e^{+}\wedge e^{-})) & = (e^{+}\wedge e^{-})\otimes_{\Omega^{\bullet}(B)} e^{0}
\end{split}
\end{equation*}
and, on the other hand, for elements of the form $\Omega^{2}(A)\otimes_{\Omega^{\bullet}(B)} \Omega^{1}(A)$, we have 
\begin{equation}
	\begin{aligned}
		&\sigma^{\bullet}((e^{+}\wedge e^{-})\otimes_{\Omega^{\bullet}(B)} e^{\pm})=0, \\
		&\sigma^{\bullet}( (e^{+}\wedge e^{-})\otimes_{\Omega^{\bullet}(B)} e^{0})=e^{0}\otimes_{\Omega^{\bullet}(B)} (e^{+}\wedge e^{-}), \\ 
		& \sigma^{\bullet}((e^{+}\wedge e^{0})\otimes_{\Omega^{\bullet}(B)} e^{+}) =0, \\ 
		& \sigma^{\bullet}((e^{+}\wedge e^{0})\otimes_{\Omega^{\bullet}(B)} e^{0}) = q^{-4}(e^{0}\otimes_{\Omega^{\bullet}(B)} (e^{+}\wedge e^{0})),\\
		&\sigma^{\bullet}((e^{+}\wedge e^{0})\otimes_{\Omega^{\bullet}(B)} e^{-}) =q^{-2}e^{-}\otimes_{\Omega^{\bullet}(B)} (e^{+}\wedge e^{0}) - q^{-2}(q^{-4}-1)(e^{-}\wedge e^{0})\otimes_{\Omega^{\bullet}(B)} e^{+}, \\
		&\sigma^{\bullet}((e^{-}\wedge e^{0})\otimes_{\Omega^{\bullet}(B)} e^{-}) =0,\\
		&\sigma^{\bullet}((e^{-}\wedge e^{0})\otimes_{\Omega^{\bullet}(B)} e^{0})= q^{4}e^{0}\otimes_{\Omega^{\bullet}(B)} (e^{-}\wedge e^{0}),\\
		& \sigma^{\bullet}((e^{-}\wedge e^{0})\otimes_{\Omega^{\bullet}(B)} e^{+}) = q^{2}e^{+}\otimes_{\Omega^{\bullet}(B)} (e^{-}\wedge e^{0}) - q^{2}(q^{4}-1)(e^{+}\wedge e^{0})\otimes_{\Omega^{\bullet}(B)} e^{-}.
	\end{aligned}
\end{equation}
In particular, the braiding is symmetric on the level of algebra but \textit{not} symmetric on differential forms.
												
\subsection{Crossed product calculi}\label{ex:CP}

In the following we discuss crossed product algebras as in Lemma \ref{lem:crossedalg}. A differential calculus structure on them was introduced in \cite{SciWeb}, generalizing the construction for smash product algebras as pioneered in \cite{PflSch}. We briefly review this construction and show that crossed product calculi are in particular complete calculi. In this section $\sigma$ will denote the $2$-cocycle of the crossed product algebra and \textit{not} the braiding.

Until the end of this section let us fix a $\sigma$-twisted $H$-module algebra $B$ with measure $\cdot\colon H\otimes B\rightarrow B$ and $2$-cocycle $\sigma\colon H\otimes H \rightarrow B$. We further fix a FODC $\Omega^{1}(B)$ on $B$ and a bicovariant FODC $\Omega^1(H)$ on $H$ with maximal prolongation $\Omega^\bullet(H)$.

\begin{definition}\label{cross product calculus}
We say that $\Omega^{1}(B)$ is a $\sigma$-twisted $H$-module FODC if there is a map $\cdot\colon H\otimes \Omega^{1}(B)\rightarrow\Omega^{1}(B)$ such that 
\begin{equation}
\begin{aligned} 
	h\cdot (b\mathrm{d}_{B}b')& = (h_{1}\cdot b)(h_{2}\cdot \mathrm{d}_{B}b'), \\  \mathrm{d}_{B}(h\cdot b)& =h\cdot \mathrm{d}_{B}b, \\ \mathrm{d}_{B}\circ \sigma & = 0,
\end{aligned}
\end{equation} 
for every $h\in H$ and $b,b'\in B$. 
\end{definition}
Recall from Lemma \ref{lem:crossedalg} that the crossed product algebra $B\#_\sigma H$ is the tensor product vector space endowed with the multiplication $(b\#_\sigma h)(b'\#_\sigma h')
=b(h_1\cdot b')\sigma(h_2\otimes h'_1)\#_\sigma h_3h'_2$. Moreover, $B\#_\sigma H$ is a right $H$-comodule algebra with respect to $\Delta_{B\#_{\sigma}H}:B\#_{\sigma}H\rightarrow B\#_{\sigma}H\otimes H$, $b\otimes h \mapsto b\otimes h_{1}\otimes h_{2}$.
The $1$-forms are modeled on the vector space
\begin{equation}
	\Omega^{1}(B\#_{\sigma}H):=(\Omega^{1}(B)\otimes H)\oplus(B\otimes \Omega^{1}(H))
\end{equation}
with $B\#_{\sigma}H$-module actions defined by
\begin{equation}\label{twistedmodule}
\begin{split}
	(b\#_\sigma h)\cdot(\omega^B\otimes h'+b'\otimes\omega^H)
	&=b(h_1\cdot\omega^B)\sigma(h_2\otimes h'_1)\otimes h_3h'_2
	+b(h_1\cdot b')\sigma(h_2\otimes \omega^H_{-1})\otimes h_3\omega^H_0,\\
	(\omega^B\otimes h+b\otimes\omega^H)\cdot(b'\#_\sigma h')
	&=\omega^B(h_1\cdot b')\sigma(h_2\otimes h'_1)\otimes h_3h'_2
	+b(\omega^H_{-2}\cdot b')\sigma(\omega^H_{-1}\otimes h'_1)\otimes \omega^H_0h'_2
\end{split}
\end{equation}
for all $h,h'\in H$, $b,b'\in B$, $\omega^B\in\Omega^1(B)$ and $\omega^H\in\Omega^1(H)$.
\begin{theorem}[{\cite[Theorem 3.7]{SciWeb}}]\label{th:th 3.7 SciWeb} Given a $\sigma$-twisted $H$-module FODC $\Omega^{1}(B)$ on $B$, we obtain a right $H$-covariant FODC $\Omega^{1}(B\#_{\sigma}H)$ on $B\#_{\sigma}H$ with module actions \eqref{twistedmodule}, differential
\begin{equation}
	\mathrm{d}_{\#_{\sigma}}\colon B\#_{\sigma}H\rightarrow \Omega^{1}(B\#_{\sigma}H),\qquad b\otimes h \mapsto \mathrm{d}_{B}b\otimes h + b\otimes \mathrm{d}_{H}h
\end{equation}
and right $H$-coaction
\begin{equation}
	\Delta_{\Omega^{1}(B\#_{\sigma}H)}\colon\Omega^{1}(B\#_{\sigma}H)\to\Omega^{1}(B\#_{\sigma}H)\otimes H,\qquad
	\omega^B\otimes h+b\otimes\omega^H\mapsto
	\omega^B\otimes\Delta(h)+b\otimes\Delta_{\Omega^1(H)}(\omega^H).
\end{equation}
Moreover, $\Delta_{B\#_\sigma H}\colon B\#_\sigma H\to B\#_\sigma H\otimes H$ is differentiable.
\end{theorem}
In the language of this article the latter statement means that $\Omega^{1}(B\#_{\sigma}H)$ is first order complete. To account for an appropriate extension to higher order forms one considers the following definition.												
\begin{definition}
Let $(\Omega^{\bullet}(B),\mathrm{d}_{B})$ be a differential calculus on $B$. We say that it is a $\sigma$-twisted $H$-module differential calculus if there exist linear maps $\cdot\colon H\otimes \Omega^{k}(B)\rightarrow \Omega^{k}(B)$, for all $k\geq 1$, such that
\begin{equation}\label{highertwisted}
\begin{aligned} 
	h\cdot (b^{0}\mathrm{d}_{B}b^{1}\wedge \cdots \wedge \mathrm{d}_{B}b^{k}) & = (h_{1}\cdot b^{0})(h_{2}\cdot \mathrm{d}_{B}b^{1})\wedge\cdots \wedge(h_{k+1}\cdot \mathrm{d}_{B}b^{k}), \\ h \cdot \mathrm{d}_{B}b & = \mathrm{d}_{B}(h\cdot b), \\ \mathrm{d}_{B}\circ \sigma & = 0,
\end{aligned}
\end{equation} 
for all $b,b^{0},\cdots, b^{k}\in B$ and $h\in H$.  
\end{definition}
Clearly, the maps $\cdot\colon H\otimes \Omega^{k}(B)\rightarrow \Omega^{k}(B)$ are uniquely determined by \eqref{highertwisted} if they exist.
\begin{theorem}[{\cite[Theorem 3.16]{SciWeb}}]\label{DC on CPA are complete}
Let $(\Omega^{\bullet}(B),\mathrm{d}_{B})$ be a $\sigma$-twisted $H$-module differential calculus on $B$. Let us define
\begin{equation}
	\Omega^{n}(B\#_{\sigma}H):= \bigoplus_{i=0}^{n}\Omega^{n-i}(B)\otimes \Omega^{i}(H),
\end{equation}
for all $n\geq 0$, and let
\begin{equation}
\begin{aligned}
	(\omega\otimes \eta)\wedge(\omega'\otimes \eta')& := (-1)^{jk}(\omega \wedge(\eta_{-2}\cdot \omega')\sigma(\eta_{-1}\otimes \eta'_{-1}))\otimes (\eta_{0}\wedge \eta'_{0}), \\ \mathrm{d}_{\#_{\sigma}}(\omega\otimes \eta) & := \mathrm{d}_{B}\omega \otimes \eta + (-1)^{i}\omega\otimes \mathrm{d}_{H}\eta,
\end{aligned}
\end{equation} 
for $\omega\in \Omega^{i}(B)$, $\eta \in \Omega^{j}(H)$ and $\omega'\in \Omega^{k}(B)$, $\eta'\in \Omega^{\ell}(H).$ Then $\Omega^{\bullet}(B\#_{\sigma}H)$ is a right $H$-covariant differential calculus on $B\#_{\sigma}H$ with respect to which the right $H$-coaction $\Delta_{B\#_{\sigma}H}$ is differentiable, i.e. $\Omega^{\bullet}(B\#_{\sigma}H)$ is a complete calculus.
\end{theorem}
In \cite[Theorem 3.16]{SciWeb} there is the additional assumption that $\Delta\colon H\to H\otimes H$ is differentiable, which is automatic in our setting since $\Omega^1(H)$ is bicovariant and $\Omega^\bullet(H)$ is the maximal prolongation. With this construction we thus obtain a whole class of complete calculi. We show that the corresponding base forms coincide with the forms on $B$.
\begin{proposition} 
Let $\Omega^{\bullet}(B\#_{\sigma}H)$ be the complete crossed product calculus obtained from the previous theorem. The corresponding base forms coincide with the DC $\Omega^\bullet(B)$ on $B$. Furthermore, the vertical forms are given by $\mathrm{ver}^\bullet=B\otimes\Omega^\bullet(H)$, while the horizontal forms are $\mathrm{hor}^\bullet=\Omega^\bullet(B)\otimes H$.
						
\begin{proof}
Clearly all elements of $\Omega^\bullet(B)\cong\Omega^\bullet(B)\otimes\Bbbk 1_H\subseteq\Omega^\bullet(B)\otimes H$ are basic, since they are invariant under $\Delta^\bullet_{B\#_\sigma H}$. We show that all basic forms are of this form.
Let  $\beta\otimes \gamma\in \Omega^{\bullet}(B\#_{\sigma}H)$.
If $\beta\otimes \gamma$ is a base form, i.e. $\beta\otimes \gamma\otimes 1=\Delta_{B\#_{\sigma}H}^{\bullet}(\beta\otimes \gamma)=\beta\otimes \gamma_{[1]}\otimes \gamma_{[2]}$, we can apply $\mathrm{id}\otimes\varepsilon^\bullet\otimes\mathrm{id}$ to the previous equation to obtain $\beta\otimes\varepsilon^\bullet(\gamma)1=\beta\otimes\gamma$. This implies that $\beta\otimes\gamma\in\Omega^\bullet(B)\otimes\Bbbk 1_H\cong\Omega^\bullet(B)$, proving the first claim. For vertical forms we obtain $\mathrm{ver}^\bullet=(B\#_\sigma H)\otimes\Lambda^\bullet\cong B\otimes\Omega^\bullet(H)$ by Proposition \ref{prop:gradedYD} $ii.)$. For the last statement let $\beta\otimes\gamma\in\mathrm{hor}^\bullet$, i.e. $\Delta_{B\#_\sigma H}^\bullet(\beta\otimes\gamma)=\beta\otimes\gamma_{[1]}\otimes\gamma_{[2]}\in\Omega^\bullet(B\#_\sigma H)\otimes H$. Applying $\mathrm{id}\otimes\varepsilon^\bullet\otimes\mathrm{id}$ gives $\beta\otimes\gamma\in\Omega^\bullet(B)\otimes H$, since $\varepsilon^\bullet$ vanishes on forms of degree $>0$. Clearly $\Omega^\bullet(B)\otimes H$ are horizontal, proving the equality.
\end{proof}  
\end{proposition}
The crossed product calculus also admits a strong connection, as spelled out in \cite[Theorem 3.19]{SciWeb}.

Using that the cleaving map $j\colon H\to B\#_\sigma H$ and its inverse $j^{-1}\colon H\to B\#_\sigma H$ read $j(h)=1\otimes h$ and $j^{-1}(h)=\sigma^{-1}(S(h_2)\otimes h_3)\otimes S(h_1)$ we can spell out the translation map
\begin{equation}
\begin{split}
	\tau\colon H&\to (B\#_\sigma H)\otimes_B(B\#_\sigma H)\\
	\tau(h)&
	=j^{-1}(h_1)\otimes_Bj(h_2)
	=(\sigma^{-1}(S(h_2)\otimes h_3)\otimes S(h_1))\otimes_B(1\otimes h_4).
\end{split}
\end{equation}
Consequently, the \DJ ur\dj evi\'c braiding $\sigma^D\colon(B\#_\sigma H)\otimes_B(B\#_\sigma H)\to(B\#_\sigma H)\otimes_B(B\#_\sigma H)$ (here written with an index $D$ to distinguish it from the $2$-cocycle $\sigma\colon H\otimes H\to B$) reads
\begin{equation}
\begin{split}
	\sigma^D((b&\otimes h)\otimes_B(b'\otimes h'))\\
	&=(b(h_1\cdot b')\sigma(h_2\otimes h'_1)(h_3h'_2\cdot\sigma^{-1}(S(h_8)\otimes h_9))\sigma(h_4h'_3\otimes S(h_7))\otimes h_5h'_4S(h_6))\otimes_B(1\otimes h_{10})
\end{split}
\end{equation}
which simplifies to 
$\sigma^D((b\otimes h)\otimes_B(b'\otimes h'))
=(b(h_1\cdot b')\otimes h_2h'S(h_3))\otimes_B(1\otimes h_4)$
in case of the smashed product $B\# H$, i.e. for trivial $\sigma$.
Similarly one works out the expressions of the extended \DJ ur\dj evi\'c braiding for $\Omega^\bullet(B\#_\sigma H)$.

\appendix

\section{The extension of the translation map}\label{app:taubullet}

This appendix is devoted to the proof of Proposition \ref{lem:taubullet}.

Let us first show that the "centraliy property" $b\tau(h)=\tau(h)b$ for all  $b\in B$ and $h\in H$, see \eqref{tau7}, extends to base forms of arbitrary degree.
\begin{lemma}
	The restriction and corestriction
	\begin{equation}\label{ChiBulRes}
		\chi^\bullet|_{A\Omega^\bullet(B)\otimes_{\Omega^\bullet(B)}\Omega^\bullet(B)A}\colon A\Omega^\bullet(B)\otimes_{\Omega^\bullet(B)}\Omega^\bullet(B)A\to A\Omega^\bullet(B)A\otimes H
	\end{equation}
	is an injection. In particular,
	\begin{equation}
		\omega\tau(h)=\tau(h)\omega
	\end{equation}
	for all $\omega\in\Omega^\bullet(B)$ and $h\in H$.
\end{lemma}
\begin{proof}
	Clearly $M:=\chi^\bullet\bigg(A\Omega^\bullet(B)\otimes_{\Omega^\bullet(B)}\Omega^\bullet(B)A\bigg)\subseteq A\Omega^\bullet(B)A\otimes H$. We show that
	$$
	f\colon M\subseteq A\Omega^\bullet(B)A\otimes H\to A\Omega^\bullet(B)A\otimes_{\Omega^\bullet(B)}\Omega^\bullet(B)A,\qquad
	a\omega a'\otimes h\mapsto a\omega a'h^{\langle 1\rangle}\otimes_{\Omega^\bullet(B)}h^{\langle 2\rangle}
	$$
	is a left inverse of \eqref{ChiBulRes}. Let $a,a'\in A$ and $\omega,\omega'\in\Omega^\bullet(B)$. Then
	\begin{align*}
		f(\chi^\bullet(a\omega\otimes_{\Omega^\bullet(B)}\omega'a'))
		&=f(a\omega\wedge\omega'a'_0\otimes a'_1)
		=a\omega\wedge\omega'a'_0(a'_1)^{\langle 1\rangle}\otimes_{\Omega^\bullet(B)}(a'_1)^{\langle 2\rangle}
		\overset{\eqref{tau5}}{=}a\omega\wedge\omega'\otimes_{\Omega^\bullet(B)}a'\\
		&=a\omega\otimes_{\Omega^\bullet(B)}\omega'a'
	\end{align*}
	follows. In particular, the codomain of $f$ is $A\Omega^\bullet(B)\otimes_{\Omega^\bullet(B)}\Omega^\bullet(B)A$. The existence of a left inverse is equivalent to injectivity of \eqref{ChiBulRes}.
	
	Now let $\omega\in\Omega^\bullet(B)$ and $h\in H$. Then
	$\chi^\bullet(\omega\tau(h))=\omega h^{\langle 1\rangle}(h^{\langle 2\rangle})_0\otimes_B(h^{\langle 2\rangle})_1\overset{\eqref{tau6}}{=}\omega\otimes_Bh$ coincides with $\chi^\bullet(\tau(h)\omega)=h^{\langle 1\rangle}(h^{\langle 2\rangle})_0\omega\otimes_B(h^{\langle 2\rangle})_1\overset{\eqref{tau6}}{=}\omega\otimes_Bh$, which implies $\omega\tau(h)=\tau(h)\omega$ by the injectivity of \eqref{ChiBulRes}.
\end{proof}

Next, we give an extension of the translation map $\tau\colon H\to A\otimes_BA$ to $1$-forms. Note that we do not have to assume $\Omega^1_\mathrm{hor}=\mathrm{hor}^1$.

\begin{lemma}\label{lem:tau1}
	The map
	\begin{equation}
		\begin{split}
			\tau^1\colon\Omega^1(H)&\to\Omega^\bullet(A)\otimes_{\Omega^\bullet(B)}\Omega^\bullet(A)\\
			h\mathrm{d}g&\mapsto\mathrm{d}(g^{\langle 1\rangle})h^{\langle 1\rangle}\otimes_{\Omega^\bullet(B)}h^{\langle 2\rangle}g^{\langle 2\rangle}
			+g^{\langle 1\rangle}h^{\langle 1\rangle}\otimes_{\Omega^\bullet(B)}h^{\langle 2\rangle}\mathrm{d}(g^{\langle 2\rangle})
		\end{split}
	\end{equation}
	is well-defined.
\end{lemma}
\begin{proof}
	Using the fundamental theorem of Hopf modules $\Omega^1(H)\cong H\otimes\Lambda^1\cong H\otimes H^+/I$ we show that if $0=h\varpi(g)=hS(g_1)\mathrm{d}(g_2)$ with $h\otimes g\in H\otimes I$, then $\tau^1(hS(g_1)\mathrm{d}(g_2))=0$. In fact,
	\allowdisplaybreaks
	\begin{align*}
		\tau^1(hS(g_1)\mathrm{d}(g_2))
		&=\mathrm{d}((g_2)^{\langle 1\rangle})(hS(g_1))^{\langle 1\rangle}\otimes_{\Omega^\bullet(B)}(hS(g_1))^{\langle 2\rangle}(g_2)^{\langle 2\rangle}
		+(g_2)^{\langle 1\rangle}(hS(g_1))^{\langle 1\rangle}\otimes_{\Omega^\bullet(B)}(hS(g_1))^{\langle 2\rangle}\mathrm{d}((g_2)^{\langle 2\rangle})\\
		&\overset{\eqref{tau2}}{=}\mathrm{d}((g_2)^{\langle 1\rangle})S(g_1)^{\langle 1\rangle}h^{\langle 1\rangle}\otimes_{\Omega^\bullet(B)}h^{\langle 2\rangle}S(g_1)^{\langle 2\rangle}(g_2)^{\langle 2\rangle}\\
		&\qquad+(g_2)^{\langle 1\rangle}S(g_1)^{\langle 1\rangle}h^{\langle 1\rangle}\otimes_{\Omega^\bullet(B)}h^{\langle 2\rangle}S(g_1)^{\langle 2\rangle}\mathrm{d}((g_2)^{\langle 2\rangle})\\
		&\overset{\eqref{tau4}}{=}\mathrm{d}((g_2)^{\langle 1\rangle})S(g_1)^{\langle 1\rangle}h^{\langle 1\rangle}\otimes_{\Omega^\bullet(B)}h^{\langle 2\rangle}S(g_1)^{\langle 2\rangle}(g_2)^{\langle 2\rangle}\\
		&\qquad+(g^{\langle 1\rangle})_0((g^{\langle 1\rangle})_1)^{\langle 1\rangle}h^{\langle 1\rangle}\otimes_{\Omega^\bullet(B)}h^{\langle 2\rangle}((g^{\langle 1\rangle})_1)^{\langle 2\rangle}\mathrm{d}(g^{\langle 2\rangle})\\
		&\overset{\eqref{tau5}}{=}\mathrm{d}((g_2)^{\langle 1\rangle})S(g_1)^{\langle 1\rangle}h^{\langle 1\rangle}\otimes_{\Omega^\bullet(B)}h^{\langle 2\rangle}S(g_1)^{\langle 2\rangle}(g_2)^{\langle 2\rangle}
		+h^{\langle 1\rangle}\otimes_{\Omega^\bullet(B)}h^{\langle 2\rangle}g^{\langle 1\rangle}\mathrm{d}(g^{\langle 2\rangle})\\
		&\overset{(*)}{=}h^{\langle 1\rangle}\otimes_{\Omega^\bullet(B)}h^{\langle 2\rangle}\mathrm{d}((g_2)^{\langle 1\rangle})S(g_1)^{\langle 1\rangle}S(g_1)^{\langle 2\rangle}(g_2)^{\langle 2\rangle}
		+h^{\langle 1\rangle}\otimes_{\Omega^\bullet(B)}h^{\langle 2\rangle}g^{\langle 1\rangle}\mathrm{d}(g^{\langle 2\rangle})\\
		&\overset{\eqref{tau1}}{=}h^{\langle 1\rangle}\otimes_{\Omega^\bullet(B)}h^{\langle 2\rangle}\mathrm{d}(g^{\langle 1\rangle})g^{\langle 2\rangle}
		+h^{\langle 1\rangle}\otimes_{\Omega^\bullet(B)}h^{\langle 2\rangle}g^{\langle 1\rangle}\mathrm{d}(g^{\langle 2\rangle})\\
		&=h^{\langle 1\rangle}\otimes_{\Omega^\bullet(B)}h^{\langle 2\rangle}\mathrm{d}(g^{\langle 1\rangle}g^{\langle 2\rangle})\\
		&\overset{\eqref{tau1}}{=}0,
	\end{align*}
	where in $(*)$ we applied
	$$
	\mathrm{d}((g_2)^{\langle 1\rangle})S(g_1)^{\langle 1\rangle}\otimes_BS(g_1)^{\langle 2\rangle}(g_2)^{\langle 2\rangle}\in\Omega^1(B)\otimes_BA.
	$$
	The first tensor factor above is a basic form, since it is right $H$-coinvariant and horizontal (see Proposition \ref{prop:baseintersection}). This is the case, since
	\begin{align*}
		\mathrm{d}(((g_2)^{\langle 1\rangle})_0)&(S(g_1)^{\langle 1\rangle})_0\otimes_BS(g_1)^{\langle 2\rangle}(g_2)^{\langle 2\rangle}\otimes ((g_2)^{\langle 1\rangle})_1(S(g_1)^{\langle 1\rangle})_1\\
		&\overset{\eqref{tau4}}{=}\mathrm{d}((g_3)^{\langle 1\rangle})(S(g_1)^{\langle 1\rangle})_0\otimes_BS(g_1)^{\langle 2\rangle}(g_3)^{\langle 2\rangle}\otimes S(g_2)(S(g_1)^{\langle 1\rangle})_1\\
		&\overset{\eqref{tau4}}{=}\mathrm{d}((g_3)^{\langle 1\rangle})(S(g_1)_2)^{\langle 1\rangle}\otimes_B(S(g_1)_2)^{\langle 2\rangle}(g_3)^{\langle 2\rangle}\otimes S(g_2)S(S(g_1)_1)\\
		&=\mathrm{d}((g_4)^{\langle 1\rangle})S(g_1)^{\langle 1\rangle}\otimes_BS(g_1)^{\langle 2\rangle}(g_4)^{\langle 2\rangle}\otimes S(g_3)S(S(g_2))\\
		&=\mathrm{d}((g_2)^{\langle 1\rangle})(g_1)^{\langle 1\rangle}\otimes_BS(g_1)^{\langle 2\rangle}(g_2)^{\langle 2\rangle}\otimes 1
	\end{align*}
	and
	\begin{align*}
		\pi_v\bigg(\mathrm{d}((g_2)^{\langle 1\rangle})&S(g_1)^{\langle 1\rangle}\bigg)\otimes_BS(g_1)^{\langle 2\rangle}(g_2)^{\langle 2\rangle}\\
		&=((g_2)^{\langle 1\rangle})_0(S(g_1)^{\langle 1\rangle})_0
		\otimes\mathrm{d}(((g_2)^{\langle 1\rangle})_1)(S(g_1)^{\langle 1\rangle})_1
		\otimes_BS(g_1)^{\langle 2\rangle}(g_2)^{\langle 2\rangle}\\
		&\overset{\eqref{tau4}}{=}(g_3)^{\langle 1\rangle}S(g_1)^{\langle 1\rangle}
		\otimes\mathrm{d}(S(g_2)_1)S(S(g_2)_2)
		\otimes_BS(g_1)^{\langle 2\rangle}(g_3)^{\langle 2\rangle}\\
		&=-(g_3)^{\langle 1\rangle}S(g_1)^{\langle 1\rangle}
		\otimes\underbrace{S(g_2)_1\mathrm{d}(S(S(g_2)_2))}_{=\varpi(\pi_\varepsilon(S(g_2)))=0}
		\otimes_BS(g_1)^{\langle 2\rangle}(g_3)^{\langle 2\rangle}\\
		&=0,
	\end{align*}
	using \eqref{tau4} twice, as in the computation before.
	The term above vanishes, since $\mathrm{Ad}(I)\subseteq I\otimes H$, $S(I)\subseteq I$ and $I=\ker\varpi$.
\end{proof}

\begin{proof}[Proof of Proposition \ref{lem:taubullet}:]
	Recall from Proposition \ref{prop:MaxGammaH}, as the maximal prolongation of a bicovariant FODC, $\Omega^\bullet(H)$ is isomorphic to $H\otimes\Lambda^\bullet$ and $\Lambda^\bullet$ is freely generated by $\Lambda^1$, modulo the relation
	\begin{equation}\label{maxrelation}
		\wedge\circ(\varpi\circ\pi_\varepsilon\otimes\varpi\circ\pi_\varepsilon)(\Delta(I))=0.
	\end{equation}
	Thus, we can determine $\tau^2$ on elements $h\varpi(g)\varpi(k)\in\Omega^2(H)$ with $h\in H$, $g,k\in H^+$.
	Applying $\tau^2$ to $h\varpi(g)\varpi(k)=hS(g_1)\mathrm{d}(g_2)\wedge S(k_1)\mathrm{d}(k_2)=-hS(g_1)\mathrm{d}(g_2S(k_1))\wedge\mathrm{d}(k_2)$ we obtain the expression 
	\begin{align*}
		\tau^2(-hS(g_1)\mathrm{d}(g_2S(k_1))\wedge\mathrm{d}(k_2))
		&=\mathrm{d}((k_2)^{\langle 1\rangle})\wedge\mathrm{d}((g_2S(k_1))^{\langle 1\rangle})(hS(g_1))^{\langle 1\rangle}\otimes_{\Omega^\bullet(B)}(hS(g_1))^{\langle 2\rangle}(g_2S(k_1))^{\langle 2\rangle}(k_2)^{\langle 2\rangle}\\
		&\quad+\mathrm{d}((k_2)^{\langle 1\rangle})(g_2S(k_1))^{\langle 1\rangle}(hS(g_1))^{\langle 1\rangle}\otimes_{\Omega^\bullet(B)}(hS(g_1))^{\langle 2\rangle}\mathrm{d}((g_2S(k_1))^{\langle 2\rangle})(k_2)^{\langle 2\rangle}\\
		&\quad-(k_2)^{\langle 1\rangle}\mathrm{d}((g_2S(k_1))^{\langle 1\rangle})(hS(g_1))^{\langle 1\rangle}\otimes_{\Omega^\bullet(B)}(hS(g_1))^{\langle 2\rangle}(g_2S(k_1))^{\langle 2\rangle}\mathrm{d}((k_2)^{\langle 2\rangle})\\
		&\quad-(k_2)^{\langle 1\rangle}(g_2S(k_1))^{\langle 1\rangle}(hS(g_1))^{\langle 1\rangle}\otimes_{\Omega^\bullet(B)}(hS(g_1))^{\langle 2\rangle}\mathrm{d}((g_2S(k_1))^{\langle 2\rangle})\wedge\mathrm{d}((k_2)^{\langle 2\rangle})
	\end{align*}
	and using the methods in the proof of Lemma \ref{lem:tau1} it is easy to show that the above vanishes if $g\in I$ or $k\in I$. It remains to show the vanishing under relation \eqref{maxrelation}.
	
	Let $h\in I\subseteq H^+$ and consider the result of $\tau^2$ applied to 
	$$
	0
	=\varpi(\pi_\varepsilon(h_1))\wedge\varpi(\pi_\varepsilon(h_2))
	=S(h_1)\mathrm{d}(h_2)\wedge S(h_3)\mathrm{d}(h_4)
	=-\mathrm{d}(S(h_1))\wedge\mathrm{d}(h_2)
	$$
	as in definition \eqref{tau2form}. We show that $\tau^2(\mathrm{d}(S(h_1))\wedge\mathrm{d}(h_2))$ vanishes, which then implies that \eqref{tau2form} is well-defined. In fact,
	\allowdisplaybreaks
	\begin{align*}
		\tau^2(\mathrm{d}(S(h_1))\wedge\mathrm{d}(h_2))
		&=-\mathrm{d}((h_2)^{\langle 1\rangle})\wedge\mathrm{d}(S(h_1)^{\langle 1\rangle})\otimes_{\Omega^\bullet(B)}S(h_1)^{\langle 2\rangle}(h_2)^{\langle 2\rangle}\\
		&\quad-\mathrm{d}((h_2)^{\langle 1\rangle})S(h_1)^{\langle 1\rangle}\otimes_{\Omega^\bullet(B)}\mathrm{d}(S(h_1)^{\langle 2\rangle})(h_2)^{\langle 2\rangle}\\
		&\quad+(h_2)^{\langle 1\rangle}\mathrm{d}(S(h_1)^{\langle 1\rangle})\otimes_{\Omega^\bullet(B)}S(h_1)^{\langle 2\rangle}\mathrm{d}((h_2)^{\langle 2\rangle})\\
		&\quad+(h_2)^{\langle 1\rangle}S(h_1)^{\langle 1\rangle}\otimes_{\Omega^\bullet(B)}\mathrm{d}(S(h_1)^{\langle 2\rangle})\wedge\mathrm{d}((h_2)^{\langle 2\rangle})\\
		&\overset{(*)}{=}-1\otimes_{\Omega^\bullet(B)}\mathrm{d}((h_2)^{\langle 1\rangle})\wedge\mathrm{d}(S(h_1)^{\langle 1\rangle})S(h_1)^{\langle 2\rangle}(h_2)^{\langle 2\rangle}\\
		&\quad-1\otimes_{\Omega^\bullet(B)}\mathrm{d}((h_2)^{\langle 1\rangle})S(h_1)^{\langle 1\rangle}\mathrm{d}(S(h_1)^{\langle 2\rangle})(h_2)^{\langle 2\rangle}\\
		&\quad+1\otimes_{\Omega^\bullet(B)}(h_2)^{\langle 1\rangle}\mathrm{d}(S(h_1)^{\langle 1\rangle})S(h_1)^{\langle 2\rangle}\mathrm{d}((h_2)^{\langle 2\rangle})\\
		&\quad+1\otimes_{\Omega^\bullet(B)}\mathrm{d}(h^{\langle 1\rangle})\wedge\mathrm{d}(h^{\langle 2\rangle})\\
		&=-1\otimes_{\Omega^\bullet(B)}\mathrm{d}((h_2)^{\langle 1\rangle})\wedge\mathrm{d}(S(h_1)^{\langle 1\rangle}S(h_1)^{\langle 2\rangle}(h_2)^{\langle 2\rangle})\\
		&\quad+1\otimes_{\Omega^\bullet(B)}\mathrm{d}((h_2)^{\langle 1\rangle})\wedge S(h_1)^{\langle 1\rangle}\mathrm{d}(S(h_1)^{\langle 2\rangle}(h_2)^{\langle 2\rangle})\\
		&\quad-1\otimes_{\Omega^\bullet(B)}\mathrm{d}((h_2)^{\langle 1\rangle})S(h_1)^{\langle 1\rangle}\mathrm{d}(S(h_1)^{\langle 2\rangle})(h_2)^{\langle 2\rangle}\\
		&\quad-1\otimes_{\Omega^\bullet(B)}\mathrm{d}((h_2)^{\langle 1\rangle})S(h_1)^{\langle 1\rangle}S(h_1)^{\langle 2\rangle}\mathrm{d}((h_2)^{\langle 2\rangle})\\
		&\quad+1\otimes_{\Omega^\bullet(B)}\mathrm{d}(h^{\langle 1\rangle})\wedge\mathrm{d}(h^{\langle 2\rangle})\\
		&=0,
	\end{align*}
	where in $(*)$ we used that
	\begin{align*}
		\omega:=(h_2)^{\langle 1\rangle}\mathrm{d}(S(h_1)^{\langle 1\rangle})\otimes_BS(h_1)^{\langle 2\rangle}\otimes (h_2)^{\langle 2\rangle}\in\Omega^1(B)\otimes_BA\otimes A
	\end{align*}
	as in the proof of Lemma \ref{lem:tau1},
	which implies that the differential of the first tensor product of the former
	$$
	\mathrm{d}((h_2)^{\langle 1\rangle})\wedge\mathrm{d}(S(h_1)^{\langle 1\rangle})\otimes_BS(h_1)^{\langle 2\rangle}\otimes (h_2)^{\langle 2\rangle}\in\Omega^2(B)\otimes_BA\otimes A,
	$$
	is an element of $\Omega^2(B)\otimes_BA\otimes A$, since $\Omega^\bullet(B)$ is a DGA.
	
	The same argumentation can be applied to higher-order forms $h\varpi(g)\ldots\varpi(\pi_\varepsilon(k_1))\wedge\varpi(\pi_\varepsilon(k_1))\ldots$ with $k\in I$.
	Since \eqref{maxrelation} is the only relation imposed on the maximal prolongation, it follows that \eqref{taubullet} is well-defined in all degrees.
\end{proof}

\subsection*{Acknowledgements}

We thank Paolo Aschieri, Branimir Ćaćić, Alessandro Carotenuto, Rita Fioresi and Chiara Pagani for helpful discussions and comments.
This research was supported by Gnsaga-Indam, by
COST Action CaLISTA CA21109, HORIZON-MSCA-2022-SE-01-01 CaLIGOLA, MSCA-DN CaLiForNIA - 101119552 and  INFN Sezione Bologna. This research was partially supported by the University of Warsaw Thematic Research Programme "Quantum Symmetries". The third author would like to thank Piotr M. Hajac for pointing out the work of \DJ ur\dj evi\'c which kick-started this project.


\begin{thebibliography}{}
\bibitem{AsBiPaSc}
\textsc{Aschieri, P., Bieliavsky, P., Pagani, C. and Schenkel, A.:}
\emph{Noncommutative Principal Bundles Through Twist Deformation}.
Comm. Math. Phys. \textbf{352} (2017) 287–344.

\bibitem{AFL} 
\textsc{Aschieri, P., Fioresi, R. and Latini, E.:}
\emph{Quantum Principal Bundles on Projective Bases. }
Commun. Math. Phys. 382 (2021) 1691--1724.
	
\bibitem{AFLW} 
\textsc{Aschieri, P., Fioresi, R., Latini, E. and Weber, T.:}
\emph{Differential Calculi on Quantum Principal Bundles over Projective Bases.}
Commun. Math. Phys. \textbf{405} 136 (2024).

\bibitem{ALP1} 
\textsc{Aschieri, P., Landi, G., and Pagani, C.:}
\emph{Braided Hopf algebras and gauge transformations.}
Preprint \hyperlink{https://arxiv.org/abs/2203.13811}{arXiv:2203.13811} .

\bibitem{ALP2} 
\textsc{Aschieri, P., Landi, G., and Pagani, C.:}
\emph{Braided Hopf algebras and gauge transformations II: $*$-Structures and Examples.}
Math Phys Anal Geom \textbf{26} 13 (2023).

\bibitem{BegMaj}
\textsc{{Beggs}, E.J. and {Majid}, S.:}
\textit{Quantum Riemannian geometry}.
Grundlehren der mathematischen Wissenschaften [Fundamental Principles of Mathematical Sciences], \textbf{355}. Springer, Cham, 2020.

\bibitem{Brz96}
\textsc{{Brzezi\'nski}, T.:}
\textit{Translation map in quantum principal bundles}.
J. Geom. Phys. 20 \textbf{4} (1996) 349–370.

\bibitem{BrJaMa}
\textsc{Brzezi\'nski, T., Janelidze, G. and Maszczyk, T.:}
\emph{Galois structures}, in P.M. Hajac (ed.), {\it Lecture Notes on Noncommutative Geometry and Quantum	Groups}. http://www.mimuw.edu.pl/$\sim$pwit/toknotes/toknotes

\bibitem{BrzMaj}
\textsc{{Brzezi\'nski}, T. and {Majid}, S.:}
\textit{Quantum group gauge theory on quantum spaces}. 
Comm. Math. Phys. \textbf{157} (1993) 591-638.

\bibitem{Cac}
\textsc{Ćaćić, B.:}
\textit{Classical gauge theory on quantum principal bundles}.
Preprint
\hyperlink{https://arxiv.org/abs/2108.13789}{arXiv:2108.13789}.

\bibitem{Cac2}
\textsc{Ćaćić, B.:}
\textit{Geometric foundations for classical $U(1)$-gauge theory on noncommutative manifolds}.
Commun. Math. Phys. \textbf{405} 209, (2024).

\bibitem{Connes}
\textsc{Connes, A.:}
\textit{Noncommutative Geometry}. 
Academic Press, 1995.

\bibitem{DGH}
\textsc{D\c{a}browski, L., Grosse, H. and Hajac, P.M.:}
\emph{Strong Connections and Chern-Connes Pairing in the Hopf–Galois Theory}.
Commun. Math. Phys. \textbf{220} (2001) 301–331.

\bibitem{AntonioThesis}
\textsc{Del Donno, A.:}
\emph{Differential calculi on quantum principal bundles in the \DJ ur\dj evi\'c approach.}
Master Thesis, \hyperlink{https://arxiv.org/abs/2406.16882}{arXiv:2406.16882}.

\bibitem{DoiTak}
\textsc{Doi, Y. and Takeuchi, M.:}
\emph{Cleft comodule algebras for a bialgebra}.
Commun. Algebra \textbf{14} 5 (1986) 801-817.

%\bibitem{DurTK}
%\textsc{Durdevic, M.:}
%\emph{Quantum principal bundles and Tannaka-Krein duality theory}.
%Rep. Math. Phys. \textbf{38} 3 (1996) 313–324.

\bibitem{DurI}
\textsc{\DJ ur\dj evi\'c, M.:}
\emph{Geometry of quantum principal bundles I}.
Commun. Math. Phys. \textbf{175} 3 (1996) 457–520.

\bibitem{DurII}
\textsc{\DJ ur\dj evi\'c, M.:}
\emph{Geometry of Quantum Principal Bundles II - Extended Version}.
Rev. Math. Phys. \textbf{9} 5 (1997) 531-607.

\bibitem{DurDS}
\textsc{\DJ ur\dj evi\'c, M.:}
\emph{Differential structures on quantum principal bundles}.
Rep. Math. Phys. \textbf{41} 1 (1998) 91–115.

\bibitem{DurGauge}
\textsc{\DJ ur\dj evi\'c, M.:}
\emph{Quantum Gauge Transformations and Braided Structure on Quantum Principal Bundles}.
Preprint \hyperlink{https://arxiv.org/abs/q-alg/9605010v1}{arXiv:q-alg/9605010} .

\bibitem{DurHG}
\textsc{\DJ ur\dj evi\'c, M.:}
\emph{Quantum Principal Bundles as Hopf-Galois Extensions}.
Preprint \hyperlink{https://arxiv.org/abs/q-alg/9507022}{arXiv:q-alg/9507022} .


%\bibitem{DurGT}
%\textsc{\DJ ur\dj evi\'c, M.:}
%\emph{Quantum Gauge Transformations and Braided Structure on Quantum Principal Bundles}.
%Preprint arXiv:q-alg/9605010.

\bibitem{FLP24}
\textsc{Fioresi, R., Latini, E. and Pagani, C.:}
\emph{Reduction of Quantum Principal Bundles over non affine bases}.
Preprint arXiv:2403.06830 [math.QA].

\bibitem{Haj96}
\textsc{Hajac, P.M.:}
\emph{Strong connections on quantum principal bundles}.
Commun. Math. Phys. \textbf{182} (1996) 579–617.

\bibitem{Kas95}
\textsc{Kassel, C.:}  
\emph{Quantum groups}. 
Graduate Texts in Mathematics, \textbf{155}. Springer-Verlag, New York, 1995.

\bibitem{KhaLanVS}
\textsc{Khalkhali, M., Landi, G. and van Suijlekom, W.:}
\emph{Holomorphic Structures on the Quantum Projective Line}.
Int. Math. Res. Not. \textbf{4} (2011) 851–884.

\bibitem{KrTa}
\textsc{Kreimer, H.F. and {Takeuchi}, M.:}  
\emph{Hopf algebras and Galois extensions of an algebra}. 
Indiana Univ. Math. J. \textbf{30} (1981) 615-692.

\bibitem{Ma99}
\textsc{Majid, S.:}  
\emph{Quantum and braided group Riemannian geometry}. 
J. Geom. Phys. \textbf{30} 2 (1999) 113-146.

\bibitem{Mon94}
\textsc{Montgomery, S.:} 
\emph{Hopf algebras and their actions on rings}. 
CBMS Regional Conference Series in Math \textbf{82}, AMS, 1994. 

\bibitem{Mon09}
\textsc{Montgomery, S.:} 
\emph{Hopf Galois theory: A survey}. 
Geometry \& Topology Monographs \textbf{16} (2009) 367–400.

\bibitem{PflSch}
\textsc{Pflaum, M. and Schauenburg, P.:}
\emph{Differential calculi on noncommutative bundles}. 
Z Phys C-Particles and Fields \textbf{76} (1997) 733–744.

\bibitem{SchDC}
\textsc{Schauenburg, P.:} 
\emph{Differential-Graded Hopf Algebras and Quantum Group Differential Calculi}. 
J. Algebra \textbf{180} (1996) 239-286.

\bibitem{SchYD}
\textsc{Schauenburg, P.:}
\emph{Hopf Modules and Yetter-Drinfel'd Modules}. 
J. Algebra \textbf{169} (1994) 874-890.

\bibitem{Schn90}
\textsc{Schneider, H-J.:}
\emph{Principal homogeneous spaces for arbitrary Hopf algebras}.
Isr. J. Math. \textbf{72} (1990) 167–195.

\bibitem{SciWeb}
\textsc{Sciandra, A. and Weber, T.:}
\emph{Noncommutative differential geometry on crossed product algebras}.
Preprint \hyperlink{https://arxiv.org/abs/2308.14662}{arXiv:2308.14662} .

\bibitem{Son15}
\textsc{Sontz, S.B.:}
\emph{Principal Bundles. The Quantum Case}.
Springer Cham, doi.org/10.1007/978-3-319-15829-7, 2015.

\bibitem{Sweedler69}
\textsc{Sweedler, M.E.:}
\emph{Hopf algebras}. Mathematics Lecture Note Series W. A. Benjamin, Inc., New York, 1969.

\bibitem{Wor89}
\textsc{{Woronowicz}, S.L.:}
\textit{Differential calculus on compact matrix pseudogroups (quantum groups)}.
Comm. Math. Phys. 122 \textbf{1} (1989) 125-170.
\end{thebibliography}
\end{document}